\documentclass[12pt]{amsart}
\usepackage{amssymb}
\usepackage{color}
\usepackage{hyperref}
\hypersetup{
    colorlinks=true,
    linkcolor=blue,
    citecolor=magenta,      
}
\usepackage{graphicx}
\usepackage{float}
\usepackage[all,cmtip]{xy}
\usepackage{tikz}
\usepackage{relsize}
\usetikzlibrary{matrix}

\newcommand{\SE}{{\mathcal{E}}}
\newcommand{\SD}{{\mathcal{D}}}
\newcommand{\SW}{{\mathcal{W}}}

\newcommand{\SF}{{\mathcal{F}}}
\newcommand{\SL}{{\mathcal{L}}}

\newcommand{\ST}{{\mathcal{T}}}
\newcommand{\SU}{{\mathcal{U}}}
\renewcommand{\SS}{{\mathcal{S}}}
\newcommand{\SH}{\mathcal{H}}

\newcommand{\LSD}{\mathcal{L}(\mathcal{D})}
\newcommand{\LLSD}{\mathcal{L}(\mathcal{L}(\mathcal{D}))}
\newcommand{\LLE}{\mathcal{L}(\mathcal{L}(\mathcal{E}))}
\newcommand{\LLW}{\mathcal{L}(\mathcal{L}(\mathcal{W}))}

\newcommand{\std}{{\operatorname{std}}}
\newcommand{\pr}{{\operatorname{pr}}}
\newcommand{\weak}{{\operatorname{weak}}}
\renewcommand{\int}{{\operatorname{int}}}

\newcommand{\R}{{\mathbb{R}}}
\newcommand{\Z}{{\mathbb{Z}}}

\newcommand{\NS}{{\mathbb{S}}}
\newcommand{\D}{{\mathbb{D}}}

\newcommand{\SO}{{\operatorname{SO}}}

\newcommand{\Op}{{\mathcal{O}p}}
\newcommand{\image}{{\operatorname{im}}}
\newcommand{\trans}{{\operatorname{trans}}}
\newcommand{\sol}{{\operatorname{sol}}}

\newcommand{\lra}{\longrightarrow}

\newcommand{\Engel}{{\operatorname{\mathfrak{Engel}}}}
\newcommand{\FEngel}{{\operatorname{\mathcal{F}\mathfrak{Engel}}}}
\newcommand{\OT}{{\operatorname{OT}}}

\newtheorem{theorem}{Theorem}[section]
\newtheorem*{theorem*}{Theorem}
\newtheorem{lemma}[theorem]{Lemma}
\newtheorem{proposition}[theorem]{Proposition}
\newtheorem{corollary}[theorem]{Corollary}
\newtheorem{definition}[theorem]{Definition}

\theoremstyle{definition}

\newtheorem{remark}[theorem]{Remark}

\newtheorem{example}[theorem]{Example}

\newtheorem*{main*}{Theorem~\ref{t:main}}

\begin{document} 

\title{The Engel-Lutz twist and overtwisted Engel structures}

\date{\today}


\author{\'Alvaro del Pino}
\address{Utrecht University, Department of Mathematics, Budapestlaan 6, 3584 Utrecht, The Netherlands}
\email{a.delpinogomez@uu.nl}

\author{Thomas Vogel}
\address{Mathematisches Institut der LMU, Theresienstr. 39, 80333 M{\"u}nchen, Germany. }
\email{tvogel@math.lmu.de}

\begin{abstract}
We introduce a modification procedure for Engel structures that is reminiscent of the Lutz twist in $3$-dimensional Contact Topology. This notion allows us to define what an Engel overtwisted disc is, and to prove a complete $h$-principle for overtwisted Engel structures with fixed overtwisted disc.
\end{abstract}

\maketitle

\section{About this paper}

A maximally non-integrable $2$-plane field $\SD$ on a $4$-manifold $M$ is called an \emph{Engel structure}. Maximal non-integrability means that $\SE = [\SD,\SD]$ is a distribution of rank $3$ that satisfies $[\SE,\SE] = TM$. Engel structures hold a privileged position in the taxonomy of distributions: they are one of the four \emph{topologically stable} families of distributions (i.e. distributions described by an open condition and having a local model). In this paper we study them from the perspective of the $h$-principle, viewing them as sections of the Grassmann bundle $\textrm{Gr}(TM,2)$ which satisfy a certain differential relation of order $2$. The goal is to study the space of such sections, comparing it with the space of formal solutions. In the case when $M$ and $\SD$ are orientable and oriented, formal solutions are trivializations of $TM$ in which the first two components of the framing span $\SD$.

Gromov's method of flexible continuous sheaves \cite{gr} shows that Engel structures, which are given by an open and Diff-invariant partial differential relation, satisfy the $h$-principle in open manifolds. This does not prove the analogous result for closed manifolds. In \cite[Intrigue F2]{em} it is stated: ``\emph{On the other hand, it is unknown whether the h-principle holds for Engel structures on closed 4-manifolds. In particular, it is an outstanding open question whether any closed parallelizable 4-manifold admits an Engel structure}''.

It was proven in \cite{vo} that every parallelizable $4$-manifold does admit an Engel structure. The proof relies on the interplay between Engel and contact structures: First, the ambient manifold is decomposed into round handles, then one can proceed handle by handle constructing the desired Engel structure. During this process, the boundary of each handlebody inherits a contact structure. The heart of the argument is to be able to manipulate these contact structures to ensure that the handles can indeed be glued. 

More recently, the existence problem for Engel structures (in every formal class) was solved \cite{cppp}. It was shown that the inclusion 
$$ i: \Engel(M) \to \FEngel(M) $$
of the Engel structures into the formal Engel structures is a surjection in homotopy groups whenever $M$ is closed. The key contribution of \cite{cppp} is a method for constructing and manipulating Engel structures locally. As such, the approach differs from the one in \cite{vo} and is more of an $h$-principle. However, the proof of $h$-principle in \cite{cppp} does {\em not} work relative to subsets $U$ of the manifold. This is because it relies on adding twisting (see Sections~\ref{sec:definitions} and~\ref{sec:development map}) along the leaves of the characteristic foliation. This is not possible when the characteristic foliation 
is tangent to the boundary of $\partial U$ (as sketched in Figure~\ref{fig:non-rel-cp3}).
\begin{figure}[htb]
\begin{center}
\includegraphics[scale=0.8]{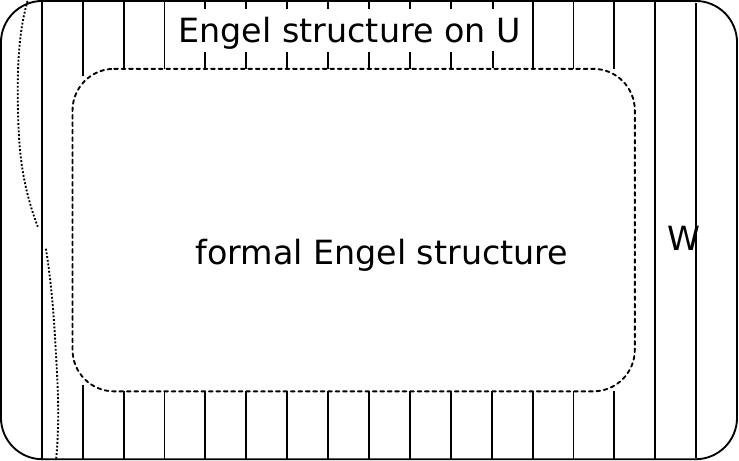}
\caption{Using the methods of \cite{cppp}, one cannot extend every Engel structure from $U=\Op(\partial \D^4)$ to the interior of $M=\D^4$, in general. The characteristic foliation is vertical.  \label{fig:non-rel-cp3}}
\end{center}
\end{figure}

Currently, the outstanding open question in Engel Topology is, whether $i$ might in fact be a homotopy equivalence. While answering this question is beyond the scope of this article, we are able to define what an \emph{overtwisted} Engel structure is and show that Engel structures with fixed overtwisted disc satisfy the complete $h$-principle. In particular, the result we obtain is relative in the parameter, and the domain. An important corollary, which does not follow from \cite{vo} nor \cite{cppp}, is that every Engel germ in $\partial\D^4$ extends to the interior of $\D^4$ if it extends formally.

\subsection{Outline of the paper}

Throughout the paper we use Gromov's notation $\Op(A)$ to denote an arbitrarily small neighborhood of the set $A$. We write $\int(A)$ for the interior of the set $A$.

\subsubsection*{The Engel-Lutz twist and overtwisted Engel structures}

In $3$-dimensional Contact Topology one defines the Lutz twist \cite{lu} as a surgery operation along a transverse knot. One replaces the contact structure on a standard neighborhood of the knot by a contact structure which ``twists'' more. While this operation does not change the homotopy class as a plane field, it can change the homotopy class as a contact structure. In particular, the resulting contact structure is always overtwisted \cite{El89}.

The main focus of this article are Engel-Lutz twists. This construction is the Engel analogue of the classical Lutz twist. One replaces the Engel structure in a neighborhood of a $2$-torus transverse to $\SD$ by another Engel structure which has more ``twisting''. This operation preserves the formal type, but we do not know whether the resulting Engel structure is Engel homotopic to the original one. This is explained in Section~\ref{sec:Lutz}. 

Before defining the Engel-Lutz twist, we show that there are many transverse $2$-tori in a given Engel manifold (Section~\ref{sec:surfaces}). More precisely, we prove that transverse $2$-tori can be constructed along knots transverse to the even-contact structure $[\SD,\SD]$. Moreover, the transverse surfaces we construct are quite flexible and can be isotoped effectively.

When one attempts to construct an Engel structure from a formal Engel structure, a certain family of Engel germs along $\partial\D^4$ is obtained (Proposition \ref{prop:reduction}). A key observation is that the Engel-Lutz twist can be used to extend these germs to the interior (Subsection \ref{ssec:nonParametric}). Motivated by this fact, in Section~\ref{sec:otDisc} we endow the $4$-ball with a specific Engel structure $\SD_\OT$, which is a portion of Engel-Lutz twist satisfying a certain numerical constraint. We say that the Engel manifold $\Delta_\OT = (\D^4,\SD_\OT)$ is an \emph{overtwisted disc}. If an Engel embedding $\Delta: \Delta_\OT \lra (M,\SD)$ exists, we say that $(M,\SD)$ is an \emph{overtwisted Engel manifold}.

\subsubsection*{Statement of the main results}

Let $\Delta: \D^4 \lra M$ be a smooth embedding. We define $\Engel_\OT(M,\Delta) \subset \Engel(M)$ to be the subspace of those Engel structures on $M$ such that their pullback by $\Delta$ is $\SD_\OT$. Its formal analogue is $\FEngel(M,\Delta) \subset \FEngel(M)$, the subspace of those formal Engel structures which are Engel on $\image(\Delta)$ and that pullback to $\SD_\OT$ under $\Delta$. 

In families, the particular embedding of the overtwisted disc may vary and might even be twisted: Let $K$ be a compact manifold. We will say that a $K$-family of Engel structures $\SD: K \lra \Engel(M)$ is \emph{overtwisted} if there is a locally trivial fibration of overtwisted discs $\Delta_k \subset (M,\SD(k))$, $k \in K$. We will say that the family $\Delta = (\Delta_k)_{k \in K}$ is the \emph{certificate of overtwistedness} of $\SD$. In Subsection \ref{ssec:foliated} we discuss a different (but equivalent up to homotopy) way in which overtwistedness can appear in the parametric setting.

The terminology ``\emph{overtwisted}'' is justified by our main result, which states that overtwisted families of Engel structures are flexible.
\begin{theorem} \label{thm:main}
Let $M$ be a smooth $4$-manifold (possibly non-compact or with boundary). Let $U \subset M$ be a closed subset such that $M \setminus U$ is connected. Let $K' \subset K$ be compact CW-complexes.

Suppose $\SD_0: K \lra \FEngel(M)$ is a family of formal Engel structures satisfying: 
\begin{itemize}
\item $\SD_0(k)$ is Engel at $p \in M$ if $p \in \Op(U)$ or $k \in \Op(K')$,
\item $\SD_0$ has a certificate of overtwistedness $(\Delta_k)_{k \in K}$ with $\Delta_k \subset M \setminus U$.
\end{itemize}

Then, there is homotopy $\SD: K \times [0,1] \lra \FEngel(M)$ such that
\begin{itemize}
\item $\SD(\cdot,0) = \SD_0$,
\item $\SD(k,s)(p) = \SD(k,0)(p)$ whenever $k \in \Op(K')$ or $p \in \Op(U) \cup \Delta_k$,
\item $\SD(\cdot,1): K \lra \Engel(M)$. 
\end{itemize}
\end{theorem}
Its proof spans Section~\ref{sec:hPrinciple}. We invite the reader to check that our approach provides an alternate proof of Eliashberg's classification of overtwisted contact structures in dimension $3$ \cite{El89}.

From the theorem, the complete $h$-principle for Engel structures with fixed overtwisted disc follows immediately (Subsection \ref{ssec:corMain}):
\begin{corollary} \label{cor:main}
The inclusion 
$$ \Engel_\OT(M,\Delta) \lra \FEngel(M,\Delta) $$
is a homotopy equivalence.
\end{corollary}

As is often the case, the particular trivialization of the overtwisted disc does not matter for $\pi_0$ statements. Write $\Engel_\OT(M) \subset \Engel(M)$ for the subspace of overtwisted Engel structures. The following statement follows from Theorem~\ref{thm:main}:
\begin{corollary} \label{cor:pi0hPrinciple}
The inclusion
$$ \Engel_\OT(M) \lra \FEngel(M) $$
induces a bijection between path components.
\end{corollary}

More generally, the homotopy type of the certificate is not relevant for lower dimensional families:
\begin{corollary} \label{cor:pi3hPrinciple}
Let $K$ be a CW-complex of dimension at most 3. Two overtwisted families $\SD_0,\SD_1: K \lra \Engel(M)$ are homotopic if and only if they are formally homotopic.
\end{corollary}
Corollaries \ref{cor:pi0hPrinciple} and \ref{cor:pi3hPrinciple} are worked out in Subsection \ref{pi3hPrincipleproof}.

Theorem~\ref{thm:main} addresses one of the shortcomings in \cite{cppp}, namely that the result proved there is not relative in the domain (see Figure~\ref{fig:non-rel-cp3}).
\begin{corollary} \label{cor:extension}
Let $K$ be a compact CW-complex and let $\SD_0: K \lra \FEngel(\D^4)$ be a family of formal Engel structures in the $4$-ball with $\SD_0(k)|_{\partial \D^4}$ Engel. Then, there is a family $\SD_1: K \lra \Engel(\D^4)$ which, relative to $\partial\D^4$, is formally homotopic to $\SD_0$.
\end{corollary}
In the non-parametric case, this is shown in Subsection \ref{ssec:nonParametric} without relying on the full force of Theorem~\ref{thm:main}. The parametric version is deduced from Theorem \ref{thm:main} in Subsection \ref{ssec:extensionProof}.

Finally, in Subsection \ref{ssec:foliated} we briefly explain how an $h$-principle for foliations endowed with a leafwise Engel structure can be deduced from our main result.

\subsubsection*{Looseness and overtwistedness}

In \cite{dp2} (see also the upcoming paper \cite{cpp}) an alternative characterization of Engel flexibility is introduced and, using it, a result similar to Corollary~\ref{cor:main} is obtained for a different subclass of Engel structures, which we called \emph{loose} in \cite{cpp}. Normally, any two different notions of overtwistedness are equivalent up to homotopy. Indeed, suppose we have two different definitions, the first given by the presence of some local model, which we call the disc of type I, and the second one given by another local model, which we call the disc of type II. Suppose $\SD$ is an overtwisted structure of type I. We modify it away from its overtwisted disc to introduce a disc of type II, and then we use the parametric nature of Theorem~\ref{thm:main} to construct a homotopy between them through overtwisted structures of type I. This shows that an overtwisted structure of type I contains a disc of type II up to homotopy.

This argument cannot be applied here because Engel looseness is {\em not} given by a \emph{local} model but by a \emph{global} property of the structure. Then, an intriguing open question is how Engel overtwistedness and looseness might be related. A closely related question is whether there exist Engel families that are not overtwisted nor loose.

A relevant remark is that the extension problem from Engel structures (Corollary \ref{cor:extension}) cannot be deduced either from the results in \cite{cpp}. This is due to the fact that looseness is a global property and therefore not well-suited for performing constructions relative in the domain.

{\bf Acknowledgments:} The work in this article has profited from an AIM-Workshop on Engel structures which was held in the spring of 2017. We would like to thank all the participants of this workshop for their questions and all the discussions we had. In particular, we would like to thank Roger Casals, Yakov Eliashberg, Dieter Kotschick, Emmy Murphy, and Fran Presas. \'A.~del Pino is supported by the NWO Vici Grant no. 639.033.312 of Marius Crainic.

\section{Definitions and standard results}

We start by reviewing the relevant definitions and going over basic results. Unless stated otherwise, all distributions we consider are smooth. 

\subsection{Engel structures} \label{sec:definitions}

\begin{definition}
Let $M$ be a smooth $4$-manifold.
\begin{itemize}
	\item An \textbf{even-contact structure} is a smooth hyperplane field $\SE$ such that $[\SE,\SE]=TM$.
	\item An \textbf{Engel structure} is a smooth plane field $\SD$ such that $[\SD,\SD]$ is an even-contact structure.	
\end{itemize}
\end{definition}
When $U\subset M$ is closed, we say that a plane field $\SD$ is Engel on $U$ if it is an Engel structure on some open neighborhood of $U$.

For an even-contact structure $\SE$ and any $p \in M$, we can consider the map
\begin{align*}
\SE(p)\times\SE(p) & \lra T_p M/\SE(p) \\
X,Y & \longmapsto [X,Y](p).
\end{align*}
The commutator is defined using local extensions of the vectors $X,Y$ to local sections of $\SE$, but the result is independent of choices. Thus, we have an antisymmetric bilinear form on $\SE(p)$ taking values in a real vector space of dimension one. Therefore, this form has a kernel, which we denote by $\SW(p)$. It defines a foliation of rank one $\SW \subset \SE$ which we will call the \textbf{kernel foliation}\footnote{Other common names for $\SW$ include {\em characteristic foliation} and {\em isotropic foliation}.} of $\SE$. It is characterized by the property $[\SW,\SE]\subset \SE$, i.e. all flows tangent to $\SW$ preserve $\SE$.

If $\SD$ is an Engel structure satisfying $\SE = [\SD,\SD]$, it follows that $\SW\subset\SD$, because otherwise $[\SD,\SD]\not\subset\SE$. Consequently, from every Engel structure $\SD$ we obtain a flag of distributions 
\begin{equation} \label{e:flag}
\SW\subset\SD\subset\SE\subset TM
\end{equation}
where the rank increases by one at every step. Whenever $\SE=[\SD,\SD]$ is induced by an Engel structure, $\SE$ has a canonical orientation defined by $\{X,Y,[X,Y]\}$, where $\{X,Y\}$ is any local frame of $\SD$. However, there are no canonical orientations for the other elements of the flag.

\begin{definition}
A \textbf{formal  Engel structure} is a complete (non-oriented) flag $\SW \subset \SD \subset \SE \subset TM$, together with two  isomorphisms
\begin{align} 
\begin{split}\label{e:canonicalIsos}
 \det(\SD) &\cong \SE/\SW, \\
 \det(\SE/\SW)& \cong TM/\SE.
\end{split}
\end{align}
\end{definition}
The space of formal Engel structures on $M$, endowed with the $C^0$-topology, is denoted by $\FEngel(M)$. Under the natural inclusion
$$ \iota : \Engel(M) \lra \FEngel(M) $$
it induces the $C^2$-topology on the space of Engel structures $\Engel(M)$. 

If $\SD$ is an Engel structure, the first isomorphism in Equation (\ref{e:canonicalIsos}) is induced by 
\begin{align}
\begin{split} \label{e:first iso}
\Lambda^2 \SD & \lra \SE/\SD \\
X\wedge X' & \longmapsto [X,X'].
\end{split}
\end{align}
The induced map being an isomorphism is a reformulation of the fact that $\SE$ inherits a canonical orientation from $\SD$. The second isomorphism is induced by the map
\begin{align}
\begin{split} \label{e:second iso}
\Lambda^2 \SE & \lra TM/\SE \\
Y\wedge Y' & \longmapsto [Y,Y']
\end{split}
\end{align}
whose kernel is $\SW \wedge \SE$. We will often write $(\SW,\SD,\SE)$ when we talk about a formal Engel structure, omitting the morphisms \eqref{e:canonicalIsos}.

Lastly, consider the case when $\SW$ and $\SD$ are oriented and fix a framing $\{W \in \SW,X \in \SD\}$ of $\SD$ compatible with the orientations. Then, the following expression yields a framing of $TM$:
\begin{align*}
\{ W,X,[W,X],[X,[W,X]] \}
\end{align*}
This framing is unique, up to homotopy, once the orientations of $\SW$ and $\SD$ have been fixed. We call it the \textbf{Engel framing}. Under these orientation assumptions the Engel condition can be rephrased in terms of a pair of forms $\alpha,\beta$ defining $\SD=\ker(\alpha)\cap\ker(\beta)$ and $\ker(\alpha)=\SE$ as follows
\begin{align*}
\alpha\wedge d\alpha & \neq 0,  \\
\alpha\wedge \beta \wedge d\alpha&=0, \\
\alpha\wedge \beta \wedge d\beta &\neq 0.
\end{align*}
In this oriented case, we can consider oriented flags as the formal counterparts of Engel structures.

\subsection{The development map} \label{sec:development map}

Following \cite{montgomery} we now recall an interpretation of the condition $[\SD,\SD]=\SE$ in terms of the holonomy of $\SW$.

Fix a point $p \in M$ and write $\SW_p$ for the leaf of $\SW$ through $p$. Let $\widetilde\SW_p$ be its universal cover. Given a point $q \in \widetilde\SW_p$, we can find a vector field $W$ which is tangent to $\SW$ and is non-vanishing on the segment of $\SW_p$ connecting $p$ and the image of $q$ in $\SW_p$ (we abuse notation and still denote it by $q$). The flow $\varphi_t$ of $W$ preserves $\SW$ and $\SE$. Then, the differential of $\varphi_{-t}$ induces a linear map
$$ h_p(q):=D\varphi_{-t}  : \SE(q)/\SW(q) \longrightarrow \SE(p)/\SW(p) $$
where $q=\varphi_t(p)$. This map is independent of the choice of $W$ and it represents the linearized holonomy of the foliation $\SW$. Therefore, one can consider the image of $\SD(q)$ for each $q$ in $\widetilde\SW_p$. This defines the {\bf development map}
\begin{align}
\begin{split} \label{e:development map}
h_p : \widetilde\SW_p & \longrightarrow \mathbb{P}(\SE(p)/\SW(p)) \\
q & \longmapsto D\varphi_{-t}(\SD(q)) \textrm{ with } \varphi_t(p)=q \textrm{ as above.}
\end{split}
\end{align}
Because of the Engel condition $[\SW,\SD]=\SE$, we deduce that $h_q$ is an immersion. 

When $\SD$ is oriented, one can replace the projective space $\mathbb{P}(\SE(p)/\SW(p))$ by the space $\widetilde{\mathbb{P}}(\SE(p)/\SW(p))$ of oriented lines in $\SE(p)/\SW(p)$.

\subsection{Curves on spheres and the Engel condition}

Fix coordinates $(p,t)$ in $\D^3\times\R$. A plane field $\SD$ in $\D^3\times\R$ containing $\partial_t$ can be described in terms of a smooth family of curves $(H_p : \R\lra \NS^2)_{p\in \D^3}$. Indeed, once we fix a framing $\{X(p),Y(p),Z(p)\}$ of $T\D^3$, we can identify $(H_p)_{p \in \D^3}$ with the vector field
$$
(p,t) \longmapsto \left(H_p(t)\right)_1\cdot X+\left(H_p(t)\right)_2\cdot Y  +  \left(H_p(t)\right)_3\cdot Z,
$$
where the subscript $i$ corresponds to taking the $i$-th coordinate. Using this identification, we construct a plane field:
\[ \SD(p,t) = \langle \partial_t, H_p(t) \rangle. \]
One can now compute
\begin{align*}
[\SD,\SD](p,t) & = \langle \partial_t, H_p(t), \dot{H}_p(t) \rangle \\
\intertext{where $\dot{H}_p(t) = [\partial_t,H_p(t)]$. Similarly,}
\big[\SD,[\SD,\SD]\big](p,t) & = \big\langle \partial_t, H_p(t), \dot{H}_p(t), \ddot{H}_p(t), [H_p(t),\dot{H}_p(t)] \big\rangle 
\end{align*}
where $\ddot{H}_p(t) = [\partial_t,\dot{H}_p(t)]$.

The following characterization of Engel structures tangent to $\partial_t$ will be fundamental. Its proof follows from the discussion above. For additional details we refer the reader to \cite{cppp}.
\begin{proposition} \label{prop:EngelCharacterisation}
A $\D^3$-family of curves $(H_p)_{p \in \D^3}$ in $\NS^2$ determines an Engel structure near $(p,t)\in \D^3\times\R$ if and only $H_p$ is an immersion at time $t$ and at least one of the following conditions holds.
\begin{itemize}
\item[(i)] The vectors $H_p(t),\dot{H}_p(t),\ddot{H}_p(t)$ are linearly independent. 
\item[(ii)] The vector fields $H_q(t)$ and $\dot{H}_q(t)$ span a contact structure in a neighborhood of $(p,t)$ in $\D^3\times\{t\}$.  
\end{itemize}
\end{proposition}
We remark that $H_p$ being an immersion is equivalent to the fact that $\SD$ is non-integrable (i.e. $[\SD,\SD]\not\subset\SD$). Note that conditions (i) and (ii) are both open and not mutually exclusive. Moreover, we do not require for the contact structure in (ii) to be positive with respect to a fixed orientation of $\D^3$.  

By definition, the kernel foliation is tangent to $\partial_t$ if and only if condition (i) fails. If condition (i) is satisfied, then $H_p$ is strictly convex or strictly concave at $t$.

\subsection{Model structures}  \label{ssec:modelStructures}

In Section~\ref{sec:hPrinciple} we will perform certain local manipulations of formal Engel structures. We will now describe a family of local models that is suitable for this purpose. The reader should compare the construction presented here to condition (ii) in Proposition~\ref{prop:EngelCharacterisation}.

Let $I$ be a $1$-manifold. We write $(y,x)$ for the product coordinates in $\Gamma = I \times [0,1]$. A pair of smooth functions $f_\pm: \Gamma \lra \R$ satisfying $f_- \leq 0 < f_+$ determines a domain 
$$ D(I,f_-,f_+) = \{ (y,x,z) \in \Gamma \times \R \quad|\quad f_-(y,x)\leq z \leq f_+(y,x)\} $$
which we endow with the contact structure 
$$ \xi = \ker(\cos(z)dy-\sin(z)dx) = \langle \partial_z, \cos(z)\partial_x+\sin(z)\partial_y \rangle. $$

Let $J=[a,b]$ be a $1$-manifold with coordinate $w$. A function $c: D(I,f_-,f_+)\times J \lra \R$ determines a complete flag in $D(I,f_-,f_+)\times J$: 
\begin{align} \label{e:Engel on shell}
\begin{split}
\SW & = \langle\partial_w\rangle \\
\SD_c & =\SW\oplus\langle X_c= \cos(c)\partial_z + \sin(c)(\cos(z)\partial_x+\sin(z)\partial_y)\rangle \\
\SE & = \SW \oplus \xi,
\end{split}
\end{align}
where the distribution $\SE$ is even-contact, does not depend on $c$, and has $\SW$ as its characteristic foliation.

Using the flag we can describe a formal Engel structure by fixing bundle isomorphisms as in Equation~\eqref{e:canonicalIsos}. First note that, since $\SE$ is even-contact, the isomorphism $\det(\SE/\SW) \cong TM/\SE$ is automatically given by Equation \eqref{e:second iso}. The other bundle isomorphism $\det(\SD) \cong \SE/\SD$ is unique, up to homotopy, once we impose for $\{\partial_w,\partial_z,\cos(z)\partial_x+\sin(z)\partial_y\}$ to be a positive framing of $\SE$. It is given by an identification
\begin{align*}
\Lambda^2\SD & \lra  \SE/\SD \\
\partial_t \wedge X_c & \longmapsto b\cdot\left(-\sin(c)\partial_z + \cos(c)\left(\cos(z)\partial_x+\sin(z)\partial_y\right)\right),
\end{align*}
where $b$ is a positive function. Often, $\SD_c$ will be Engel on some subset $U$ of the model. Then, the isomorphism is prescribed $U$ by Equation \eqref{e:first iso}  on, and we will assume that $b|_U = (\partial_t c)|_U$.

\begin{definition} \label{def:model Engel structure}
The formal Engel manifold 
$$
M(I,J,f_-,f_+,c) = \big(\left(D(I,f_-,f_+)\times J\right),\SW,\SD_c,\SE\big)
$$
is said to be a \textbf{model structure} with
\begin{center}
{\def\arraystretch{1.3}\tabcolsep=10pt
\begin{tabular}{|c|c|}
\hline
angular function & $c$ \\
height function & $c(y,x,z,b) - c(y,x,z,a)$ \\
height & $\min_{(y,x,z) \in \partial D(I,f_-,f_+)} (c(y,x,z,b) - c(y,x,z,a))$ \\
bottom boundary & $\{w=a\}$ \\
top boundary & $\{w=b\}$ \\
vertical boundary & $ \left\{(y,x,z,w)\,\big|\, (y,x,z) \in \partial D(I,f_-,f_+)\right\}$\\
projection & $\pi: D(I,f_-,f_+) \times J \lra D(I,f_-,f_+)$ \\
\hline
\end{tabular} }
\end{center}
The model is {\bf solid} if $\partial_w c>0$ everywhere. 
\end{definition}
It is immediate that parametric families of models can be considered by smoothly varying the functions $f_-$, $f_+$, $c$, and the parametrisation of $I$ and $J$.

Condition (ii) of Proposition~\ref{prop:EngelCharacterisation} implies that a model structure is Engel at $(y,x,z,w)$ if and only if $\partial_w c(y,x,z,w) \neq 0$. Having chosen $\{\partial_w,\partial_z,\cos(z)\partial_x+\sin(z)\partial_y\}$ as a positive framing of $\SE$, we need to further require $\partial_w c(y,x,z,w) > 0$ to obtain the correct formal class.
\begin{lemma} \label{lem:Bolzano}
Let $M(I,J,f_-,f_+,c)$ be a model that is Engel along the boundary and has strictly positive height function everywhere. Then $\SD_c$ can be homotoped to a solid model through model structures and relative to the boundary.
\end{lemma}
\begin{proof}
Since $c(y,x,z,b) > c(y,x,z,a)$, one can construct a function 
$$ c^\sol: D(I,f_-,f_+)\times J \lra \R $$
which coincides with $c$ on a neighborhood of the boundary and satisfies $\partial_w c^\sol>0$. Then $M(I,J,f_-f_+,tc^{sol}+(1-t)c)$, $t\in[0,1]$, provides the desired homotopy through model structures.
\end{proof}

\subsection{Loops transverse to the even-contact structure} \label{ssec:transverseKnots}

Let $(M,\SD)$ be an Engel manifold. In Sections~\ref{sec:surfaces} and~\ref{sec:Lutz}, we will construct surfaces and hypersurfaces transverse to $\SD$. These submanifolds will lie in standard neighborhoods of curves transverse to $\SE=[\SD,\SD]$. In this subsection we will explain some elementary facts about such curves.

\begin{definition}
 A {\bf transverse curve} is an embedding of a $1$-manifold into $(M,\SD)$ which is transverse to $\SE=[\SD,\SD]$.
\end{definition}
When a transverse curve is closed, we will sometimes call it a \emph{transverse loop}. 

\subsubsection{Normal form of $\SD$ close to a transverse loop} \label{sssec:modelKnot}

Let $I$ be a compact $1$-manifold. We consider $I\times \R^3$ with coordinates $(y,x,z,w)$ and endowed with the Engel structure 
$$ \SD_\trans = \ker(\alpha_\trans=dy-zdx,\beta_\trans=dx-wdz). $$
The following proposition states that any Engel structure $\SD$ is isomorphic to this model in the vicinity of a transverse curve.
\begin{proposition} \label{prop:modelKnot}
Let $(M, \SD)$ be an Engel manifold, $\SE$ the associated even-contact structure, and $\gamma: I \lra M$ a transverse curve. Then, the map $\gamma$ can be extended to an Engel diffeomorphism
$$ \varphi : \Op(I\times\{0\}) \subset (I \times \R^3 , \SD_\trans) \lra  \Op(\image(\gamma)) \subset (M,\SD). $$
The space of such diffeomorphisms is weakly contractible. 
\end{proposition}

\begin{proof}
Since $\gamma$ is transverse to $\SE$, there is a hypersurface $N\simeq I\times \R^2$ containing $\gamma$ which is transverse to $\SW$. Then $TN\cap\SE$ is a contact structure $\xi$ on $N$, $\gamma$ is transverse to it, and the Engel structure induces a Legendrian line field $\SH = TN\cap \SD \subset TN\cap\SE$ on $N$.

We choose a vector field $H$ spanning $\SH$ and a surface $\Gamma = I \times (-\varepsilon,\varepsilon) \subset N$ which is transverse to $H$ and contains $\gamma$. Using the characteristic foliation $\xi(\Gamma)$ one obtains coordinates $(x,y)$ on a neighborhood of $\gamma$ within $\Gamma$ such that $\gamma=\{x=0\}$ and $\xi(\Gamma)$ is defined by $dy$. Then we use the flow of $H$ and the contact condition to obtain coordinates $(x,y,z)$ in $N$ such that $\xi=\ker(dy-zdx)$. By construction, $\SH$ is spanned by $\partial_z$.

Finally, we choose a vector field $W$ spanning $\SW$ near $N$. Due to the condition $[\SD,\SD]=\SE$ there are coordinates $(x,y,z,w)$ on a neighborhood of $\gamma$ in which the even-contact structure $\SE$ is defined by $\ker(dy-zdx)$ and $\SD=\SE\cap\ker(dx-wdz)$.  

If $\varphi$ is a diffeomorphism with the desired property, then there are submanifolds $\Gamma\subset N$ and a vector field $H$ which produce $\varphi$ as above: we simply pick $N=\varphi(\{w=0\})$, $\Gamma=\varphi(\{z,w=0\})$, and $H=\varphi_*(\partial_z)$. Now assume $K$ is a compact space and $\{\varphi_k\}_{k\in K}$ is a family of diffeomorphisms with the desired properties. Then the germs of $\Gamma_k \subset N_k$ and $H_k$ along $\gamma$ as above form a weakly contractible space. Thus, the same is true for the space of germs of diffeomorphisms satisfying the hypotheses of the proposition.  
\end{proof}

\subsubsection{Existence and classification of transverse curves}

The following lemma proves the $h$-principle in $\pi_0$ for \emph{embedded} transverse curves. It is an elementary consequence of the corresponding $h$-principle for \emph{immersed} transverse curves in contact $3$-manifolds.
\begin{lemma} \label{lem:hPrincipleTransverseCurves} 
Let $\gamma: [0,1] \lra M$ be a smooth map that, in a neighborhood of the endpoints, is embedded and positively transverse to $\SE$. Then it can be homotoped, relative to the boundary, to an embedded transverse arc. 

Let $\gamma_0, \gamma_1 : [0,1] \lra M$ be embedded arcs, transverse to $\SE$, that agree in a neighborhood of their endpoints, and which lie in the same homotopy class as maps. Then they are isotopic, relative to the boundary, as embedded transverse arcs.  
\end{lemma} 
\begin{proof}
Let us prove the second statement, which is slightly more involved. The proof of the first is analogous.

Since $\gamma_0$ and $\gamma_1$ lie in the same homotopy class, we can choose a smooth homotopy $(\gamma_s)_{s\in[0,1]}$ between them. Since $\gamma_0$ and $\gamma_1$ are transverse to $\SE$ and agree at their endpoints, we can view $\gamma_s$ as family of formal immersions transverse to $\SE$ (that is, as a family of curves together with an injective bundle map $T[0,1]\lra TM/\SE$). These bundle maps are unique up to homotopy. Using the Smale-Hirsch theorem \cite[8.2.1]{em} we can assume that the $(\gamma_s)_{s\in[0,1]}$ are immersions, but not necessarily transverse to $\SE$. For dimensional reasons, standard transversality theorems imply that the $(\gamma_s)_{s\in[0,1]}$ can be assumed to never be tangent to $\SW$. Similarly, after a perturbation we may assume that $\gamma_s$ is embedded for all $s$. 

We can pick a family of embeddings $\psi_s : [0,L]\times\D^2 \longrightarrow M$ which are transverse to $\SW$ and satisfy $\psi_s(y,0,0)=\gamma_s(y)$. Intersecting the tangent space of the image of $\psi_s$ with $\SE$ we obtain a contact structure $\xi_s$ on $N_s = \psi_s([0,1]\times\D^2)$. The $h$-principle for transverse immersions in contact manifolds (see \cite[14.2.2]{em}) implies that there is a smooth family of immersions 
$$ \widehat{\gamma}_s : [0,L] \lra (N_s,\xi_s) \subset M $$ 
which are transverse to $\xi_s$ (and therefore transverse to $\SE$) interpolating between $\gamma_0$ and $\gamma_1$. We may assume that $\widehat{\gamma}_s$ coincides with $\gamma_0$ and $\gamma_1$ on a neighborhood $[0,l]\cup[1-l]$ of the points $\{0,1\}$. 

In order to turn $\widehat{\gamma}_s$ into a family of transverse embeddings, we use the extra dimension complementary to $N_s$. Let $g: [0,1] \lra [0,1]$ be a smooth function such that 
\begin{itemize}
\item $g\equiv 0$ on $[0,l/2]\cup[1-l/2,1]$,
\item $g$ is strictly monotone outside of $[l,1-l]$. 
\end{itemize}
Locally in a neighborhood of $\widehat{\gamma}_s$ we can choose a vector field $W$ orienting $\SW$; denote its flow by $h_r$. Let $\lambda : [0,1]\lra [0,1]$ be a non-decreasing surjective function which is constant close to its endpoints. If $\varepsilon>0$ is small enough
\begin{align*}
\gamma_s : [0,1] & \longrightarrow M \\
y & \longmapsto h_{\varepsilon \lambda(s)g(y)}\big(\widehat{\gamma}_s(y)\big) 
\end{align*}
is a family of transverse embeddings interpolating between the curves $\gamma_0$ and $\gamma_1$. 
\end{proof}

\subsection{Curves tangent to the even-contact structure} \label{ssec:Legendrians}

One of the technical steps in Section~\ref{sec:hPrinciple} requires us to manipulate curves that are \emph{tangent} to the even-contact structure but transverse to its kernel. The following statement provides standard neighborhoods for such curves:
\begin{proposition} \label{prop:tangentCurvesModel}
Let $(M,\SE)$ be a $4$-dimensional even-contact manifold and $\nu: I \lra M$ an embedded curve tangent to $\SE$ but transverse to its kernel.

Then, there is an even-contact embedding:
\[ \big(\Op(I \times \{0\}) \subset I \times \R^3, \ker(dz-ydx)\big) \lra (\Op(\image(\nu)) \subset M, \SE), \]
where the coordinates in $I \times \R^3$ are $(y,x,z,w)$.
\end{proposition}
\begin{proof}
This follows from the standard neighborhood theorem for Legendrians in a contact $3$-manifold and the fact that any even-contact structure is semi-locally given by a contact slice stabilized by $\R$.
\end{proof}

The following proposition states the \emph{existence $h$-principle} for curves tangent to an even-contact structure.
\begin{proposition} \label{prop:tangentCurvesExistence}
Let $K$ be a contractible set and $(M,\SE_k)_{k \in K}$ a $K$-family of $4$-dimensional even-contact manifolds with corresponding kernels $(\SW_k)_{k \in K}$. Let $J \subset I$ be intervals.

Fix a family of embeddings $(\nu_k: I \lra M)_{k \in K}$ with $\nu_k|_{\Op(J)}$ tangent to $\SE_k$ and transverse to $\SW_k$. Then, there is a $C^0$-small perturbation $\tilde\nu_k$ of $\nu_k$ such that
\begin{itemize}
\item $\tilde\nu_k$ is an embedded curve everywhere tangent to $\SE_k$ but transverse to $\SW_k$, and
\item $\tilde\nu_k \equiv \nu_k$ over the interval $J$. 
\end{itemize}
\end{proposition}

For its proof, we need two standard $h$-principle results. The first one provides embeddings transverse to $\SW$:
\begin{lemma} \label{lem:technical1}
Let $K$ be a contractible set, $M$ a $4$-dimensional manifold, and $(\SW_k)_{k \in K}$ a family of line fields in $M$. Let $J \subset I$ be intervals.

Fix a family of embeddings $(\nu_k: I \lra M)_{k \in K}$ with $\nu_k|_{\Op(J)}$ transverse to $\SW_k$. Then, there is a $C^0$-small perturbation $\tilde\nu_k$ of $\nu_k$ which satisfies:
\begin{itemize}
\item $\tilde\nu_k$ is an embedded curve everywhere transverse to $\SW_k$, and
\item $\tilde\nu_k \equiv \nu_k$ over the interval $J$. 
\end{itemize}
\end{lemma}
\begin{proof}
For a fixed $k_0 \in K$, standard transversality implies that $\nu_{k_0}$ can be $C^\infty$-perturbed to yield an embedded curve transverse to $\SW_{k_0}$. The homotopy between the two, as embeddings, endows $\nu_{k_0}$ with the structure of an embedding that is \emph{formally transverse} to $\SW_{k_0}$. Then, since the parameter space $K$ is contractible, it follows that the whole family $(\nu_k)_{k \in K}$ can be regarded as a family of embeddings that are formally transverse to the corresponding line fields $(\SW_k)_{k \in K}$.

For all points $p \in M$, the subset $T_pM \setminus \SW_k(p)$ is connected and its convex hull is $T_pM$, i.e. it is an \emph{ample} subset. Then the result follows from the complete $h$-principle for ample, Diff-invariant, $1$-dimensional differential relations (see \cite[Section 10.4]{em}).
\end{proof}

The second ingredient we need is a method for producing transverse knots in contact $3$-manifolds. 
\begin{lemma} \label{lem:technical2}
Let $K$ be a contractible set and $(N_k,\xi_k)_{k \in K}$ a $K$-family of $3$-dimensional contact manifolds. Let $J \subset I$ be intervals.

Fix a family of knots $(\nu_k: I \lra N_k)_{k \in K}$ with $\nu_k|_{\Op(J)}$ tangent to $\xi_k$. Then, there is a $C^0$-small perturbation $\tilde\nu_k$ of $\nu_k$ which satisfies:
\begin{itemize}
\item $\tilde\nu_k$ is a Legendrian knot for $\xi_k$, and
\item $\tilde\nu_k \equiv \nu_k$ over the interval $J$. 
\end{itemize} 
\end{lemma}
\begin{proof}
If $(N_k,\xi_k) \simeq (\R^3,\xi_\std)$, we can use the front projection and view $\nu_k$ as a (possibly singular) curve in the plane together with a slope at every point recording the missing coordinate. Then $\tilde\nu_k$ is drawn in the front projection as a cuspidal curve such that the slope approximates the given one. In the general case, for each $t \in [0,1]$, we can find, parametrically in $k$, a Darboux ball $(\Op(\nu_k(t)),\xi_k)$. We may then use the corresponding local front projections and the relative nature of the statement (in the domain, not the parameter) to conclude.
\end{proof}

\begin{proof}[Proof of Proposition~\ref{prop:tangentCurvesExistence}]
First we apply Lemma \ref{lem:technical1} to yield a new family of embedded curves $(\nu_k')_{k \in K}$ which are transverse to the line fields $(\SW_k)_{k \in K}$. Transversality with respect to $\SW_k$ implies that each $\nu_k'$ can be thickened to an embedded transverse $3$-manifold endowed with a contact structure $(N_k,\xi_k = TN_k \cap \SE_k)$. Then we apply Lemma \ref{lem:technical2}, parametrically in $k$, to homotope $\nu_k'$ to an embedding $\tilde\nu_k$ which is still contained in $N_k$ but is additionally transverse to $\SE_k$. This concludes the proof.
\end{proof}

\section{Transverse surfaces in Engel manifolds} \label{sec:surfaces}

Let $(M,\SD)$ be an Engel manifold, $I$ a compact $1$-manifold, and $\gamma: I \lra M$ a transverse curve. The purpose of this section is to construct embedded surfaces transverse to $\SD$ and contained in a tubular neighborhood of $\gamma$. Later, in Section~\ref{sec:Lutz}, we will describe how to perform an \emph{Engel-Lutz} twist along such a surface. 

Throughout this section we will assume that $M$, $\SD$, and all surfaces appearing in the discussion are oriented. In particular, $\SD$ is trivial as a vector bundle. 

\subsection{Transverse surfaces}  \label{ssec:many trans tori}

The following notion is the central object in the discussions that follow.
\begin{definition}
Let $(M,\SD)$ be an Engel manifold. A {\bf transverse surface} in $M$ is an immersed surface $S$ whose tangent space at each point is transverse to $\SD$.
\end{definition}
It follows that the even-contact structure $\SE = [\SD,\SD]$ intersects a transverse surface in a line field.

An immersed surface $S\subset (M,\SD=\ker(\alpha)\cap\ker(\beta))$, with $\SE = \ker(\alpha)$, is transverse if and only if:
\begin{itemize}
\item $\ker(\alpha)\cap TS = \SE\cap TS$ has rank $1$ everywhere,
\item $\ker(\beta)\cap TS$ has rank $1$ everywhere, and
\item these two line fields are everywhere transverse to each other.   
\end{itemize}

\subsubsection{Profiles and transverse cylinders}

Let $\gamma : I \lra M$ be a transverse curve.  We will construct transverse surfaces in $\Op(\gamma)$ using the normal form from Proposition~\ref{prop:modelKnot}. A reference concerning  transverse knots in contact $3$-manifolds is~\cite{etnyre}. 

Our first example is an explicit  transverse cylinder/torus:
\begin{example} \label{ex:explicitParametrisationTorus}
Let $r: I \lra \R$ be a positive function. The map 
\begin{align*}
S : I \times \NS^1 \quad\longrightarrow &\quad (I \times \R^3,\SD_\trans) \\
(y,\theta)  \quad\longmapsto 	& \quad 
\left(\begin{array}{l} 
y\\ 
x(y,\theta)=r(y)^2\sin(\theta)\cos(\theta)\\ 
z(y,\theta)=r(y)\sin(\theta)\\
w(y,\theta)=2r(y)\cos(\theta)\end{array}\right)
\end{align*}
is an embedding of $I \times \NS^1$. We now determine a condition which ensures that the $1$-forms $S^*(dy-zdx)$ and $S^*(dx-wdz)$ are linearly independent everywhere, i.e.~that $S$ is transverse to $\SD_\trans$. We compute
\begin{align}
\begin{split} \label{e:T2 transv}
S^*(dy-zdx) 
            & = (1-2r(y)^2r'(y)\sin(\theta)^2\cos(\theta))dy \\
						& + r(y)^3\sin(\theta)(\sin(\theta)^2 - \cos(\theta)^2)d\theta \\
S^*(dx-wdz) 
            & = -r(y)^2d\theta.
\end{split}
\end{align}
Therefore, $S$ is transverse to $\SD_\trans$ if $|2r(y)^2r'(y)| < 1$. In particular, the construction yields a transverse surface whenever the function $r(y)$ is constant. \hfill $\Box$
\end{example}

Let us explain in more geometric terms the role of the non-integrability of $\SD_\trans$ in Equation \eqref{e:T2 transv}. We will work under the simplified assumption that $r(y)$ is constant: The $2$-torus $S$ intersects the level $\{y=y_0\}$ in a curve $\eta$ that is transverse to $\ker(\beta_\trans)$ (and in particular transverse to $\partial_w$). The front projection of $\eta$ (i.e. its projection to the $(x,z)$-plane) is a planar curve that does not depend on $y_0$; it is a figure-eight, as depicted in Figure~\ref{fig:Torus1}. The short line segments in Figure~\ref{fig:Torus1} indicate how the line field $T\SD_\trans \cap \{w=w(y,\theta)\}$ varies along the knot. Using $\beta_\trans$ we can recover the missing coordinate $w$ from the line field.

If $\SD_\trans$ was \emph{integrable} and tangent to $\langle \partial_w\rangle$, the line field $T\SD_\trans \cap \{w=w_0\}$ would be independent of $w_0$ and it would project to a non-singular line field in the $(x,z)$-plane. Since the projection of $\eta$ is a closed curve, it would necessarily be tangent to this line field somewhere. The non-integrability of $\SD_\trans$ precisely implies that $T\SD_\trans \cap \{w=w_0\}$ varies with $w_0$, providing enough flexibility to construct embedded transverse curves in the projection and therefore embedded transverse tori in the model.

\begin{figure}[ht] 
\centering
\includegraphics[scale=0.8]{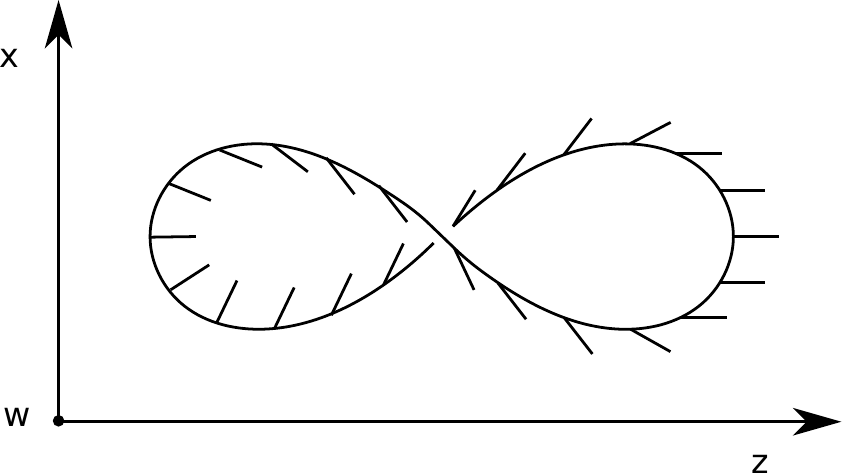}
\caption{Cross section of a transverse torus that is $C^0$-close to a transverse knot $\gamma$.} 
\label{fig:Torus1}
\end{figure}

Example~\ref{ex:explicitParametrisationTorus} can be generalized as follows.
\begin{lemma} \label{lem:constantProfile}
Let $\eta$ be a transverse knot in $(\R^3,\xi_\std=\ker(\beta_\trans))$. The surface 
$$ S = I \times \eta \subset (I \times \R^3,\SD_\trans) $$  
is embedded and transverse to $\SD$.
\end{lemma}
\begin{proof}
Since $S$ is tangent to $\partial_y$, it is transverse to $\ker(\alpha_{\trans})$. It is also transverse to $\ker(\beta_{\trans})$ because $\eta$ itself is transverse to this hyperplane field. Using $\beta_\trans(\partial_y)=0$ we have: 
$$ 
\alpha_\trans\wedge\beta_\trans\left(\frac{\partial}{\partial y},\dot{\eta}\right)  
 = \alpha_\trans(\partial_y)\beta_\trans(\dot{\eta})-\alpha_\trans(\dot{\eta})\beta_\trans(\partial_y) 
 = \beta_\trans(\dot{\eta}) > 0. 
$$
\end{proof}
In order to generalize this construction we fix the following terminology. 
\begin{definition} 
Let $\gamma$ be a transverse curve. Suppose $(\eta_y)_{y\in I}$ is an isotopy of transverse knots in $(\R^3,\xi_\std)$ such that
\begin{align*}
S: I \times \NS^1& \longrightarrow (I \times \R^3,\SD_\trans) \\
(y,\theta)    & \longmapsto \left(y,\eta_y(\theta)\right)
\end{align*}
is an embedded transverse surface lying in a standard neighborhood of $\gamma$. We say that $\gamma$ is the {\bf core} of $S$ and the family $(\eta_y)_{y\in I}$ is its {\bf profile}.  
\end{definition}
As stated before, we think of the profile as a family of planar curves in the front projection $(x,z)$ to which we add an oriented line every point to record the missing coordinate $w$. The explicit expression we gave in Example~\ref{ex:explicitParametrisationTorus} corresponds to the constant profile shown in Figure~\ref{fig:Torus1}.

\subsubsection{Scaling of profiles}

In order to build more general surfaces using gluing constructions, we will use diffeomorphisms of the form
\begin{align}
\begin{split} \label{e:psi lambda}
\psi_\lambda : I \times \R^3 & \longrightarrow I \times \R^3 \\
(y,x,z,w) & \longmapsto (y, \lambda^2(y) x, \lambda(y) z, \lambda(y) w)
\end{split}
\end{align}
where $\lambda$ is a positive function of $y$.

\begin{lemma}\label{lem:smallProfile}
Let $(\eta_y)_{y \in I}$ be an isotopy of transverse knots in $(\R^3,\xi_\std)$ for a compact $1$-manifold $I$. The embedded surface 
\begin{align*}
S: I \times \NS^1& \longrightarrow (I \times \R^3,\SD_\trans) \\
(y,\theta)    & \longmapsto \psi_\lambda\left(y,\eta_y(\theta)\right)
\end{align*}
is transverse to $\SD_\trans$ if $\lambda(y)=\lambda>0$ is small enough.
\end{lemma}
\begin{proof}
For $\lambda=1$ there is no guarantee that $S$ is transverse to $\SD_\trans$ because of the uncontrolled behavior of $\eta_y$ as $y$ varies. We use the notation
\begin{alignat*}{3}
\psi_\lambda^*\alpha_\trans \quad= &&\quad dy-\lambda^3zdx   \quad=:&&\quad \alpha_\lambda \\
\psi_\lambda^*\beta_\trans  \quad= &&\quad \lambda^2(dx-wdz) \quad=:&&\quad \beta_\lambda.
\end{alignat*}
Again,  we compute
\begin{align*}
\alpha_\trans\wedge\beta_\trans\left(\frac{\partial S}{\partial y}(y,\theta),\frac{\partial S}{\partial \theta}(y,\theta) \right) & = 
\alpha_\lambda\wedge\beta_\lambda \left(\partial_y+\frac{\partial\eta_y}{\partial y}(\theta), \dot{\eta}_y(\theta)\right) \\
& = (1+O(\lambda^3))\lambda^2C + O(\lambda^5).
\end{align*}
Here $C>0$ is a lower bound for $\beta(\dot{\eta}_y)$ and the constants implicit in the Landau symbols depend only on the family $(\eta_y)_{y\in I}$. Therefore, $S$ is transverse if $\lambda>0$ is sufficiently small.
\end{proof}

Finally, we show how to isotope a cylinder, relative to the boundary and through transverse surfaces, so that its profile $\eta_y$ can become arbitrarily small for some values of $y$ if a quantitative condition is satisfied.
\begin{lemma}\label{lem:scalingProfile}
Let $\eta$ be a transverse knot in $(\R^3,\xi_\std)$. Then there exists a constant $\tau_0$  depending only on $\eta$ with the following property.

For every function $\lambda: I \lra (0,1]$ satisfying $|\lambda'(y)| < \tau_0$, the surface
\begin{align*}
S: I \times \NS^1 &\longrightarrow (I \times \R^3, \SD_\trans) \\
(y,\theta) &\longmapsto \psi_\lambda(y,\eta(\theta))
\end{align*}
is transverse to $\SD$.
\end{lemma}
\begin{proof}
The proof is almost identical to the one of Lemma~~\ref{lem:smallProfile}, and we use notation similar to the one introduced there. The pullbacks of $\alpha_{\trans}$ and $\beta_{\trans}$ under $\psi_\lambda$ are  
\begin{alignat*}{3}
\psi_\lambda^*\alpha_\trans \quad=&&\quad (1-2\lambda^2\lambda'xz)dy - \lambda^3zdx    \quad=:&&\quad \alpha_\lambda \\
\psi_\lambda^*\beta_\trans  \quad=&&\quad \lambda^2(dx-wdz) + \lambda\lambda'(2x-zw)dy \quad=:&&\quad \beta_\lambda.
\end{alignat*}
Again, we seek a condition on $\lambda$ which ensures that $S$ is a transverse surface: 
\begin{align*}
\alpha_\trans\wedge\beta_\trans \left(\frac{\partial S}{\partial y}(y,\theta),\frac{\partial S}{\partial \theta}(y,\theta) \right) & =
\alpha_\lambda\wedge\beta_\lambda(\partial_y, \dot{\eta}(\theta)) \\
&  = (1-2\lambda^2\lambda'xz)\lambda^2C - O(\lambda^4\lambda'). 
\end{align*}
As above, $C$ is a lower bound for $\beta(\dot{\eta})$ and the constants implicit in the Landau symbols depend only  on $\eta$. Hence,  if $\lambda^2\lambda'$ is sufficiently small, then $S$ is transverse to $\SD_\trans$. This is immediate if $|\lambda'|<\tau_0$ for a sufficiently small constant $\tau_0$. 
\end{proof}
Thus, if the core of a cylinder is sufficiently long, it is possible to shrink its profile as much as we want over some subinterval. This will be crucial when we define the Engel overtwisted disc.

\begin{remark} \label{rem:moving core is moving Lutz twist}
Consider an isotopy $(\gamma_t)_{t\in[0,1]}$ of curves transverse to $\SE$. Let $S_0$ be a transverse surface with core $\gamma_0$ and profile $(\eta_y)_{y \in I}$. Suppose that the profile (and thus $S_0$) can be shrunk into an arbitrarily small tubular neighborhood of $\gamma_0$ using the dilation from Equation \eqref{e:psi lambda} (as in Lemma \ref{lem:smallProfile}). Then one obtains a family of surfaces $(S_t)_{t\in[0,1]}$ transverse to $\SD$ by shrinking $(\eta_y)_{y \in I}$ until it fits into a tubular neighborhood of $\gamma_0$ that can be identified with the standard neighborhoods of all the $\gamma_t$.

Moreover, once the profile is sufficiently thin, one can apply Lemma~\ref{lem:smallProfile} to interpolate through transverse surfaces between $S_0$ and some other $S_1$, lying in a standard neighborhood of $\gamma_1$, which is given in terms of some other profile $(\tilde\eta_y)_{y \in I}$. 
\end{remark}

\subsection{Universal twist systems} \label{ssec:universalTwistSystem}

We finish the section describing a particular collection of transverse surfaces. This collection of cylinders will play an important role in the proof of our main theorem, and the properties we require are motivated by that proof.

Let $I$ be the interval $[a,b]$. Fix coordinates $(x,z,w)$ in $\R^3$ and write $\R^3_- = \{z\le 0\}$. Suppose we are given constants $\lambda_0,\delta_0,\varepsilon_0 > 0$ which are potentially very small. We will now construct a $t$-dependent family $(\SS_t)_{t \in [0,1]}$ of embedded transverse surfaces in $(I \times \R^3, \SD_\trans)$, depending also smoothly on the parameters $\lambda_0$, $\delta_0$, and $\varepsilon_0$. Each $\SS_t$ will be a collection of infinitely many cylinders.

The cylinders comprising $\SS_t$ are all copies of a single cylinder shifted in the $x$-direction by   consecutive applications of the translation
\begin{align*} 
T_\lambda : I \times \R^3 & \lra I \times \R^3_- \\
(y,x,z,w) & \longmapsto (y,x+\lambda^2,z,w).
\end{align*} 
The core of each cylinder will be a curve parallel to the $y$-axis. The construction is done in several steps and it is only at the very end that we will achieve transversality with $\SD_\trans$. We abbreviate $T := T_1$.

\subsubsection{A nice transverse knot}

Let $\eta_1\subset (\R^3_-,\xi_\trans=\ker(\beta_\trans))$ be a transverse unknot with self-linking number $-3$ satisfying the following properties:
\begin{enumerate}
\item $\eta_1$ is disjoint from $\{z=0\}$ in the complement of $\eta_1^+ = \eta_1 \cap \{z,w=0\}$.
\item $\cup_{n\in\Z} T^n(\eta_1^+)$ is a covering of the $x$-axis. 
\item $\eta_1$ is disjoint from $T^n(\eta_1)$ for $n \neq \pm 1$.
\item $\eta_1$ and $T(\eta_1)$ intersect in a non-degenerate interval $P$ which is contained in the $x$-axis, and they are otherwise disjoint. $P$ and $T(P)$ are disjoint.
\item \label{isotopy prop 1} There is an isotopy $(\eta_t)_{t\in[1/2,1]}$ of $\eta_1$, supported in a neighborhood of $P$, such that $\eta_t$ is disjoint from $T(\eta_t)$ for $t \in [1/2,1)$. This isotopy is purely vertical: only the $w$-coordinate varies. The following condition holds for $t<1$:
\begin{itemize}
\item[(B)] Given any two points $(x,z,w)\in\eta_t$ and $(x,z,w')\in T(\eta_t)$, it holds that $w<w'$.  \label{condition B} 
In particular, the link $\cup_{n\in\Z} T^n(\eta_t)$ is an unlink. 
\end{itemize}
\item \label{isotopy prop 2} There is an isotopy $(\eta_t)_{t\in[0,1/2]}$ of $\eta_{1/2}$ such that $\eta_0$ is contained in $\{x_0-1/2<x<x_0+1/2\}$ for some $x_0$. This isotopy induces a $T$-equivariant isotopy of unlinks $(\cup_{n\in\Z} T^n(\eta_t))_{t\in[0,1/2]}$ such that all the components of $\cup_{n\in\Z} T^n(\eta_0)$ have disjoint front projection. We further assume that the isotopy is constant for $t \in [0,1/4]$.
\end{enumerate}
The isotopy $(\eta_t)_{t\in[0,1]}$ can be assumed to be smooth. A knot $\eta_1$ with the desired properties exists: its front projection is shown in Figure~\ref{fig:univ-twist-overlay-neu}, along with the front projection of the translation $T(\eta_1)$. The isotopy $(\eta_t)_{t\in[0,1/2]}$ separating the front projections simply makes the upper half of the knot progressively smaller.

\begin{figure}[ht]
\begin{center}
\includegraphics[scale=0.7]{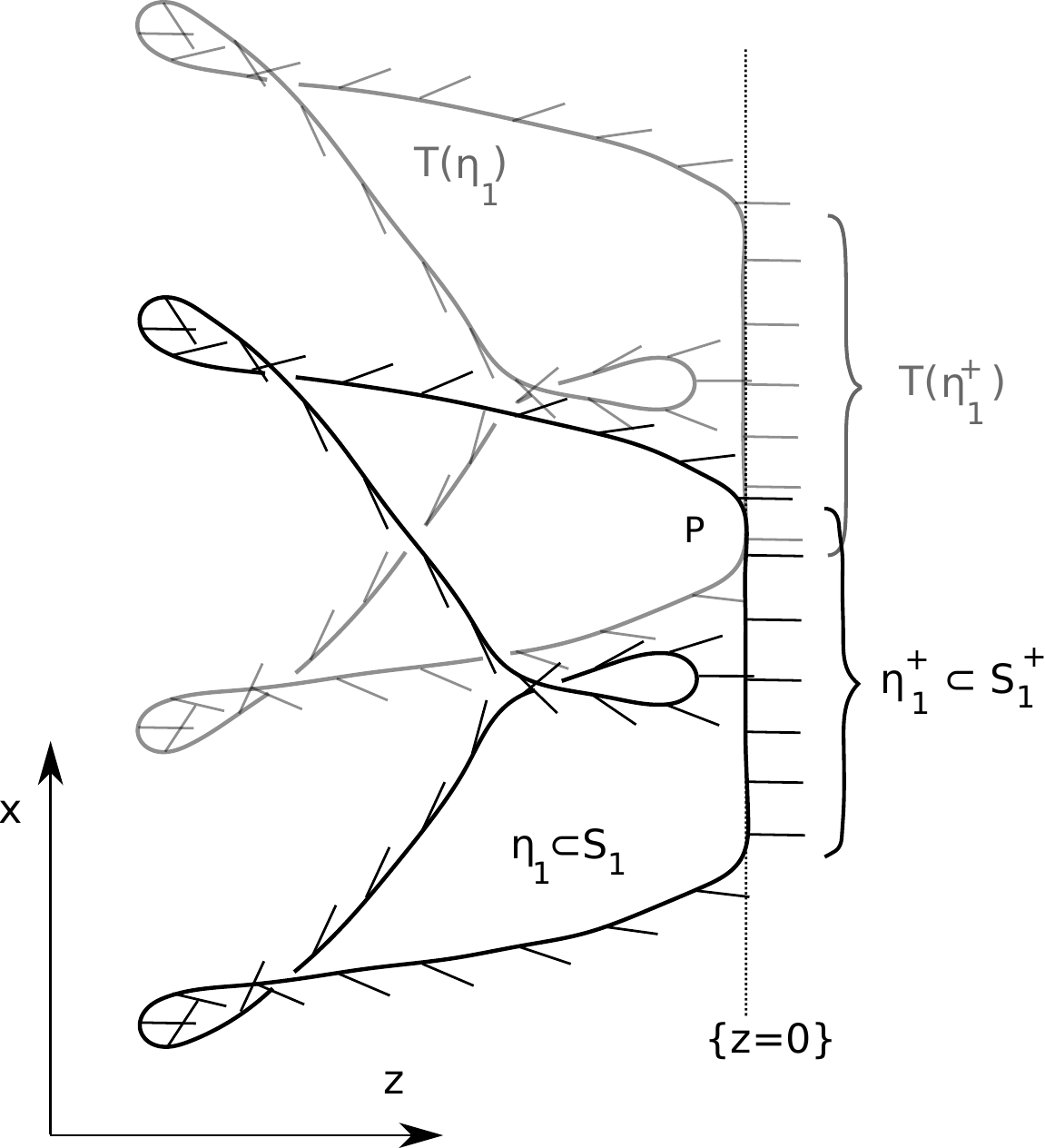}
\caption{The profile $\eta_1$ (in black) arising in the construction of universal twist systems and its translate $T(\eta_1)$ (in gray). The $w$-direction points away from the reader. \label{fig:univ-twist-overlay-neu}}
\end{center}
\end{figure}

\subsubsection{An infinite family of cylinders} 

Write $\R^3_-(y_0)$ for the half-hyperplane $\{y=y_0\} \subset I \times \R^3_-$. Recall the parameters $\lambda_0,\delta_0,\varepsilon_0 > 0$ introduced at the beginning of the subsection. We now fix a smooth family of bump functions $(\chi_t)_{t \in [0,1]}: [a,b] \lra [0,1]$ with the following properties.
\begin{itemize}
\item $\chi_t(y) \equiv t$ if $y \in [a+\delta_0,b-\delta_0]$.
\item $\chi_t(y) \equiv 0$ if $y \in [a,a+2\delta_0/3] \cup [b-2\delta_0/3,b]$.
\item $\chi_t$ is non-decreasing if $y\in[a,a+\delta_0]$.
\item $\chi_t$ is non-increasing if $y\in[b-\delta_0,b]$.
\end{itemize}
We may construct another smooth family of functions $(\rho_t)_{t \in [0,1]}: [a,b] \lra [\lambda_0,1]$ satisfying:
\begin{itemize}
\item $\rho_t(y) \equiv 1$ if $y \in [a+2\delta_0/3,b-2\delta_0/3]$ and $t \in [1/4,1]$.
\item $\rho_t(y) \equiv \lambda_0$ if $y \in [a,a+\delta_0/3] \cup [b-\delta_0/3,b]$ or $t \in \Op(0)$.
\item $\rho_t$ is non-decreasing if $y\in[a,a+2\delta_0/3]$.
\item $\rho_t$ is non-increasing if $y\in[b-2\delta_0/3,b]$.
\item $t \to \rho_t(y)$ is non-decreasing.
\end{itemize}
\begin{remark} \label{rem:lambda0}
The functions $\rho_t$ and $\chi_t$ depend on the parameters $\delta_0$ and $\lambda_0$. We may assume that this dependence is smooth. Further, we assume that the derivative $\rho_t'$ remains uniformly bounded as $\lambda_0$ goes to zero. \hfill $\Box$
\end{remark}

These two families of functions allow us to define a family of (not yet transverse) cylinders $(S_t)_{t \in [0,1]}$ by imposing 
\[ S_t \cap \R^3_-(y) = \psi_{\rho_t(y)}(\eta_{\chi_t(y)}) \]
where $\psi_\lambda$ is the scaling function defined in Equation~\ref{e:psi lambda}. What we achieve is the following: For $t \in [1/4,1]$ and $y \in [a+\delta_0,b-\delta_0]$, the profile of $S_t$ is precisely $\eta_t$. As $y$ approaches $0$ or $1$, the profile becomes $\eta_0$ and then it decreases in size; the precise scaling is given by the constant $\lambda_0$. Similarly, the surfaces $(S_t)_{t \in [0,1/4]}$ have a constant profile which is a rescaling of $\eta_0$ which becomes smaller as $t$ goes to $0$.

\subsubsection{Scaling}

Consider the collections of infinite cylinders $(\bigcup_{n\in\Z}T^n(S_t))_{t \in [0,1]}$. The proof of Lemma~\ref{lem:smallProfile} can be applied parametrically in $t$. Using this fact we claim that there is a constant $\lambda >0$, depending smoothly on $\lambda_0$, $\delta_0$, and $\varepsilon_0$, such that the collection of surfaces
\[ \SS_t = \psi_\lambda\left(\bigcup_{n\in\Z}T^n(S_t)\right) =  \bigcup_{n\in\Z}T^n_\lambda(\psi_\lambda(S_t)) \]
satisfies:
\begin{itemize}
\item $\SS_t$ is a collection of transverse cylinders,
\item if $t\neq 1$, the cylinders in $\SS_t$ are disjoint.
\item $\SS_t \subset \{|z|,|w| < \varepsilon_0\}$.
\end{itemize}
The first fact follows from Lemma~\ref{lem:smallProfile},  the second is immediate from the construction. Finally, the third holds for any $\lambda$ sufficiently small (depending on $\varepsilon_0$).


\begin{definition} \label{def:universalTwistSystem}
The collection of transverse cylinders $(\SS_t)_{t\in[0,1]}$ is a {\bf universal twist system}. 
\end{definition}

\begin{remark} \label{rem:twistSystemConstants}
The family $(\SS_t)_{t\in[0,1]}$ depends on the three parameters $\lambda_0$ (an upper bound for the size of the profiles at $y=0,1$ and $t=0$), $\delta_0$ (the size of the interval in the $y$-direction in which the profiles perform the unlinking and the shrinking), and $\varepsilon_0$ (an upper bound for the size of the cylinders in the $(x,z,w)$-coordinates). Let us clarify how they depend on each other.

We fix $\delta_0$ and $\varepsilon_0$ first. The constant $\lambda$ that determines the final scaling of the profiles depends on both; as they converge to zero so does $\lambda$. However, Lemma~\ref{lem:scalingProfile} and Remark~\ref{rem:lambda0} imply that $\lambda$ does not depend on $\lambda_0$. We will fix $\lambda_0$ last in our constructions.
\end{remark}

\section{The Engel-Lutz twist} \label{sec:Lutz}

Having constructed transverse tori, we immediately obtain many examples of transverse hypersurfaces in $(M,\SD)$: Consider the boundary of a suitable tubular neighborhood of a given transverse torus. If the torus is contained in the vicinity of its core $\gamma$, so is the $3$-manifold we obtain.

In this section we describe how to add \emph{Engel torsion} along a transverse $3$-manifold; this is analogous to adding \emph{Giroux torsion} along a transverse torus in $3$-dimensional Contact Topology. The special case when this operation is performed along a transverse $3$-torus obtained from a transverse loop $\gamma$ is what we call an \emph{Engel-Lutz twist} with core $\gamma$. These procedures are well-defined up to homotopy through Engel structures.

Additionally, we show that if the core is contractible, then the Engel-Lutz twist preserves the homotopy type of the Engel framing. Similarly, the formal type of the Engel structure is unchanged when one introduces Engel torsion along the same transverse $3$-manifold twice. Again, all of this is very reminiscent of the Lutz twist.

\subsection{Hypersurfaces transverse to the Engel structure} \label{ssec:transverseHypersurfaces}

\begin{definition}
Let $(M,\SD)$ be an Engel manifold. An immersed $3$-dimensional submanifold $N\subset M$ is a {\bf  transverse hypersurface} if its tangent space at each point is transverse to $\SD$. 
\end{definition}
Given a closed, orientable, embedded, transverse $3$-manifold $N \subset (M, \SD)$, we want to describe the germ $\SD|_N$. First, we note that $\SD$ imprints the following data on $N$:
\begin{itemize}
\item a line field $\SH_N = TN \cap \SD$, and 
\item a $2$-plane field $\xi_N = TN \cap \SE$ which contains $\SH_N$.
\end{itemize}
Of course, if $N$ is not only transverse to $\SD$ but in addition transverse to $\SW$, then $\xi_N$ is a contact structure and $\SH_N$ is a Legendrian line field. 

Once an orientation of $N$ is fixed, we can define subsets 
\begin{align*}
N^+ & = \{x\in N\,|\,\xi_N \textrm{ is a positive contact structure close to }x\}\\
N^- & = \{x\in N\,|\,\xi_N \textrm{ is a negative contact structure close to }x\}\\
N^0 & = N\setminus(N^+\cup N^-).
\end{align*}
Equivalently, $N^0$ is the set where $\SW$ is tangent to $N$. By definition, the regions $N^\pm$ are open while $N^0$ is closed. If $N$ is chosen $C^\infty$-generically, $N^+$ and $N^-$ are open $3$-dimensional manifolds separated by the possibly disconnected surface $N^0$.

Let $T$ be a vector field transverse to $N$ and tangent to $\SD$; denote its flow by $\phi_t$. Using $\phi_t$ for small times $t \in [-\varepsilon,\varepsilon]$, we obtain a tubular neighborhood $\SU(N) \cong N\times(-\varepsilon,\varepsilon)$ of $N \cong N\times\{0\}$. We use coordinates $(p,t) \in N\times(-\varepsilon,\varepsilon)$ in the model, i.e. $\partial_t=T$. Choosing $T$ suitably we may assume that the following conditions hold: \label{choice of T}
\begin{itemize}
\item $T$ is tangent to $\SW$ in the complement of an arbitrarily small neighborhood of $N^0 \times (-\varepsilon,\varepsilon)$,
\item If $T(p,0)$ is tangent to $\SW$ for $p\in N$, then the $T$-orbit $\cup_{t\in(-\varepsilon,\varepsilon)} \phi_t(p)$ is part of the leaf of $\SW$ through $p$. 
\item Conversely, if $T(p,0) \notin \SW$ for $p\in N$, then the segment $\cup_{t\in(-\varepsilon,\varepsilon)} \phi_t(p)$ is never tangent to $\SW$.  
\end{itemize}

Write $N_t=N\times\{t\}$. Fix a vector field $H$ on $N\times(-\varepsilon,\varepsilon)$ such that $\SD \cap TN_t$ is spanned by the restriction of $H$ to $N_t$ and let 
\begin{align*}
H' & = [\partial_t,H]& H'' & = [\partial_t,H']. 
\end{align*}
We can associate a sign to each connected component of $N^0$ indicating whether the framing $\langle H,H',H'' \rangle$ of $TN$ is positive or negative. We write $N^{0,+}$ (respectively $N^{0,-}$) for the union of those connected components of $N^0$ in which it is positive (respectively negative).

\begin{proposition} \label{prop:localModel}
Let $(M,\SD)$ be an Engel manifold and $N \subset M$ a transverse hypersurface with framing $\{X \in \SD \cap TN,Y \in \SE\cap TN,Z\}$. Then, the Engel structure in $\SU(N)$ can be  written as 
$$ \SD(p,t) = \langle \partial_t, H = X+tY+g(p,t)Z \rangle, $$
where $g:\SU(N) \cong N \times (-\varepsilon,\varepsilon) \lra \R$ is a function with the following properties:
\begin{itemize}
\item $g(p,t) = 0$ if $\partial_t \in \SW(p,0)$,
\item $g(p,t)$ is convex if $p\in N^{0,+}$, 
\item $g(p,t)$ is concave if $p\in N^{0,-}$.
\end{itemize}
\end{proposition}
\begin{proof}
Consider the framing $\{T,X,Y,Z\}$ of $TM$ along $N$. Using the flow of $T$ we can extend this framing to a translation invariant framing on  $\SU(N)$. The Engel structure $\SD$ can now be described by a smooth family of curves $(H_p: (-\varepsilon,\varepsilon) \lra \NS^2)_{p\in N}$ such that $\SD(p,t)$ is spanned by  
$$  T=\partial_t \quad\textrm{ and }\quad H(p,t) = (H_p(t))_1\cdot X + (H_p(t))_2 \cdot Y +(H_p(t))_3\cdot Z. $$
Since $\SD$ is an Engel structure, the curves $H_p$ satisfy at least one of the conditions in Proposition~\ref{prop:EngelCharacterisation}. The conditions on $T$ stated above ensure that each curve $H_p$ is an immersion whose image is either contained in a great circle or is convex/concave everywhere. In either case, $H_p(t)$ is graphical over the equator 
$$
t \longmapsto (H_p(t))_1\cdot X + (H_p(t))_2\cdot Y.
$$ 
Shrinking $\varepsilon$ and scaling $T$ appropriately one obtains the claim using the implicit function theorem.
\end{proof}

The following elementary constructions provide many examples of transverse $3$-mani\-folds.
\begin{definition} \label{def:thin nbhd}
Let $(M,\SD)$ be an Engel manifold and let $S$ be a transverse surface. A trivialized tubular neighborhood $U\simeq S\times \D^2$ of $S$ is {\bf thin} if the $3$-manifolds $S\times\partial \D^2_\varepsilon$ are transverse to $\SD$ for all $0<\varepsilon\le 1$. 
\end{definition}

\begin{lemma} \label{lem:t3fromt2}
Any transverse surface in an Engel manifold has a thin tubular neighborhood.
\end{lemma}

The transverse submanifolds obtained in this way are confined to small neighborhoods of $S$, but once a transverse hypersurface is available it can be isotoped through transverse hypersurfaces.
\begin{lemma} \label{lem:thickenTransverse3fold}
Let $N$ be a hypersurface transverse to $\SD$. Let $X$ be a vector field tangent to $\SD$ and transverse to $N$ with flow $\varphi_t$. Then $\varphi_t(N)$ is a transverse hypersurface for all $t$.
\end{lemma}

\subsection{Adding Engel torsion along a transverse hypersurface} \label{ssec:Lutz}

Let us continue using the notation introduced in the previous section. We want to modify $\SD$ on a neighborhood $\SU(N)$ of $N$ such the resulting Engel structure $\SL(\SD)$ satisfies:
\begin{itemize}
\item $\SL(\SD)$ coincides with $\SD$ outside of $\SU(N)$,  
\item $\SL(\SD)$ is tangent to $T$, and
\item in the region where $T(p,0)$ is tangent to $\SW(p,0)$, the even-contact structures associated to $\SL(\SD)$ and $\SD$ are the same, but $\SL(\SD)$ performs one additional turn along the flow lines of $T$.
\end{itemize}
This is achieved by replacing the family of curves $(H_p)_{p\in N}$ by a family $(\lambda_p)_{p\in N}$. First, we construct a $C^2$-family of $C^2$-curves $(\eta_p)_{p\in N}$ that has all the desired properties except that the curves are only piecewise smooth. In a second step we smooth $(\eta_p)_{p\in N}$ (both individually and as a family) to obtain $(\lambda_p)_{p\in N}$. During the argument we do not keep track of the paremetrisation of the curves since any regular paremetrisation works. 

We now describe $(\eta_p: (-\varepsilon,\varepsilon) \lra \NS^2)_{p\in N}$. Each curve $\eta_p$ consists of three pieces. The first and third pieces are, respectively, $H_p((-\varepsilon,0])$ and $H_p([0,\varepsilon))$. The middle piece consists of the (unique, possibly non-maximal) circle in $\NS^2$ which
\begin{itemize}
\item passes through $(1,0,0)$, 
\item is tangent to the maximal circle spanned by $(1,0,0)$ and $(0,1,0)$, and
\item has the same geodesic curvature as $H_p$ at $t=0$.
\end{itemize}
These three pieces depend smoothly on $p$, each one of them is smooth and, at the gluing points, the assumption on the curvature guarantees $C^2$-regularity.

The smoothing step from $\eta_p$ to $\lambda_p$ is clear. Indeed, since the relevant properties of the curves depend only on their $2$-jet and the segments $\eta_p$ are already $C^2$-smooth, any $C^2$-small smoothing yields a smooth family of curves with the same convexity properties as the family $(\eta_p)_{p\in N}$.

\begin{lemma}
The plane field $\SL(\SD)$ is an Engel structure.
\end{lemma}
\begin{proof}
Consider the curves $(\lambda_p)_{p\in N}$ from the construction. Whenever $H_p$ is convex (resp. concave), so is $\lambda_p$. Therefore, condition (i) from Proposition~\ref{prop:EngelCharacterisation} implies that $\SL(\SD)$ is Engel whenever this is the case. Let $U$ be the complement, i.e. the set of those $p$ such that $H_p$ is everywhere tangent to the maximal circle $C$ given by $(1,0,0)$ and $(0,1,0)$. The curves $\lambda_p$ are tangent to $C$ on $U$. Since $U$ is closed, we cannot invoke condition (ii) of Proposition~\ref{prop:EngelCharacterisation}. However, one notes that the maximal circles tangent to the curves $\lambda_p$ $C^\infty$-converge to $C$ as $(p,t)$ approaches $U$. This implies that condition (ii) holds in small neighborhood of $U$. This concludes the proof.
\end{proof}

\begin{definition}
We say that the plane field $\SL(\SD)$, obtained from $\SD$ and $N$ by the procedure we just described, is the result of introducing \textbf{Engel torsion} along $N$.

If $N$ is a $3$-torus arising from a $2$-torus $S$, as in Definition~\ref{def:thin nbhd}, we say that $\SL(\SD)$ is obtained from $\SD$ by introducing an \textbf{Engel-Lutz twist} along $S$. The tubular neighborhood in which the Engel-Lutz twist is performed is an {\bf Engel-Lutz tube} carried by $S$.
\end{definition}
The construction is shown pictorially in Figure \ref{fig:EngelLutzTwist}. In this figure, the kernel foliation is vertical, and each segment transverse to $N$ represents a fibre of the tubular neighborhood $\SU(N)$. The horizontal segment on the right goes through the region $N^{0,+}$. The vertical ones on the left are away from $N^0$ and are therefore tangent to the kernel. For each segment we draw two spheres showing the corresponding curves $H_p$ (left) and $\lambda_p$ (right). For the horizontal one, $H_p$ is a short convex curve and $\lambda_p$ is a convex loop with the same endpoints. For the vertical ones, $H_p$ is a short piece of equator and $\lambda_p$ is a curve running around that same equator with the same endpoints but describing one more turn; we draw $\lambda_p$ not overlapping with $H_p$ to show  the additional turn and the boundary condition more clearly.

\begin{remark}
Since $N$ is still transverse to $\SL(\SD)$, Engel torsion can be introduced several times along the same hypersurface $N$.
\end{remark}

\begin{figure}
\begin{center}
\includegraphics[scale=1.3]{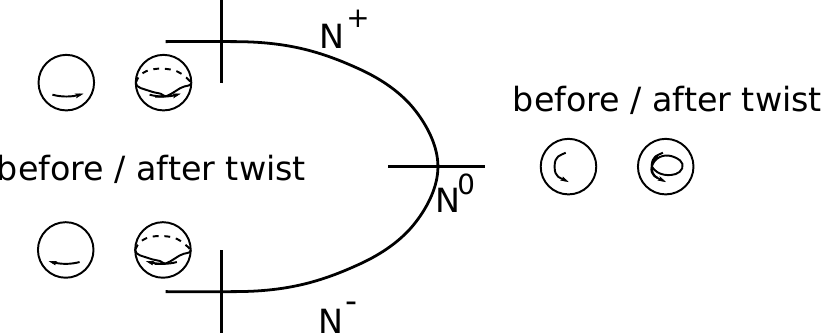} 
\caption{The Engel-Lutz twist in terms of the curves $H_p$ and $\lambda_p$ near a hypersurface $N=N^+\cup N^-\cup N^0$ (represented by the $U$-shaped curve). }
\label{fig:EngelLutzTwist}
\end{center}
\end{figure}

\begin{remark} \label{contractible choice in Engel Lutz}
The construction described above is well-defined up to deformations through Engel structures. To see this, it suffices to note that the vector fields $T$ satisfying the required properties in a neighborhood of $N$ form a contractible space. More generally, note that the construction can also be carried out parametrically when $(N_k)_{k \in K}$ is a family of closed manifolds transverse to a family of Engel structures $(\SD_k)_{k \in K}$. In the proof of the $h$-principle for overtwisted Engel structures, it will be essential to introduce Engel-Lutz twists parametrically along a family of transverse surfaces. \hfill $\Box$
\end{remark}

The addition of Engel torsion along a hypersurface has already appeared in the literature in less generality, which we now recall. The first example deals with Engel structures obtained from contact structures which are trivial as bundles.
\begin{example} \label{ex:adding a twist}
Let $(N,\xi)$ be a contact $3$-manifold admitting a global framing $C_1,C_2$. Let $t$ be the coordinate $\NS^1=\R/2\pi\Z$. For each positive integer $k$, we define an Engel manifold 
$$  \left(N\times \NS^1, \SD_k = \left\langle \partial_t, X_k=\cos(kt)C_1+\sin(kt)C_2 \right\rangle\right). $$
The even-contact structure of $\SD_k$ is the preimage of $\xi$ under the bundle projection and the characteristic foliation $\SW$ is spanned by $\partial_t$. The hypersurfaces $N\times\{t\}$ are clearly transverse to $\SD_k$ (indeed, they are transverse to their kernel $\SW$). Introducing Engel torsion to $\SD_k$ along $N\times\{t\}$ produces $\SD_{k+1}$. It is easy to see that the homotopy type of the Engel framing changes when one passes from $\SD_{k}$ to $\SD_{k+1}$, but it only depends on the parity of $k$ (see \cite{ks,dp}).
\end{example}

The second example deals with the Engel structures on mapping tori. It is very similar, but not equivalent, to the previous example.
\begin{example}
In \cite{ge-review} H.~Geiges constructs an Engel structure $\SD_k$ on any mapping torus with trivial tangent bundle.  The Engel structures $\SD_k$ produced in this way are tangent to the suspension vector field. Trivializing the suspension vector field allows us to express $\SD_k$ as a $3$-dimensional family of curves, as in Proposition~\ref{prop:EngelCharacterisation}. The key property of his construction is that these curves become convex if $k$ is large enough. The fibers are $3$-manifolds transverse to $\SD_k$ and passing from $\SD_k$ to $\SD_{k+1}$ amounts to introducing Engel torsion along one of them.
\end{example}

\subsection{Modification of the formal data when adding Engel torsion} \label{ssec:formalData}

Example~\ref{ex:adding a twist} shows that an Engel-Lutz twist sometimes does change the homotopy type of the Engel framing.  We first show that the Engel-Lutz twist applied twice along the same hypersurface results in an Engel structure with an Engel framing that is homotopic to the original one. This is analogous to the case of contact structures. 

\begin{proposition} \label{p:double twist}
Let $\SD$ be an Engel structure on $M$ and let $N$ be a transverse hypersurface. The Engel framing of $\LLSD$ is homotopic to the Engel framing of $\SD$.  
\end{proposition}
\begin{proof}
We will make use of the notation introduced in Subsections~\ref{ssec:transverseHypersurfaces} and~\ref{ssec:Lutz}. In the first part of this proof we compare the pairs $(\SD,\SE)$ and $(\LSD,[\LSD,\LSD])$. The need for the double Lutz twist will come into play only in the second part. 

The even-contact structure $\SE=[\SD,\SD]$ has basis $\{T,H_p(t),\dot{H}_p(t)\}$ at $(p,t)\in N(U)$. We claim that the even-contact structure associated to $\LSD$ is homotopic to $\SE$. Recall the curves $(\eta_p)_{p \in N}$, those appearing in the intermediate step of the construction of an Engel-Lutz twist, before smoothing. It suffices to see that one can homotope the middle piece of $\eta_p$, through non-maximal circles, to the great circle spanned by $H_p(0)$ and $\dot{H}_p(0)$. The plane field $\SD_\weak$ associated to the homotoped curves is not integrable and the hyperplane field $\SE_\weak = [\SD_\weak,\SD_\weak]$ is homotopic to the even contact structure of $\LSD$ by construction. It is clear that $\SE_\weak$ is in turn homotopic to $\SE$. Iterating this argument, we see that the even-contact structure associated to $\LLSD$ is homotopic to $\SE$ as well.

Notice that $T$ is tangent to all the distributions throughout all homotopies that have been discussed so far. In the following, we may therefore pretend that the hyperplane fields underlying the even-contact structures associated to $\SD$, $\LSD$, and $\LLSD$ are the same (as long as we do not use assumptions on their characteristic foliation). 

In the model, the plane field $\SD$ is coorientable within the oriented hyperplane field $\SE$, so we may fix a coorientation. Then, one can homotope $\SD$ inside of $\SE$ to a plane field which coincides with $\SD$ in a neighborhood of $\partial \SU(N)$ and which is transverse to $T$ in the interior of the model. Choose the neighborhood of $\partial \SU(N)$ so small that $\SD$, $\LSD$, and $\LLSD$ coincide there. Then, we tilt $\SD$ away from $T$ (staying within $\SE$) in the direction prescribed by its coorientation until it coincides with $\SE\cap(N\times\{t\})$. The same can be done, exactly by the same argument, with $\LSD$ and $\LLSD$. Therefore, the pairs 
$$
\big(\SD,\SE\big), \big(\LSD,[\LSD,\LSD]\big),\textrm{ and }\big(\LLSD,[\LLSD,\LLSD]\big)
$$ 
are homotopic as oriented pairs of hyperplane fields containing a plane field.

Now we pass to the second part of the proof, which focuses on $\LLSD$. Let us abbreviate $[\LLSD,\LLSD]=\LLE$ and write $\LLW$ for its kernel. It is difficult to compare the characteristic foliation of $\LSD$ with the characteristic foliation of $\SD$, but we can homotope the pair $(\LLSD,\LLE)$ to $(\SD,\SE)$ and carry the characteristic foliation of $\LLSD$ along (as a family of oriented vector fields contained in the plane fields). 

By construction, we can view $\LLSD$ as obtained from $\SD$ by inserting two copies of the same layer $N\times(-\delta,\delta)$ carrying the same Engel structure (that is, two copies of the Engel torsion model). We can then choose the homotopy connecting $(\LLSD,\LLE)$ and $(\SD,\SE)$ in such a way that the two layers always coincide at every step, including the homotopy of $\LLW$. Denote by $\SW_1 \subset \SD$ the line field at the end of the homotopy.

Then, $\SW_1$ is a line field in $\SD$ which coincides with $\SW$ near $\partial \SU(N)$ and, in the interior, describes twice the same loop within $\SD$ as one moves along the fibers of $\SU(N) = N\times(-\varepsilon,\varepsilon)\lra N$. We have shown that the Engel framing of $\LLSD$ is homotopic to a framing whose last two components (i.e the components complementary to $\SD$) coincide with those of the Engel framing of $\SD$, while the first two describe a certain loop twice. Since $\pi_1(\SO(4)) = \Z_2$,  the Engel framings of $\LLSD$ and $\SD$ are homotopic.
\end{proof}

Without topological assumptions on $N$ it is not true, in general, that the Engel framings of $\LSD$ and $\SD$ are homotopic. However, the following statements are corollaries of the first part of the proof of the proposition:
\begin{corollary} \label{cor:homotopyChange1}
Assume that $N$ is contained in a ball. Then the Engel framing of $\LSD$ is homotopic to the Engel framing of $\SD$ relative to the boundary of the ball.
\end{corollary}
\begin{proof}
In view of the first half of the proof of Proposition~\ref{p:double twist} we may assume that, after a homotopy, the last two components of the Engel framings of $\LSD$ and of $\SD$ agree, and thus, the first two components span the same plane field. Since the ball is simply connected, the framings are homotopic. 
\end{proof}

\begin{corollary} \label{cor:homotopyChange2}
Assume that $N$ is the boundary $\partial(S \times \D^2)$ of a thin neighborhood of a transverse torus $S$. Then the Engel framing of $\LSD$ is homotopic, relative to the boundary of a slightly bigger neighborhood, to the Engel framing of $\SD$.

In particular, introducing an Engel-Lutz twist does not change the formal type of the Engel structure.
\end{corollary}
\begin{proof}
We argue as in the previous corollary, noting this time that  $\pi_1(\partial(S \times \D^2)) \lra \pi_1(S \times \D^2)$ is surjective. This implies that any two framings differing only in the first two components and agreeing along the boundary must be homotopic relative to the boundary.
\end{proof}

\section{Overtwisted Engel structures} \label{sec:otDisc}

In $3$-dimensional Contact Topology, a Lutz twist along a knot transverse to the contact structure gives rise to a $\NS^1$-family of overtwisted discs.  Motivated by this, we will define an Engel overtwisted disc to be a portion of an Engel-Lutz tube that satisfies an additional quantitative property.

\subsection{The overtwisted disc} \label{ssec:otDisc}

Let $I=[0,L]$ and fix coordinates $(y,x,z,w)$ on $I \times \D^3$. On $I\times\D^3$ we consider the Engel structure
\begin{equation} \label{e:normal form D}
\SD_\trans = \ker(\alpha = dy-zdx)\cap \ker(\beta = dx-wdz) = \langle \partial_w, \partial_z + w(\partial_x+z\partial_y)\rangle.
\end{equation}
Of course, this is the standard form of an Engel structure in the vicinity of a transverse curve $\gamma \cong I\times\{0\}$ obtained in Proposition~\ref{prop:modelKnot}.

Let $\eta \subset (\D^3,\ker(\beta))$ be a transverse unknot with self-linking number $-3$, as shown in Figure~\ref{fig:univ-twist-overlay-neu}. According to Lemma~\ref{lem:constantProfile}, the cylinder $\Sigma = I \times \eta$ with profile $\eta$ is transverse to $\SD_\trans$. Let $\tau_0$ be the scaling constant obtained from $\eta$ by applying Lemma~\ref{lem:scalingProfile}.

\begin{definition} \label{def:otDisc}
Consider the Engel structure $\SL(\SD_\trans)$ obtained from $\SD_\trans$ by an Engel-Lutz twist along $\Sigma = I \times \eta$. If the length of $I$ satisfies $L > \frac{2}{\tau_0}$, then
\[ \Delta_\OT = (I \times \D^3, \SD_\OT = \SL(\SD_\trans)) \]
is an \textbf{overtwisted disc}. The curve $I \times \{0\} \subset I \times \D^3$ is said to be its \textbf{core}.

An Engel structure is {\bf overtwisted} if it admits an Engel embedding of an overtwisted disc.
\end{definition}

\subsection{Self-replication of overtwisted discs}

As we showed in Lemma~\ref{lem:scalingProfile}, the requirement on $L$ ensures that we can shrink the transverse cylinder $\Sigma$ 
\begin{itemize}
\item through transverse embedded surfaces and 
\item relative to the boundary
\end{itemize} 
so that a piece of the resulting surface is an arbitrarily thin cylinder of small (but fixed) length. The existence of this homotopy guarantees the self-replicating property of the overtwisted disc, i.e. in the presence of an overtwisted disc we can produce new ones by an Engel homotopy. This is a property that appeared in a similar form in E.~Murphy's definition of a loose chart for higher dimensional Legendrians \cite{Mur}.
\begin{lemma} \label{lem:selfReplication}
Let $\Delta_\OT=([0,L]\times\D^3,\SD_\OT)$ be an overtwisted disc. Then, there are:
\begin{itemize}
\item a path of Engel structures $([0,L] \times \D^3,\SD_r)_{r\in [0,r_0]}$,
\item a path of Engel embeddings $F_r: ([0,L] \times \D^3,\SD_\OT) \longrightarrow ([0,L] \times \D^3,\SD_r)$, and
\item an Engel embedding $G: ([0,L] \times \D^3,\SD_\OT) \longrightarrow ([0,L] \times \D^3,\SD_{r_0})$
\end{itemize}
such that $\SD_r=\SD_\OT$ on a neighborhood of the boundary, $\SD_0=\SD_\OT$, $F_0$ is the identity, and the images of $G$ and $F_{r_0}$ are disjoint.
\end{lemma}
\begin{proof}
The Engel structure $\SD_\OT$ is obtained from the standard Engel structure $([0,L]\times\D^3,\SD_\trans = \ker(\alpha=dy-zdx)\cap\ker(\beta=dx-wdz))$ by applying an Engel-Lutz twist along the surface $\Sigma = I \times \eta$. By assumption, $L$ and the constant $\tau_0$ from Lemma~\ref{lem:scalingProfile} satisfy $2<L\tau_0$. 

Recall the $y$-dependent rescaling of the $\D^3$-factor
\begin{align*}
\psi_\lambda : [0,L] \times \D^3 & \longrightarrow [0,L] \times \D^3 \\
(y,x,z,w) & \longmapsto (y, \lambda(y)^2x, \lambda(y)z, \lambda(y)w),
\end{align*}
from Equation~\ref{e:psi lambda}. Since $L > \frac{2}{\tau_0}$, we can construct a path of functions $\lambda_r: [0,L] \lra (0,1]$, $r \in [0,1)$, such that
\begin{itemize}
\item $\lambda_0(t) = 1$ for all $t$,
\item $\lambda_r(t) = 1$ for all $t$ in a neighborhood of $\partial[0,L]$,
\item $\lambda_r(t) = 1-r$ if $t \in [Lr/2,L(1-r/2)]$, and
\item $|\lambda_r'(t)| < \tau_0$.
\end{itemize}
By Lemma~\ref{lem:scalingProfile}, the cylinders $\Sigma_r = \psi_{\lambda_r}([0,L] \times \eta)$ are embedded and transverse. Adding an Engel-Lutz twist to $\SD_\trans$ along $\Sigma_r$ yields a path of Engel structures $\SD_r$ in $[0,L] \times \D^3$ with $\SD_0 = \SD_\OT$. All of them agree near the boundary of the model.

The embeddings $F_r$ are 
\begin{align*}
F_r : ([0,L] \times \D^3,\SD_\trans) & \longrightarrow ([Lr/2,Lr/2+L(1-r)^3] \times \D^3,\SD_\trans) \\
(y,x,z,w) & \longmapsto \left(y(1-r)^3 + Lr/2, x(1-r)^2, z(1-r), w(1-r)\right).
\end{align*}
A simple computation shows that $F_r$ is a well-defined embedding, since $Lr/2+L(1-r)^3 \leq L(1-r/2) \leq 1$ for all $r$. Additionally, the conformal nature of the Engel structure implies that $F_r$ is an Engel embedding. Since $F_r^{-1}(\Sigma_r)=\Sigma$, we can invoke Remark~\ref{contractible choice in Engel Lutz} and regard $F_r$ as an Engel embedding of $([0,L] \times \D^3,\SD_\OT)$ into $([0,L] \times \D^3,\SD_r)$.

When $r$ is sufficiently close to $1$, the map 
\begin{align*}
G_r: ([0,L] \times \D^3,\SD_\trans) & \longrightarrow ([L/2,Lr/2+L(1-r)^3] \times \D^3,\SD_\trans) \\
(y,x,z,w) & \longmapsto \left(y(1-r)^3 + L/2, x(1-r)^2, z(1-r), w(1-r)\right).
\end{align*}
is well-defined and its image is disjoint from $F_r(I\times \D^3)$. We fix such a value $r_0$ for $r$ and set $G = G_{r_0}$.
\end{proof}
We do not know whether this result still holds when the condition on $L$ is dropped.

\begin{remark}
Mimicking the proof of the lemma one can show that the choice of constant $L$ in the definition of overtwisted disc is not important as long as $L\tau_0 > 2$.
\end{remark}

\subsection{Overtwisted Engel structures}

Finally, we define what an overtwisted family of Engel structures is.
\begin{definition} \label{def:OTfamily}
Let $K$ be a compact manifold. A family of Engel structures $\SD: K \lra \Engel(M)$ is \textbf{overtwisted} if there is a submanifold $\Delta \subset M \times K$ satisfying:
\begin{itemize}
\item $(\Delta,\SD) \lra K$ is a locally trivial fibration of Engel manifolds
\item whose fiber is Engel diffeomorphic to $\Delta_\OT$.
\end{itemize}
The manifold $\Delta$ is said to be the \textbf{certificate of overtwistedness} of $\SD$.
\end{definition}
That is, we require that there exists a family of overtwisted discs compatible with the family of Engel structures. We will often write $\Delta_k$ for an explicit paremetrisation of the overtwisted disc of $\SD(k)$. If $\Delta$ is not a globally trivial fibration, we cannot choose $\Delta_k$ parametrically in $k$ globally.

Examples of overtwisted Engel structures arise from Engel-Lutz twists.
\begin{lemma} \label{lem:LutzTwistOT}
Let $(M,\SD)$ be an Engel manifold. Let $\Sigma$ be a transverse $2$-torus with core transverse to $[\SD,\SD]$ and profile as depicted in Figure~\ref{fig:univ-twist-overlay-neu}. Then, the Engel manifold $(M,\SL(\SD))$ obtained by Engel-Lutz twisting along $\Sigma$ is overtwisted up to Engel homotopy. 
\end{lemma}
\begin{proof}
We homotope $\Sigma$ to make its profile arbitrarily small using Lemma \ref{lem:smallProfile}. This homotopy of transverse surfaces provides, by parametric Engel-Lutz twisting, a homotopy of Engel structures $(\SD_s)_{s \in [0,1]}$ which starts on $\SD_0 = \SL(\SD)$ and finishes on $\SD_1$ overtwisted. This follows from the fact that the shrinking allows us to ensure that the length constant $L$ is large enough.
\end{proof}
It is unclear to the authors whether this result still holds when $\Sigma$ is not obtained from a transverse knot.

An immediate consequence of the lemma is as follows:
\begin{corollary} \label{cor:existence}
Fix a smooth embedding $\Delta: [0,L] \times \D^3 \lra M$. The following inclusion induces a surjective map
$$ \Engel_\OT(M,\Delta) \lra \FEngel_\OT(M,\Delta). $$
\end{corollary}

\begin{proof}
Given any $\NS^m$-family $\SD_0$ of formal Engel structures in $\FEngel_\OT(M,\Delta)$, we apply the existence $h$-principle from \cite{cppp} to produce a family of Engel structures $\SD_{1/3}$ which is formally homotopic to $\SD_0$. Doing so destroys the overtwisted disc of $\SD_0$ because \cite{cppp} is not an $h$-principle relative in the domain (as pointed out in the introduction). Still, after a homotopy, we may assume that $\image(\Delta)$ is a Darboux ball for each $\SD_{1/3}(k)$, $k \in \NS^m$.

Choose a family of knots $\big(\gamma_k \subset \image(\Delta) \subset M\big)_{k \in \NS^m}$ with $\gamma_k$ transverse to $[\SD_{1/3}(k),\SD_{1/3}(k)]$. From this we obtain a $\NS^m$-family of $2$-tori. Parametrically introducing an Engel-Lutz twist yields an overtwisted family $\SD_{2/3}$ which is formally homotopic to $\SD_{1/3}$ by Corollary \ref{cor:homotopyChange2}. We isotope the structures $\SD_{2/3}$ to yield a family $\SD_1$ such that $\Delta$ is an Engel embedding of the overtwisted disc.
\end{proof}

\subsection{Replication of the certificate}

The following is a corollary of Lemma~\ref{lem:selfReplication}:
\begin{lemma} \label{lem:selfReplication2}
Let $K$ be a compact manifold and $K' \subset K$ a smooth ball. Let $\SD_0: K \lra \Engel(M)$ be a family of Engel structures with certificate $\Delta^0 \subset M \times K$. Then, there is a homotopy of Engel structures $(\SD_s)_{s \in [0,1]}: K \lra \Engel(M)$ supported in a neighborhood of $\Delta^0$ such that
\begin{itemize}
\item $\SD_s$ is overtwisted with certificate $\Delta^0$,
\item there is a $K'$-family of Engel embeddings of the overtwisted disc 
$$
\Delta_k^1: \Delta_\OT \lra \left(M,\SD_1(k)\right), \qquad k \in K'
$$ 
such that $\Delta_k^1(\Delta_\OT)$ is contained in $\Op(\Delta^0) \setminus \Delta^0$.
\end{itemize}
\end{lemma}
\begin{proof}
The desired Engel homotopy will be constant on the complement of $\Op(K')$ in the parameter. Let us assume that $K = \Op(K')$ is a smooth ball. Then $\Delta^0$ is a globally trivial fibration by overtwisted discs $\Delta_k^0: \Delta_\OT \lra (M,\SD_0(k))$, $k \in K$.

We can apply Lemma~\ref{lem:selfReplication} to $\SD_0$ to obtain a homotopy of Engel families $(\widetilde\SD_s)_{s \in [0,1]}$ supported near $\Delta^0$. It yields
\begin{alignat*}{5}
\widetilde{\SD}_s: &\quad K          &\lra \Engel(M),             \qquad& \widetilde{\SD}_0 &= \SD_0, \\
F_{k,s}:           &\quad \Delta_\OT &\lra (M,\widetilde{\SD}_s), \qquad& F_{k,0}           &= \Delta_k^0,
\end{alignat*}
and a certificate $G_k: \Delta_\OT \lra (M,\widetilde{\SD}_1)$ which is disjoint from $F_{k,1}$.

By the isotopy extension theorem applied to $F_{k,s}$ there are isotopies $\psi_{k,s}$ of $M$ such that  $F_{k,s}= \psi_{k,s} \circ \Delta_k^0$. We set $\SD_s(k) = \psi_{k,s}^*\widetilde{\SD}_s(k)$. The certificate $\Delta_k^1$ is  $\psi_{k,1}^{-1} \circ G_k$. We can cut-off the replication as $k$ goes from $\partial K'$ to $\partial \Op(K')$ to ensure that the homotopy is relative to the complement of $\Op(K')$. 
\end{proof}
Lemma~\ref{lem:selfReplication2} yields copies of the certificate in the vicinity of $\Delta^0$. We will use such copies whenever a certain Engel homotopy requires a certificate of overtwistedness. This way, Engel homotopies can be performed relative to the original certificate $\Delta^0$.

\subsection{Homotopies of overtwisted discs} \label{ssec:homotopyOTdisc}

We conclude this section showing how homotopies of the core can be used to construct homotopies of the overtwisted disc. A key observation was explained in Remark~\ref{contractible choice in Engel Lutz}: adding Engel torsion is parametric as the Engel structure and the transverse $3$-manifold vary in families. 

Let $(M,\SD)$ be an overtwisted Engel manifold with $\Delta: \Delta_\OT \lra (M,\SD)$ an Engel embedding of the overtwisted disc. We denote its core by $\gamma: [0,L] \lra M$, the corresponding transverse cylinder by $\Sigma$, and its (constant) profile by $\eta$. Apart from $L$ (the length of $\Delta_\OT$ in the $y$-coordinate) we fix a constant $l$ satisfying $L/2 > l >1/\tau_0$. Here $\tau_0$ is the constant from Lemma~\ref{lem:scalingProfile} measuring the allowed speed of scaling for the profile $\eta$ (which already appeared in the definition of the overtwisted disc). We use the terminology
\begin{align*}
A_0 = \Delta(\{y=0\}) &\textrm{ is the \textsl{lower boundary}, } \\
A_L = \Delta(\{y=L\}) & \textrm{ is the \textsl{upper boundary}. }
\end{align*}
The endpoint $\gamma(0)$ (respectively $\gamma(L)$) of the core of the overtwisted disc lies in $A_0$ (respectively $A_L$).

We cut $M$ along $A_0$ and $A_L$ to produce a manifold with boundary and corners $M^\Delta$. The boundary $\partial M^\Delta$ has two connected components, each of which consists of two copies of $A_0$ (respectively, $A_L$) glued to one another along their common boundary. $M^\Delta$ inherits an Engel structure, which we still denote by $\SD$. Similarly, we write $\gamma$ and $\Delta$ for their lift to $M^\Delta$. Observe that $(M^\Delta, \SD)$ is, by definition, obtained from some other Engel structure $(M^\Delta, \SL^{-1}(\SD))$ by Engel-Lutz twisting along the cylinder $\Sigma$ with core $\gamma$ and profile $\eta$. As such, the structures $\SD$ and $\SL^{-1}(\SD)$ differ from one another only on  $\image(\Delta)$. $\SL^{-1}(\SD)$ is given instead by an Engel embedding of the standard model around $\gamma$:
$$ \varphi: ([0,L] \times \D^3,\SD_\trans) \lra (M^\Delta,\SL^{-1}(\SD)). $$
We will refer to the region $\varphi(\{y \in [0,l] \cup [L-l,L]\}$ as the \emph{scaling region}.

Let $(\gamma_t)_{t \in [0,1]}$ be a homotopy of embedded transverse arcs in $(M^\Delta, \SL^{-1}(\SD))$ with $\gamma_0$ being the core $\gamma$. By Proposition \ref{prop:modelKnot} this homotopy extends to a homotopy of standard neighborhoods
$$ \varphi_t: (U_t,\SD_\trans) \lra (M^\Delta,\SL^{-1}(\SD)) \textrm{ with }\varphi_0 = \varphi. $$
Here $U_t$ is a neighborhood of $[0,L_t]\times\{0\}$ in $[0,L_t]\times\R^3$, $L_t$ is a constant that depends smoothly on $t$, and $L_0 = L$. We will henceforth assume that $L_t>2l$, and we require that the curves $\gamma_t$ agree with $\gamma$ on the scaling region $[0,l]\cup[L_t-l,L_t]$, that is:
\begin{itemize}
\item $\gamma_t(y) = \gamma(y)$ and $\varphi_t(y,\cdot) = \varphi(y,\cdot)$  for all $y \in [0,l]$, and
\item $\gamma_t(L_t-y) = \gamma(L-y)$ and $\varphi_t(L_t-y,\cdot) = \varphi(L-y,\cdot)$ for all $y \in [0,l]$.
\end{itemize}

We consider a homotopy of profiles $\eta_t(y)$ with $t \in [0,1]$ and $y\in [0,L_t]$ such that
\begin{itemize}
\item $\eta_0(y) = \eta$,
\item $\eta_t(y) = \eta$ and $\eta_t(L_t-y) = \eta$ for all $y \in [0,l]$.
\end{itemize}
Recall the $y$-dependent rescaling from Equation~\ref{e:psi lambda}
\begin{align*}
\psi_\lambda : [0,L_t] \times \D^3 & \longrightarrow [0,L_t] \times \D^3 \\
(y,x,z,w) & \longmapsto (y, \lambda(y)^2x, \lambda(y)z, \lambda(y)w).
\end{align*}

The following result is an immediate consequence of Lemma~\ref{lem:scalingProfile}.
\begin{proposition} \label{prop:homotopyOTdisc}
In the situation described above, there is a homotopy of overtwisted Engel structures $(\SD_t)_{t \in [0,1]}$ satisfying:
\begin{itemize}
\item $\SD_0 = \SD$,
\item the overtwisted disc $\Delta_t$ of $\SD_t$ has $\gamma_t$ as core and a rescaling of $(\eta_t(y))_{y \in [0,L_t]}$ as profile.
\end{itemize}
\end{proposition}
\begin{proof}
As shown in Lemma~\ref{lem:smallProfile}, there is a number $\lambda_0\in(0,1]$ such that all the surfaces 
\begin{equation} \label{e:homotopy transverse}
\bigcup_{y\in[0,L_t]} \{y\}\times\psi_{\lambda_0}(\eta_t(y)) \subset [0,L_t] \times \D^3
\end{equation}
are transverse to $\SD_\trans$ and contained in $U_t$ for all $t\in[0,1]$. Our choice of $l$ readily implies that there exists a smooth family of functions $\rho_t : [0,L_t] \longrightarrow(0,1]$ satisfying
\begin{itemize}
\item $\rho_0 \equiv 1$,
\item $|\rho_t'| < \tau_0$, 
\item $\rho_t \equiv 1$ in  $\Op(\{0,L_t\})$, and 
\item $\rho_t \equiv \lambda_0$ on $[l, L_t-l]$ for all $t$ in the complement of $\Op(\{0\})$.
\end{itemize}
It follows from the properties of $\rho_t$ that all the surfaces in the homotopy
\begin{equation} \label{e:boundary surface}
\Sigma_t = \varphi_t\left(\bigcup_{y\in [0,L_t]} \{y\}\times \psi_{\rho_t(y)}\big(\eta_t(y)\big)\right), \qquad t \in [0,1],
\end{equation}
are transverse and, by construction, $\Sigma_0 = \Sigma$. One of these surfaces is shown in Figure~\ref{fig:transverse homotopy}.

\begin{figure}[htb]
\begin{center}
\includegraphics[scale=0.8]{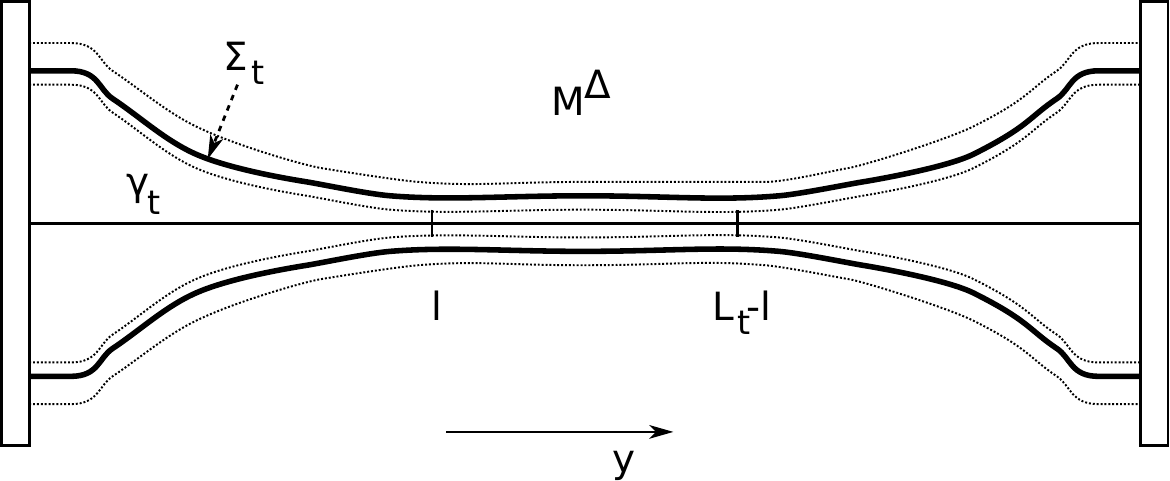}
\caption{A transverse surface $\Sigma_t$ together with a tubular neighborhood where an Engel-Lutz twist is performed.} \label{fig:transverse homotopy}
\end{center}
\end{figure}

We then add an Engel-Lutz twist to $\left(M^\Delta,\SL^{-1}(\SD)\right)$ along the surface $\Sigma_t$. Doing this parametrically in $t$ yields a homotopy of overtwisted Engel structures $(\SD_t)_{t\in[0,1]}$ starting at $\SD$. Since this is a homotopy relative to the lower and upper boundaries of the overtwisted disc, it can be regarded as a homotopy in $M$ instead of $M^\Delta$.
\end{proof}

\section{$h$-principle for overtwisted Engel structures} \label{sec:hPrinciple}

The rest of the article is concerned with the proof of Theorem~\ref{thm:main} and its corollaries. First, we go over the general setup, applying standard methods to reduce the proof to a simplified problem which will be treated in Subsections~\ref{ssec:extension1} and~\ref{ssec:extension}.

\subsection{Setup} \label{ssec:setup}

We fix a $4$-manifold $M$, a smooth compact manifold $K$ of arbitrary dimension, and a $K$-parametric family of formal Engel structures
$$ \SW_k\subset\SD_k\subset \SE_k\subset TM, \qquad k \in K. $$
When we say that this triple is Engel we mean that it is a formal Engel structure arising from a genuine Engel structure. Often we will just write $\SD_k$ for the formal Engel structure, even if it is not genuine, to keep the notation less cluttered. 

The $h$-principle we want to show will be valid for overtwisted Engel structures. As such, we require that:
\begin{itemize}
\item the family $(\SD_k)_{k \in K}$ has a certificate of overtwistedness $\Delta \subset M \times K$.
\end{itemize}
Since the $h$-principle we want to prove is parametric and relative in the parameter and the domain, we also assume that:
\begin{itemize}
\item there is a CW-complex $K' \subset K$ such that $\SD_k$ is Engel for all $k \in K'$,
\item there is a submanifold $U \subset M$ such that $(\SD_k)|_U$ is Engel for all $k$,
\item the manifold $M \setminus U$ is connected,
\item $\Delta$ is contained in $(M \setminus U) \times K$.
\end{itemize}
We denote $M_k = M\times \{k\}$. For ease of notation, we will write $V = (M \times K') \cup (U \times K) \subset M \times K$.

To deal with the parametric nature of the statement it is sometimes more convenient to view $(\SD_k)_{k\in K}$ as a formal Engel structure on the foliation $\SF_{M \times K}$ by fibers of $M \times K \lra K$. We use the notation $\SD_{M\times K}$ for the plane field on $M\times K$ which coincides with $\SD_k$ on each fibre $M_k$. The distributions $\SE_{M\times K}$ and $\SW_{M\times K}$ are defined analogously. We say that $\SD_{M\times K}$ is an Engel structure on an open subset $A$ of $M\times K$ if, for each $k$, $\SD_k$ is an Engel structure on $M_k\cap A$.

In this setting, the $h$-principle for overtwisted Engel structures is
\begin{theorem} \label{thm:mainPrime}
$\SD_{M\times K}$ is homotopic to a fiberwise Engel structure $\widetilde{\SD}_{M\times K}$ through a family of fibrewise formal Engel structures, relative to $V$ and $\Delta$.
\end{theorem}
Note that the parametric and relative nature of the statement yields the analogous result when $K$ is a CW-complex (as stated in Theorem \ref{thm:main}).

In order to make quantitative statements on angles and distances, we fix Riemannian metrics on $M$ and $K$.

\subsection{Reductions using standard $h$-principle methods}  \label{sec:reduction}

In this section we reduce the proof of Theorem~\ref{thm:mainPrime} to an extension problem given in a standardized form. To describe this extension problem, we use the language of model structures (Subsection \ref{ssec:modelStructures}), which we recall briefly.

A model structure $M(I,J,f_-,f_+,c)$ is a formal Engel structure with domain 
$$ 
\{y\in I, x\in [0,1], f_-\le z\le f_+, w \in J\}
$$ 
where $I$ is a $1$-manifold and $J$ is an interval. The functions $f_-$, $f_+$, and $c$ define a formal Engel structure with flag
\begin{align*}
\SW&=\langle \partial_w\rangle \\
\SD_c & = \SW\oplus\left\langle \cos(c)\partial_z+\sin(c)\left(\cos(z)\partial_x+\sin(z)\partial_y\right) \right\rangle \\
\SE & = \SW\oplus\mathrm{ker}(\cos(z)dy-\sin(z)dx).
\end{align*} 
A simple computation shows that $\SD_c$ is an Engel structure with the correct formal class if and only if $\partial_w c>0$ everywhere.
\begin{definition} \label{def:shell}
Let $\varepsilon < 0$ be a small constant and $K_0$ a smooth manifold, possibly with boundary. A family of model structures 
$$ B_k = M(I,[-\varepsilon,1],-\varepsilon,f_{k,+},c_k) $$
with $k\in K_0$ satisfying
\begin{itemize}
\item $\partial_w c_k > 0$ if $k \in \partial K_0$,
\item $\partial_w c_k > 0$ on a $2\varepsilon$-neighborhood of the boundary of the model,
\item $c_k(y,x,z,w) \equiv \arcsin(w)$ if $w \in [-\varepsilon,\varepsilon]$, and
\item $c_k(y,x,z,w) > -\pi$
\end{itemize}
is a \textbf{shell} if $I$ is a closed interval, and a \textbf{circular shell} if $I = \NS^1$. 
\end{definition}

The following proposition is the main result of this subsection.
\begin{proposition} \label{prop:reduction}
There is a $K$-family of formal Engel structures
\[ \SW_{M\times K}' \subset \SD_{M\times K}' \subset \SE_{M\times K}' \]
satisfying the following conditions:
\begin{itemize} \label{(i)-(iv)}
\item[(i)] The family is homotopic to the original formal data $\SW_{M\times K}\subset \SD_{M\times K}\subset \SE_{M\times K}$, relative to $\Delta$ and $V$.
\item[(ii)] $\SE_{M\times K}'$ is a $K$-family of even-contact structures whose family of characteristic line fields is $\SW_{M\times K}'$.
\item[(iii)] $\SD_{M\times K}'$ is an Engel structure $($with associated even-contact structure $\SE_{M\times K}')$ in the complement of a finite collection of shells $\{B_i \subset M \times K \setminus (V \cup \Delta)\}_i$.
\item[(iv)] For every subset of indices $\{i_j\}_j$, the intersection 
$$
\bigcap_j B_{i_j}\big/\SF_{M \times K}
$$ 
is either contractible or empty. 
\end{itemize}
\end{proposition}
Condition (iv)  will be used in Subsection \ref{ssec:extension1}, Proposition \ref{prop:connections exist}, to connect each ball $B_i$ with the overtwisted disc.

The rest of the subsection is dedicated to the proof of Proposition \ref{prop:reduction}. The result will be achieved in several steps, none of which use the overtwistedness of $\SD_{M\times K}$.\footnote{The proof is relatively technical but nonetheless standard. In particular, we introduce notation whose sole purpose is establishing claim (iv). The proof of Theorem~\ref{thm:main} can be understood  treating Proposition \ref{prop:reduction} as a black box.}

\subsubsection{Reduction to balls and genuine even-contact structures}

The Engel condition is open and $\mathrm{Diff}$-invariant as a relation in the space of $2$-jets of plane fields on $M$. Therefore, Gromov's $h$-principle (see \cite[Proposition 7.2.3]{em}) for open, $\mathrm{Diff}$-invariant relations on open manifolds implies that the given formal Engel structure is homotopic to a genuine Engel structure in the complement of a family of balls $(D_k)_{k\in K}$ disjoint from $V$ and $\Delta$. These balls can be assumed to vary smoothly with $k$ (for instance, by placing them in a neighborhood of $\Delta$). This $h$-principle is also parametric and relative, i.e. we do not modify the distributions on $V$ and $\Delta$. 

We may consider the neighborhoods $\Op(\Delta_k)$ that contain $D_k$ as a (smoothly) trivial fibration over $K$ with $\D^4$ fibres. This allows us to assume, for the rest of the proof, that the manifold $M$ is just $\D^4$ (so it is in particular compact) and that the formal Engel structure is a genuine Engel structure near $\partial\D^4$. It follows that all distributions are orientable and coorientable.

According to \cite[Proposition 7.2]{mcd} (see also \cite[Section 10.4]{em}) the formal even-contact structure $(\SE_k,\SW_k)$ is homotopic to a family of honest even-contact structures $(\widetilde{\SE}_k,\widetilde{\SW}_k)$. By this we mean that $\widetilde{\SE}_k$ is a smooth family of even-contact structures whose characteristic foliations are $\widetilde{\SW}_k$. The relative nature of this $h$-principle allows us to assume that this homotopy is constant on $V$ and $\Delta$. 

From now on, we assume that $(\SE_k,\SW_k)$ is an even-contact structure containing $\SD_k$ for all $k\in K$. This even-contact structure will change in the course of the argument. We have achieved condition (ii) from Proposition \ref{prop:reduction}, condition (i) is satisfied by all homotopies introduced so far.

\subsubsection{Adapted triangulations} \label{ssec:triangulation}

A basic ingredient in our constructions is a triangulation $\ST$ of $M\times K$ which is adapted to several distributions on $M\times K$. This means that every simplex $\sigma$ is so small that each distribution is almost constant with respect to the affine coordinates $\sigma\simeq\{x_0+\ldots+x_n=1 \textrm{ and }x_i\ge 0\}\subset\R^{n+1}$. The following definitions are minor adaptations of definitions from Thurston's paper \cite{th}.
  
\begin{definition} \label{d:general position}
Let $N$ be a manifold of dimension $n$ and $\xi$ a smooth distribution of codimension $q$. A top-dimensional simplex $\sigma \subset N$ is in {\bf general position} with respect to $\xi$ if, in the coordinates provided by $\sigma$ and for all points $p\in \sigma$, the linear projection $\sigma\longrightarrow\R^n/\xi_p$ along $\xi_p$ maps each $q$-dimensional subsimplex of $\partial\sigma$ to a non-degenerate simplex of $\R^q\simeq\R^n/\xi_p$.

A triangulation $\ST$ of $N$ is in general position with respect to $\xi$ if every top-dimensional simplex is in general position.
\end{definition}
In particular, being in general position guarantees that all simplices are transverse to $\xi$ in the sense that the intersection of the tangent space of each simplex with $\xi$ has minimal dimension everywhere.

\begin{definition} \label{d:adapted triang}
A top-dimensional simplex $\sigma \subset M \times K$ is {\bf adapted} to $\SW_{M\times K} \subset \SD_{M\times K}\subset\SE_{M\times K}$ if the following conditions are satisfied:
\begin{itemize}
\item[i.] $\sigma$ is in general position with respect to the foliation $\SF_{M \times K}$ by fibres of $M\times K \lra K$,
\item[ii.] $\sigma$ is in general position with respect to the distributions $\SW_{M\times K}$, $\SD_{M\times K}$, and $\SE_{M\times K}$,
\item[iii.] the plane field $\SD_{M\times K}|_\sigma$ describes less than one projective turn with respect to the line field $\SW_{M\times K}|_\sigma$.
\end{itemize}  

A triangulation $\ST$ is adapted, if:
\begin{itemize}
\item it is a triangulation of the pair $(M\times K, \Op(V \cup \Delta))$, where $\Op(V \cup \Delta)$ is an Engel neighborhood of $V \cup \Delta$,
\item every top simplex $\sigma \in \ST$ is adapted.
\end{itemize} 
\end{definition}
By a triangulation of $(M\times K, \Op(V \cup \Delta))$ we mean that there is a triangulation $\ST_\partial$ of $\partial(\Op(V \cup \Delta))$ which $\ST$ extends to the rest of $(M\times K) \setminus \Op(V \cup \Delta)$.

\begin{remark}
Condition (iii) for an adapted simplex should be understood in terms of the development map: We defined it for Engel structures in Equation \eqref{e:development map} p.~\pageref{e:development map}, and this definition extends to the present formal setting. The only difference with the Engel case is that now the development maps of $\SD_{M\times K}|_\sigma$ are not immersions in general (as that would imply that $\SD_{M\times K}|_\sigma$ is Engel). Condition (iii) is then equivalent to the assumption that the (oriented) development maps do not take antipodal values. This does not follow from Condition (ii), since the leaves of the kernel foliation may still twist around each other, as indicated in Figure~\ref{b:sufficiently fine}.
\end{remark}
 
\begin{figure}
\begin{center}
\includegraphics[scale=1.5]{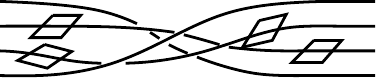} 
\caption{A plane field tangent to the line field shown in the image and also containing some fixed coordinate direction is $C^0$-close to a foliation and, as such, may be in general position. However, its development map might perform a full turn.}
\label{b:sufficiently fine}
\end{center}
\end{figure}

Let us introduce some notation:  We will write $\SW_\sigma\subset\SD_\sigma\subset\SE_\sigma$ for the restriction of the formal Engel structure to a top-dimensional simplex $\sigma$. If $\sigma$ is adapted, the subsimplices of $\sigma$ are transverse to $\SW_\sigma$, $\SD_\sigma$, and $\SE_\sigma$. In particular, a codimension-$1$ simplex is never tangent to $\SW_\sigma$, intersects $\SD_\sigma$ in a line field, and $\SE_\sigma$ in a plane field. 

Fix auxiliary orientations of $\SW_\sigma$ and $\SD_\sigma$. Then, we write $\sigma_-$ for the union of those faces where $\SW_\sigma$ points into $\sigma$. Similarly, the line field $T\sigma_- \cap \SD_\sigma$ divides the boundary of $\sigma_-$ into the regions $\partial_-\sigma_-$ and $\partial_+\sigma_-$, depending on whether it points inwards or outwards, respectively. Both $\partial_\pm \sigma_-$ are codimension-$2$ complexes and homeomorphic to closed balls.

Given a subset of $M \times K$, in this case $\sigma$, we write $K_\sigma$ for the subset of those $k\in K$ such that $M_k$ meets $\sigma$. For $k\in K_\sigma$ let $\sigma_k=\sigma\cap M_k$ and $\sigma_{k,-}=\sigma_-\cap M_k$. When $\sigma$ is adapted, $\SE_\sigma$ induces a contact structure $\xi_k$ on $\sigma_{k,-}$. This contact structure contains the line field $\SH_k = T\sigma_{k,-} \cap \SD_\sigma$.

Figure~\ref{b:sigma notation} illustrates the various parts of $\sigma$ in the non-parametric case: $\sigma$ is shown as a $3$-dimensional simplex and $\SW_\sigma$ corresponds to the vertical direction. The line field in the bottom face represents the imprint $\SH_k$ of $\SD_\sigma$ on $\sigma_-$.

\begin{figure}[htb] 
\begin{center}
\includegraphics[scale=1.0]{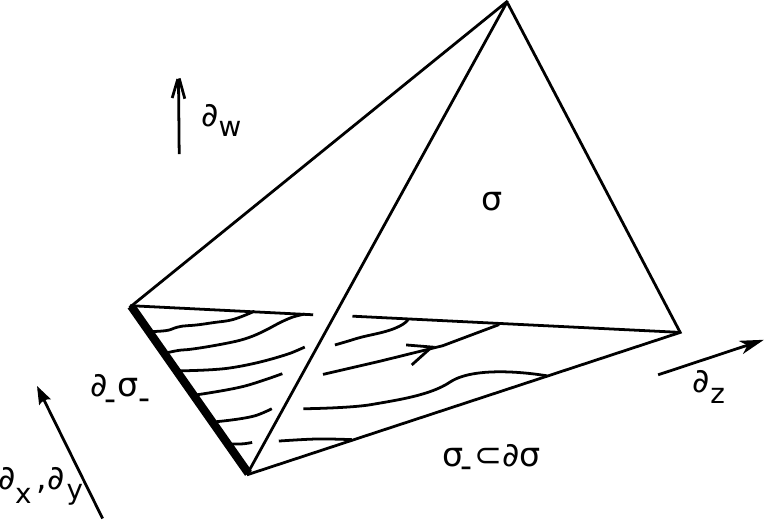}
\caption{Various parts of $\sigma$. The thickened line represents $\partial_-\sigma_-$. The coordinate directions shown are provided, roughly, by the identification of $\sigma$ with a model structure (Section~\ref{ssec:modelStructures}); this will be proven in Subsection~\ref{ssec:shells}.}
\label{b:sigma notation}
\end{center}
\end{figure}
Some immediate properties of adapted simplices are summarized in the next lemma:
\begin{lemma} \label{lem:simplexToShell}
Let $\sigma$ be an adapted simplex. Then:
\begin{itemize}
\item the set $K_\sigma$ is homeomorphic to a closed ball of dimension $\dim(K)$,
\item for every $k \in \partial(K_\sigma)$, $\sigma_k$ is a point,
\item for every $k \in \int(K_\sigma)$, $\sigma_k$ is a polyhedron homeomorphic to a closed $4$-ball, 
\item for every $k \in \int(K_\sigma)$, the union of the codimension-$1$ faces of $\sigma_k$ which are positively transverse to $\SW_{M\times K}$ is homeomorphic to a closed $3$-ball. The same is true for the negatively transverse faces.
\end{itemize}
\end{lemma}

\subsubsection{Sequences of triangulations and coverings}

Conditions (i), (ii), and (iii) in Proposition \ref{prop:reduction} can be achieved by modifying the development map of $\SD_{M \times K}$ in the vicinity of the codimension-$1$ skeleton of an adapted triangulation. This is the content of Lemma \ref{lem:reduction}. It yields a structure $\SD_{M \times K}'$ that is Engel in the complement of finitely many balls, each of which is obtained by slightly shrinking one of the top-dimensional simplices of the triangulation. However, condition (iv) does not follow immediately from this argument. To achieve it, we will need to consider an infinite sequence of triangulations, each of which is a subdivision of the previous one. For a sufficiently fine subdivision condition (iv) is satisfied. This will be shown in the upcoming lemmas. 

The first auxiliary lemma constructs the triangulations.
\begin{lemma} \label{lem:adapted triang}
There is
\begin{itemize}
\item a finite cover $\{U_i\}$ of $(M \times K) \setminus \Op(V \cup \Delta)$ by balls that are simultaneously flowboxes of $\SW_{M \times K}$ and foliation charts of $\SF_{M \times K}$,
\item an infinite sequence $\{\ST_j\}_{j=0}^\infty$ of adapted triangulations, and
\item a finite list of model top-dimensional simplices $\{\tau_k \subset \R^{\dim(M \times K)}\}$,
\end{itemize}
with the following properties.
\begin{itemize}
\item $\ST_{j+1}$ is a subdivision of $\ST_j$.
\item Every top-dimensional simplex $\sigma$ of a triangulation $\ST_j$ is adapted and contained in some $U_i$. In the coordinates provided by $U_i$, $\sigma$ agrees with one of the model simplices $\{\tau_k\}$ up to rescaling and translation.
\end{itemize} 
\end{lemma}
\begin{proof}
There is a finite covering of $M \times K$ by foliation charts of $\SF_{M \times K} = \cup_{k \in K} M_k$ that are also flowboxes of $\SW_{M \times K}$; this follows from the compactness of $M \times K$. We fix a triangulation that is subordinate to this covering. We apply to it Thurston's argument \cite[Section 5]{th} of perturbing and carefully subdividing using a crystalline subdivision scheme. This yields only finitely many shapes of simplices up to rescaling for all subdivisions; note that this is not the case for the barycentric subdivision. Such a scheme is described in detail in \cite[p.~358]{whitney}. This method can be applied simultaneously to all the distributions involved. Let us provide some additional details. 

There are only finitely many shapes of simplices in the following sense: Once a first adapted triangulation $\ST_0$ has been found, each top-dimensional simplex $\sigma \in \ST_0$ is in general position with respect to $\SW_{M \times K}$ and $\SF_{M \times K}$. This implies that
\begin{itemize}
\item $\sigma$ can be identified with the standard simplex $\sigma'$ in $\R^{\dim(M \times K)}$,
\item under this identification $\SF_{M \times K}$ is mapped to a linear foliation which is transverse to all the simplices in the standard triangulation of the unit cube,
\item $\SW_{M \times K}$ is mapped to a linear line field which is transverse to all the simplices in the standard triangulation of the unit cube.
\end{itemize}
This identification can be extended to a diffeomorphism between a neighborhood $U \supset \sigma$ and a neighborhood of $\sigma'$. By construction, $U$ can be assumed to be a foliation chart of $\SF_{M \times K}$ and a flowbox of $\SW_{M \times K}$. Doing this for each top-dimensional simplex $\sigma_i \in \ST_0$ we find the covering of $M \times K$ by flowboxes/foliation charts $\{U_i\}$. Because $M\times K$ is compact, there are only finitely many of them.

Every subsequent triangulation $\ST_j$ is obtained from $\ST_0$ by subdividing each top-simplex $\sigma_i \in \ST_0$ in a standard manner in the coordinates of $\sigma_i$ (which are the same as the coordinates of $U_i$). This yields finitely many shapes independently of $j$ after sufficiently many subdivisions have been performed.
\end{proof}

For condition (iv) in Proposition \ref{prop:reduction} to hold, we need to achieve the Engel condition in a neighborhood of uniform size of the codimension--$1$ skeleton (that is, once the top-simplices are identified with the model simplices, the size of the neighborhood should not depend on $j$). For this purpose, we construct coverings adapted to the triangulations we are working with:
\begin{definition} \label{def:associatedCovering}
Let $\ST$ be an adapted triangulation, which we regard as a collection of simplices $\{\sigma\}$. A covering $\{\SU(\sigma)\}_{\sigma \in \ST}$ of $(M \times K) \setminus \Op(V \cup \Delta)$ is \textbf{associated} to $\ST$ if:
\begin{itemize}
\item[i.] Every simplex $\sigma$ is contained in the union $\cup_{\sigma' \subset \sigma} \SU(\sigma')$, where the union ranges over all the subsimplices of $\sigma$ (including $\sigma$ itself). 
\item[ii.] Each $\SU(\sigma)$ is both a flowbox of $\SW_{M \times K}$ and foliation chart of $\SF_{M \times K}$. In particular, the leaves of $\SW_{M \times K}|_{\SU(\sigma)}$ are connected non-degenerate intervals. The boundary of $\SU(\sigma)$ decomposes into its horizontal part $\partial^h\SU(\sigma)$, which is transverse to $\SW_{M \times K}$, and its vertical part $\partial^v\SU(\sigma)$, which is tangent to $\SW_{M \times K}$.
\item[iii.] $\SU(\sigma)$ and $\SU(\sigma')$ only intersect if one is a subsimplex of the other. Moreover, if  $\sigma' \subset \sigma$, then $\SU(\sigma)$ intersects the boundary of $\SU(\sigma')$ in its vertical part. 
\end{itemize}
\end{definition}
These properties will allow us to inductively achieve the Engel condition on $\SU(\sigma)$ for all simplices  $\sigma$ of positive codimension (see Lemma \ref{lem:reduction}).

In the following lemma we construct neighborhoods of the codimension-$1$ skeleton with uniform size.
\begin{lemma} \label{lem:uniformSize}
There are:
\begin{itemize}
\item coverings $\{\SU(\sigma)\}_{\sigma \in \ST_j}$ of $(M \times K) \setminus (V \cup \Delta)$ associated to each triangulation $\ST_j$, and
\item simplices $\{\tau_k' \subset \tau_k\}$ obtained from the model simplices $\{\tau_k\}$ by scaling with respect to the barycenter using a constant smaller than $1$,
\end{itemize}
such that, after identifying the top-dimensional simplex $\sigma \in \ST_j$ with $\tau_k$, 
$$
\partial\tau_k' \subset \bigcup_{\sigma' \subsetneq \sigma} \SU(\sigma').
$$
\end{lemma}

\begin{proof}
A similar result was proven in \cite[Proposition 29]{cppp}. We now give a sketch of the proof.

Every simplex $\sigma \in \ST_j$ arises from the subdivision of some top-dimensional simplex $\sigma_i \in \ST_0$ and is therefore contained in the $\SW_{M \times K}$-flowbox $U_i$. We proceed by induction on their dimension $\dim(\sigma) < \dim(M \times K)$.

Let us start with the $0$-simplices: Given $\sigma_i \in \ST_0$, we fix neighborhoods $\SU(\sigma) \supset \sigma$ for every $j$ and every vertex $\sigma \in \ST_j$ obtained in the subdivision of $\sigma_i$. We require $\SU(\sigma)$ to be a foliation chart and a flowbox as in the statement. If we fix the length of $\SU(\sigma)$ along $\SW_{M \times K}$ and we shrink the other directions, every $1$-simplex incident to $\sigma$ will enter through the vertical boundary $\partial^v\SU(\sigma)$. This follows from the fact that every face is transverse to $\SW_{M \times K}$. 

If $\dim(\sigma) > 0$, we first shrink $\sigma$ to a slightly smaller simplex $\sigma'$ whose boundary is still contained in the union $\cup_{\tau \subsetneq \sigma} \SU(\tau)$. Then, we thicken $\sigma'$ to $\SU(\sigma)$ in such a way that the thickening along $\SW_{M \times K}$ is much greater than in the complementary directions. It follows that the simplices containing $\sigma$ will enter $\SU(\sigma)$ through the vertical boundary $\partial^v\SU(\sigma)$. 

All sets we construct can be assumed not to depend on the simplex itself but on the model simplex they are identified with. In particular, once we have dealt with all the simplices in the codimension-$1$ skeleton, we can choose a constant (which is smaller than but sufficiently close to $1$, and does not depend on $j$) to scale $\tau_k$ and obtain the desired $\tau_k'$. 
\end{proof}

\subsubsection{Achieving the Engel condition in the codimension-$1$ skeleton.} \label{ssec:reduction}

Given any adapted triangulation $\ST$ with associated cover $\{\SU(\sigma)\}_{\sigma \in \ST}$, we can deform $\SD_{M\times K}$ into a family of formal Engel structures $\SD_{M\times K}'$ which is Engel over the neighborhood 
$$
\bigcup_{\{\sigma \in \ST\,|\, \dim(\sigma) < \dim(M \times K)\} } \SU(\sigma)
$$
of the codimension-$1$ skeleton of $\ST$.  
\begin{lemma} \label{lem:reduction}
Let $\SW_{M\times K}\subset \SD_{M\times K} \subset \SE_{M\times K}$ be a family of formal Engel structures such that condition (ii) of Proposition \ref{prop:reduction} holds. Let $\ST$ be an adapted triangulation and $\{\SU(\sigma)\}_{\sigma \in \ST}$ an associated covering.

Then, there is a plane field $\SD'_{M\times K}$ such that
\begin{itemize}
\item $\SD_{M\times K}'$ is homotopic to $\SD_{M \times K}$ through plane fields contained in $\SE_{M\times K}$ and containing $\SW_{M\times K}$,
\item $\SD_{M\times K}'=\SD_{M\times K}$ on $V$ and $\Delta$,
\item  $\SD_{M\times K}'$ is Engel in $\SU(\sigma)$ for every simplex $\sigma \in \ST$ with $\dim(\sigma) < \dim(M \times K)$.
\end{itemize}
\end{lemma}

\begin{proof}
Recall that there is a subcomplex $\ST_\partial \subset \ST$ which triangulates the boundary $\partial(\Op(V \cup \Delta))$. The structure $\SD_{M \times K}$ is already Engel along $\ST_\partial$. The proof is by induction on the dimension of the simplices of $\ST \setminus \ST_\partial$. The relative nature of our claim follows by not modifying $\SD_{M \times K}$ along $\ST_\partial$. 

Let $\sigma \in \ST$ be a vertex. We fix a trivialization $W$ of $\SW_{M\times K}$ and an oriented $W$-invariant framing $\{W,X,Y\}$ of $\SE_{M\times K}$ on a neighborhood $\Op(\SU(\sigma))$ such that $\SD_{M\times K}$ is spanned by $\{W,\cos(c)X+\sin(c)Y\}$, where $c: \Op(\SU(\sigma)) \lra \R$ is some function. We then homotope $c$, relative to the boundary of $\Op(\SU(\sigma))$, to yield a function $\tilde{c}: \Op(\SU(\sigma)) \lra \R$ with the property that $\SL_W\tilde{c}>0$ in $\SU(\sigma)$. This implies that the corresponding plane field is an Engel structure on $\SU(\sigma)$. This can be iterated over all the $0$-simplices. The resulting plane field is again denoted by $\SD_{M\times K}$. 

Because $\ST$ is adapted, a $1$-simplex is nowhere tangent to $\SW_{M\times K}$. Hence, we can choose a $W$-invariant framing $\{W,X,Y\}$ of $\SE_{M\times K}$ on a neighborhood $\Op(\SU(\sigma))$ of the simplex. Then $\SD_{M\times K}$ is spanned by $\{W,\cos(c)X+\sin(c)Y\}$, where $c$ is a smooth function. For every vertex $\sigma' \subset \sigma$ it holds that $\SL_W c>0$ on $\SU(\sigma) \cap \SU(\sigma')$. Condition (iii) of Definition \ref{def:associatedCovering} implies that each leaf of $\SW_{M \times K}|_{\SU(\sigma)}$ is either contained or   disjoint from $\cup_{\sigma' \subsetneq \sigma} \SU(\sigma')$. Therefore, we can replace $c$ by a function $\tilde{c}$ which coincides with $c$ on $\partial \Op(\SU(\sigma))$ and on $\SU(\sigma')$ and satisfies $\SL_W\tilde{c}>0$. We thus obtain a plane field which is Engel on $\SU(\sigma)$. 

This procedure can be iterated until we have dealt with all simplices of codimension at least $1$, always relative to the neighborhoods of simplices of lower dimension. In the regions where $\SL_W c \leq 0$, we can choose $\SL_W\tilde{c}$ to be positive but arbitrarily close to zero. This ensures that $\ST$ is still adapted to the resulting  formal Engel structure $\SD_{M \times K}'$.
\end{proof}

\subsubsection{Achieving the contractibility hypothesis} \label{ssec:contractibility}

We now explain how to achieve claim (iv) from Proposition \ref{prop:reduction}.

Recall that each top-dimensional simplex $\sigma$ in $\ST_j$ is contained in a flowbox/foliation chart $U_{i_0}$. Under the identification of $U_{i_0}$ with the standard model, $\sigma$ is mapped to (a scaled and translated copy of) some model simplex $\tau_{l_0}$ (from a finite list $\{\tau_l\}$). 

Apply Lemma \ref{lem:reduction} to each adapted triangulation $\ST_j$ (produced by Lemma \ref{lem:adapted triang}) with associated covering $\{\SU(\sigma)\}_{\sigma \in \ST_j}$ (constructed in Lemma \ref{lem:uniformSize}). Under the identification of $\sigma$ with $\tau_{l_0}$, the neighborhood $\cup_{\sigma' \subsetneq \sigma} \SU(\sigma')$ of $\partial \sigma \cong \partial\tau_{l_0}$  covers some rescaling $\tau_{l_0}'$ with fixed scaling factor. Fix some strictly convex neighborhood $\tau_{l_0}' \subset \widetilde{\tau}_{l_0} \subset \tau_{l_0}$ and write $\widetilde\sigma$ for its image in $\sigma$. 

Consider the simplex $\sigma$ in the coordinates provided by the foliation charts $\{U_i\}$. In terms of $U_{i_0}$, $\sigma$ is a genuine linear simplex. However, this might not the case for the other charts. Still, as finer subdivisions are considered (i.e. $j$ goes to infinity) $\sigma$ converges (in the $C^\infty$ topology and after rescaling to have fixed diameter) to an affine simplex in all $U_i$. This follows from the fact that the change of coordinates between $U_i$ and $U_{i'}$ converges to linear map when we rescale it in progressively smaller neighborhoods of a point.

This argument implies that the set $\widetilde\sigma$, which was convex in the coordinates of $U_{i_0}$, is also convex in terms of all the other charts $U_i$ if $j$ is large enough. The same is true for their projections to the leaf space $K$. That is, the set $\widetilde\sigma/\SF_{M \times K}$ is strictly convex in the coordinates provided by each $U_i/\SF_{M \times K}$ for $j$ large enough. By compactness, this can be achieved for all $\sigma \in \ST_j$ simultaneously. This implies that any finite intersection of them is either empty or contractible, as desired.

We fix a sufficiently large number $j$, and henceforth we just write $\ST = \ST_j$.

\subsubsection{Construction of shells} \label{ssec:shells}

The last step in the proof is to produce a shell $B = (B_k)_{k \in K_{\widetilde\sigma}}$ from each $\widetilde\sigma$, where $\sigma$ is a top-dimensional simplex of $\ST$. We consider a rescaling $\sigma' \supset \widetilde\sigma$ of $\sigma$ with respect to the barycenter by a factor smaller than but very close to $1$; $B_k$ will be a smoothing of $\sigma_k'$. The coordinates $(y,x,z,w)$ provided in the simplex by $B_k$ are shown roughly in Figure~\ref{b:sigma notation}, p.~\pageref{b:sigma notation}.

Fix a vector field $W$ spanning $\SW_\sigma$. Since $\sigma'$ is adapted, $W$ is transverse to the faces of each $\sigma_k'$. We write $\sigma'_-$ for the union of those faces where $W$ points into $\sigma'$, and we restrict our attention to the subset $(\sigma_{k,-}'=M_k\cap \sigma_-')_{k \in K_{\widetilde\sigma}}$. Each face of $\sigma_{k,-}'$ is either positively or negatively transverse to the line field $\SD_\sigma \cap\sigma_-'$ and the union of the positive (resp. negative) faces is homeomorphic to a $2$-ball. It follows that $\sigma_{k,-}'$ can be smoothed, yielding a hypersurface $N_k$ with the following properties:
\begin{enumerate}
\item $N_k$ is a smooth hypersurface in $M_k$, transverse to $\SW_k$, and contained in a small neighborhood of $\sigma'_{k,-}$. In particular, $N=\cup_{k\in K_{\widetilde\sigma}} N_k$ is a smooth hypersurface in $M\times K$ transverse to $\SW_\sigma$.  
\item The holonomy of $\SW_\sigma$ yields an embedding  $\pr: \sigma_{k,-}' \lra N_k$ such that the image is disjoint from $\partial N_k$.
\item The boundary $\partial N_k$ is a smooth $2$-sphere. It decomposes into two discs, depending on whether $TN_k\cap\SD_k$ is inwards or outwards pointing. The positive disc $\Gamma_k$ is a smoothing of $\partial_-\sigma_{k,-}'$.  
\end{enumerate}

Since $\sigma'$ is adapted, we can assume that $\partial \Gamma_k$ consists of two intervals: the characteristic foliation $\Gamma_k\cap \SE_k$ is transverse to $\partial\Gamma_k$ on the interior of both intervals and its holonomy identifies the two. Hence, we can equip $\Gamma_k$ with coordinates $(x,y)$ such that 
\begin{equation} \label{e:f-<0<f+}
T\Gamma_k\cap\SE_k=\left\langle \partial_x\right\rangle. 
\end{equation}

Using the flow of a vector field spanning $TN_k\cap\SD_k$ we obtain a family of contact embeddings
\begin{align}\label{e:sign1} 
\psi_k: (N_k,\SE_k \cap TN_k) \lra (\R^3,\ker(\cos(z)dy \pm \sin(z)dx))
\end{align}
mapping $\Gamma_k$ to a smoothing of $[0,1] \times [0,1] \subset \{z=0\}$ (because of \eqref{e:f-<0<f+}) and taking the leaves of $TN_k\cap\SD_k$ to lines parallel to the $z$-axis. 

To obtain the correct sign for the contact structure in Equation~\ref{e:sign1}, we might need to apply the reflection along the $\{x=0\}$ hyperplane. Therefore, we assume that we have a contact embedding 
$$ \psi_k: (N_k,\SE\cap TN_k) \lra (\R^3,\ker(\cos(z)dy - \sin(z)dx)). $$
Since $N_k$ is a compact manifold with boundary,  can extend these embeddings slightly so that their image is 
$$ \{y\in [0,1], x\in[0,1], z\in [f_{-,k},f_{+,k}]\} $$
for some smooth functions $f_{-,k}<0<f_{+,k}$. We may choose $f_{-,k}\equiv-\varepsilon$ with $\varepsilon>0$ small enough. 

The embedding $\psi_k$ extends to an embedding $\Psi_k: B_k = N_k \times [-\varepsilon,1] \lra \R^3 \times \R$ such that 
\begin{itemize}
\item $N_k$ is mapped to $\{w=0\}$,
\item each leaf of $\SW_\sigma$ intersecting $N_k$ is mapped to the interval $\{\psi_k(p)\} \times [-\varepsilon,1]$ for $p \in N_k$.
\end{itemize}
Since the development map of $\SD_\sigma$ does not describe a projective turn (condition (iii) in Definition \ref{d:adapted triang}), the extended embedding $\Psi_k$ determines an angular function $c_k$ vanishing along $\{w=0\}$ such that its values lie in $(-\pi,\pi)$. 

It might happen that $\partial_w c_k < 0$ along the boundary of the model we have constructed. To fix this, one applies the contact transformation $(x,y,z)\mapsto (-x,y,-z)$. After applying this transformation, $\partial_w c_k > 0$ along the boundary of the model. In particular, we can reparametrize $\SW_\sigma$ so that $c_k\equiv \arcsin(w)$ if $w \in [-\varepsilon,\varepsilon]$ (by possibly reducing $\varepsilon$ and the domain of $\Psi_k$). 

This construction goes through for every top-dimensional simplex $\sigma$ of $\ST$, yielding a shell $(B_k)_{k \in \widetilde\sigma}$. The proof of Proposition \ref{prop:reduction} is complete. \hfill $\Box$

\subsection{Extension in the non-parametric case}\label{ssec:nonParametric}

The $\pi_0$-version of Corollary~\ref{cor:extension} states that any Engel germ $(\Op(\partial\D^4),\SD)$ with formal extension to the interior can also be extended to the interior as an Engel structure. With the tools introduced so far we can give a streamlined proof of this fact without invoking Theorem~\ref{thm:main}. The proof is a simplified version of the proof of our main theorem and showcases an important part of the argument.

The first ingredient is the following corollary of Proposition~\ref{prop:reduction}.
\begin{corollary} 
Let $(M,\SD)$ be a formal Engel manifold with $\SD$ Engel over some closed region $U \subset M$. Assume that $M \setminus U$ is connected and that there exists a formal Engel embedding of the overtwisted disc $\Delta: \Delta_\OT \lra (M \setminus U,\SD)$.

Then, $\SD$ is homotopic to a formal Engel structure $\SD'$ satisfying:
\begin{itemize}
\item $\SD'=\SD$ on $U$ and $\Delta$,
\item $\SD'$ is Engel in the complement of a finite collection of circular shells.
\end{itemize}
\end{corollary}
\begin{proof}
It can be proven directly mimicking the last step (Subsection \ref{ssec:shells}) in the proof of Proposition~\ref{prop:reduction}. 

Alternatively, we apply Proposition~\ref{prop:reduction} to reduce to the case where the germ $(\D^4,\SD)$ is described by a shell $B$. We find a slightly smaller shell $B' \subset B$. Then we connect its boundary components $\{y=0\}$ and $\{y=1\}$ using a curve $\gamma \subset B$ transverse to $[\SD,\SD]$ which is otherwise disjoint from $B'$. We can extend the parametrisation of $B'$ to a neighborhood of $\gamma$, yielding the claim. We leave the details to the reader.
\end{proof}

Then, the $\pi_0$-version of Corollary~\ref{cor:extension} follows from an application of:
\begin{proposition} \label{prop:nonParametric}
Any circular shell is formally homotopic, relative to the boundary, to a genuine Engel structure.
\end{proposition}
\begin{proof}
Let $B = M(\NS^1,[-\varepsilon,1],-\varepsilon,f_+,c)$ be the circular shell. We first construct a transverse surface $\SS \subset B$ as in Section~\ref{sec:surfaces}. The core is parametrised by $y \mapsto (y,x=1/2,z=0,w=0)$ and the profile $\eta$ (a family of transverse loops in $\NS^1 \equiv [0,1]/_\sim \lra (\R^3,\ker(dx-wdz))$ parametrised by $y$) is independent of $y$. Moreover, we require 
\begin{itemize}
\item $\eta(\theta) = (x=\theta,z=0,w=0)$, for $\theta \in [\varepsilon,1-\varepsilon]$,
\item $z \circ \eta(\theta) < 0$, for $\theta \notin [\varepsilon,1-\varepsilon]$, and
\item $|w \circ \eta(\theta)| < \varepsilon$.
\end{itemize}
The last requirement implies the following condition on the slope: $\left|\frac{dx(\dot{\eta})}{dz(\dot{\eta})}\right| < \varepsilon$. Therefore, it is easy to construct a curve $\eta$ with the desired properties if we introduce sufficiently many stabilizations, c.f. Figure~\ref{fig:manyEights}. By Lemma~\ref{lem:constantProfile}, $\SS$ is a torus transverse to $\SD_c$. 

\begin{figure}[ht]
\begin{center}
\includegraphics[scale=0.8]{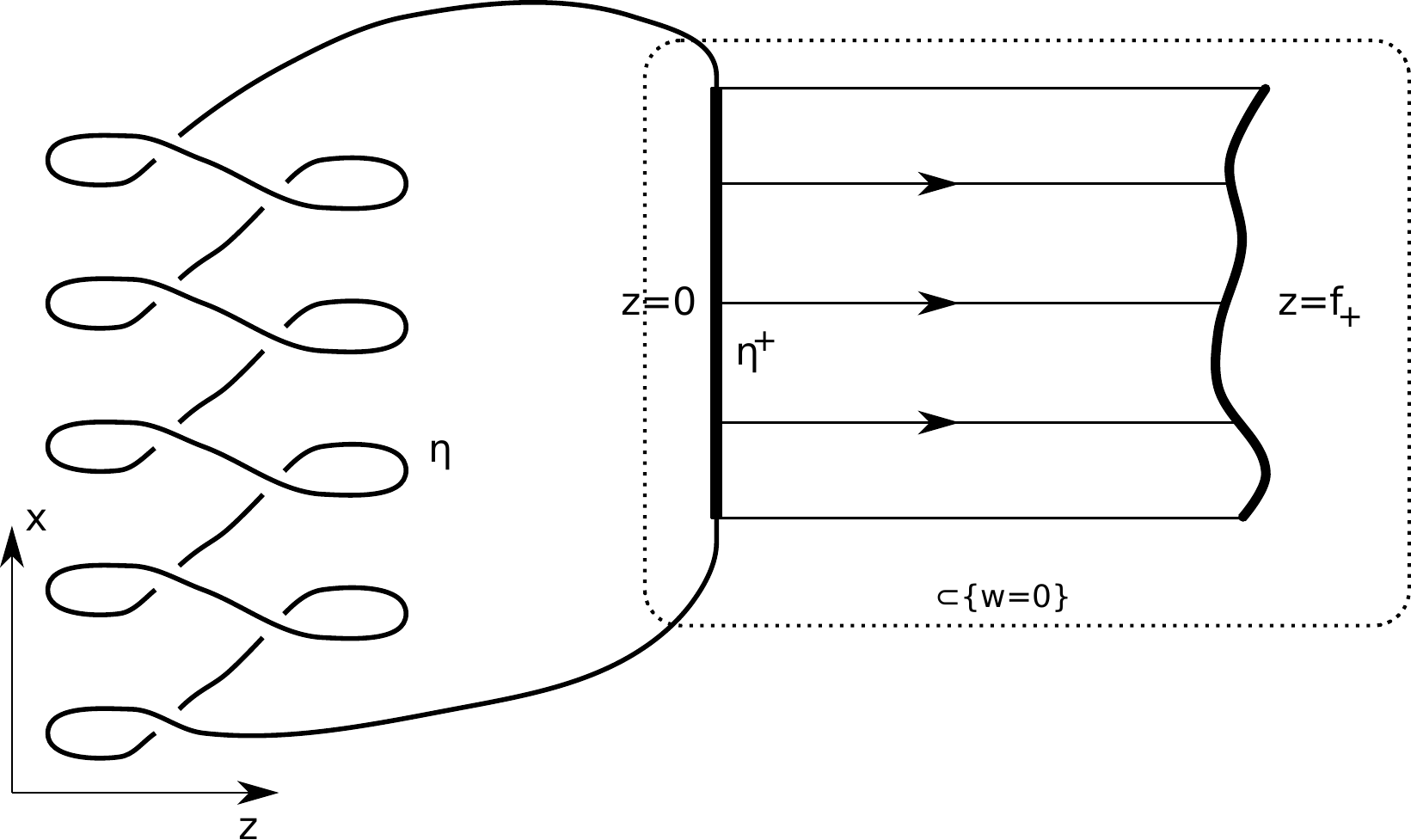}
\caption{Front projection of the transverse curve $\eta$.}
\label{fig:manyEights}
\end{center}
\end{figure}

Consider the rank $1$-foliation $\SH= \cup_{w_0} (\SD_c \cap T\{w=w_0\})$. In the region $\{|w_0|\leq \varepsilon\}$ we may orient it using the vector field $\partial_z + w_0(\partial_x + z\partial_y)$. We pick a constant $\rho>0$ so that $U_\rho(\SS)=\{q\in B\,|\,\mathrm{dist}(p,\SS)\le\rho\}$ is a thin tubular neighborhood satisfying:
\begin{itemize}
\item $U_\rho(\SS)$ is contained in $\{|w|\le\varepsilon\}$,
\item $N=\partial U_\rho(\SS)$ is transverse to $\SD_c$, and
\item the set of tangencies $N_0 = \SW \cap TN$ is a disjoint union of two tori. Each torus is obtained from $\SS$ by pushing along the line field $\SH$ and is, in particular, transverse to $\SD_c$. 
\end{itemize}

Let $\eta^+$ be the part of the profile $\eta$ where $\eta(\theta) = (x=\theta,z=0,w=0)$ and 
$$
\SS^+=\bigcup_{y\in[\varepsilon,1-\varepsilon]}\eta^+ \subset \SS.
$$
We denote  the region within $N_0$ obtained from $\SS^+$ by flowing along $\SH$ positively by $N_0^+$. According to Lemma~\ref{lem:thickenTransverse3fold}, one can isotope a transverse hypersurface using the flow of a vector field that is both tangent to the Engel structure and transverse to the hypersurface. In the present situation, we use the leaves of $\SH$ to isotope $N$ (note that $\SH$ is transverse to $N$ at least in the vicinity of $N_0$).

We want the isotopy to push a small neighborhood of $N_0^+ \subset N$ positively while leaving the rest of the $3$-manifold fixed. We denote the flow of $\partial_z + w_0(\partial_x + z\partial_y)$ by $\varphi$. We then choose a function $\chi : N \lra \R$ such that $\chi|_{N_0^+} = f_+ - \varepsilon$, and $\chi \equiv 0$ in the complement of $\Op(N_0^+)$, allowing us to define:
$$ N' =  \left\{\varphi_{\chi(p)}(p) \,\left.\right|\, p\in N \right\}. $$
This is a transverse $3$-torus which is isotopic to $N$ through transverse hypersurfaces.

The restriction of $\SD_c$ to 
\[ M' = \{2\varepsilon \leq x \leq 1-2\varepsilon, 0 \leq z \leq f_+ - 2\varepsilon, w \geq \varepsilon\} \]
is a circular shell. The hypersurface $N'$ is disjoint from $M'$ and, by construction, its projection $\pi(N')$ to the $(y,x,z)$-hyperplane contains $\pi(M')$. Now we perform an Engel-Lutz twist along $N'$. This takes place away from $M'$, so the resulting formal Engel structure $\SL(\SD_c)$ is still a circular shell in $M'$. Let $c': M' \lra \R$ be the corresponding angular function. A priori, $c'$ agrees with $c$, but the presence of the Engel-Lutz twist allows us lower its value by at most $2\pi$, effectively homotoping $\SL(\SD_c)$ in the band $\{ w \in  (0,\varepsilon]\}$. Doing this, we can set $c'(y,x,z,\varepsilon) < c(y,x,z,1) = c'(y,x,z,1)$ everywhere. 

Lemma~\ref{lem:Bolzano} shows that $(M',\SD_{c'})$ is homotopic to a solid model relative to its boundary. This implies that $\SL(\SD_c)$ is homotopic to a genuine Engel structure relative to the boundary of the original circular shell.
\end{proof}
Observe that the proof we have just presented is clearly parametric (although we avoided this for simplicity). However, we do not know whether it is relative in the parameter: \emph{is there an Engel homotopy that adds Engel torsion?} An affirmative answer to this question would prove that all Engel structures are overtwisted up to homotopy (proving the complete $h$-principle for Engel structures).

\subsection{Setup for the extension}\label{ssec:extension1}

The rest of the paper is dedicated to completing the proof of Theorem~\ref{thm:main} (equivalently, Theorem \ref{thm:mainPrime}). After the reduction argument from Proposition~\ref{prop:reduction}, we may assume that $\SD_{M \times K}$ is Engel in the complement of a finite collection of pairwise disjoint shells $\{B^i = (B^i_k)_{k \in K_i \subset K}\}_i$. We have to construct a homotopy, relative to $V$ and $\Delta$, between $\SD_{M \times K}$ and a $K$-family of Engel structures.

The construction can be carried out independently for each shell, so let us focus on the particular shell $B = B_k^0 = (B_k = B_k^0)_{k \in K_0}$. The construction has two main steps.
\begin{enumerate}
\item First, we homotope $\SD_{M\times K}$ inside the region where $\SD_{M\times K}$ is already Engel. Doing so, we make overtwisted discs with specific properties appear within $B$. We will describe them using the language of universal twist systems from Subsection~\ref{ssec:universalTwistSystem}. Their properties are motivated by the next step.
\item The second step is similar to the strategy outlined in Proposition~\ref{prop:nonParametric}: The presence of the overtwisted discs allows us to modify the angular functions $(c_k)_{k \in K_0}$ of $B$ close to its bottom boundary. Some further manipulations allow us apply Lemma~\ref{lem:Bolzano} and conclude the proof. Since this takes place within $B$, it is relative to $V$ and $\Delta$.  
\end{enumerate}
This process is relative to the other shells, so the argument can be iterated. The current Subsection introduces certain definitions and constructions needed for the proof, which we then carry out in Subsection~\ref{ssec:extension}.

\subsubsection{Overtwisted discs associated to shells} \label{ssec:replication} 

We want all subsequent formal Engel homotopies to be relative to the certificate and yet, we do need a certificate to complete the proof. The replication Lemma~\ref{lem:selfReplication2} can be used to obtain new overtwisted discs as follows.

We use Lemma~\ref{lem:selfReplication2} to modify $\SD_{M \times K}$ in the Engel region $\Op(\Delta) \cap (M \times \Op(K_0))$. This modification is an Engel homotopy that introduces an additional family of overtwisted discs 
\[ \Delta_k': \Delta_\OT \lra (M,\SD_k), \qquad  k \in K_0. \]
It is relative in the parameter to the complement of $\Op(K_0) \subset K$, and relative in the domain to $\Delta$ and the complement of $\Op(\Delta)$. 

Any further operations involving overtwisted discs and the shell $B$ will make use of $\Delta' = (\Delta_k')_{k \in K_0}$, so that our constructions do not affect the original certificate $\Delta$.

\subsubsection{Connections between the shell and the overtwisted disc}

The following definition packages the notion of a smooth family of paths connecting $B$ with the certificate $\Delta'$ associated to it:
\begin{definition} \label{def:connection to core}
A submanifold $\nu: [0,1] \times K_0 \lra M \times K$ is a {\bf connection} between $\Delta'$ and $B$ if
\begin{itemize}
\item $\nu(\cdot,k) \subset M_k = M$ is an embedded path transverse to $\SW_k$ and tangent to $\SE_k$,
\item $\nu(\cdot,k)$ is disjoint from $V$, $\Delta$, and all shells except $B$,
\item $\nu(\cdot,k)$ intersects $B_k$ only at the endpoint $\nu(1,k)$. In the coordinates of $B_k$ this corresponds to $(0,0,-\varepsilon,0)$.
\item $\nu(\cdot,k)$ is disjoint from the scaling region of $\Delta_k'$. It intersects $\Delta_k'$ only at the endpoint $\nu(0,k)$. In the coordinates of $\Delta_k'$ this corresponds to $(L/2,0,0,0)$.
\end{itemize}
\end{definition}
In particular, a connection is contained in the region in which $\SD_{M \times K}$ is Engel. A key fact, which follows from the first two items in the definition and Proposition \ref{prop:tangentCurvesModel}, is that the Engel structure is unique up to diffeomorphism on a neighborhood of $\nu(\cdot,k)$.

We will construct a connection between $B$ and $\Delta'$ in Proposition \ref{prop:connections exist}. For this we will use the following lemma (which rephrases the parametric ambient connected sum lemma in \cite[Lemma 9.1]{BEM}).

\begin{lemma} \label{lem:BEM}
Let $P_0$ be a contractible set. Let $\{P_i \subset P_0\}_{i=1,\ldots,m}$ be a finite collection of subsets such that
\begin{equation} \label{e:contractible condition}
\bigcap_{j\in J} P_{j}\textrm{ is contractible or empty for all }J\subset\{1,\ldots,m\}.
\end{equation}
Let $M$ be a manifold and fix a collection of families of embeddings of the disc
\begin{align*}
\mathcal{B}^i: \D^{\dim(M)} \times P_i \lra M \times P_0 \\
\mathcal{B}^i(p,k) = \left(\mathcal{B}_k^i(p),k\right). 
\end{align*}

Then, there is a collection $\{D_i\}_{i=0,\ldots,m}$ of disjoint embeddings of $\D^{\dim(M)}$ into $M$, and a $P_0$-family of isotopies $(\psi_k: M \lra M)_{k \in P_0}$ such that
$$ \psi_k \circ \mathcal{B}_k^i(p) = D_i(p) \quad\textrm { for all $i$, $k$, and $p$}. $$
\end{lemma}

\begin{proof}
We argue by induction on $m$. The inductive hypothesis is that the lemma holds for all collections of size at most $m-1$ satisfying the contractibility hypothesis and for all manifolds (not just $M$).

In the base case $m=1$ we reason as follows: Since $P_0$ is contractible, the isotopy extension theorem ensures the existence of a $P_0$-family of isotopies $\psi_k^0: M \lra M$ satisfying $\psi_k^0 \circ \mathcal{B}_k^0(p) = D_0(p)$, for some embedding of the disc $D_0$. Then, we regard $\mathcal{B}^1$ as a family of embeddings into $M \setminus D_0$, and we use the same argument to find a $P_1$-family of isotopies $\psi_k^1: (M,D_0) \lra (M,D_0)$ (i.e. relative to $D_0$) satisfying $\psi_k^1 \circ \mathcal{B}_k^1(p) = D_1(p)$.

We now explain the inductive step. First, we may assume that $\mathcal{B}_k^0 = D_0$ for all $k \in P_0$ using the same reasoning as above. Then, the sets $P_1$, $\{P_1 \cap P_i\}_{i>1}$, the families $\{(\mathcal{B}_k^i)_{P_1 \cap P_i}\}_{i=1,\ldots,m}$, and the manifold $M \setminus D_0$ satisfy the inductive hypothesis. 

Therefore, there is a family of isotopies $(\psi_k': (M,D_0) \lra (M,D_0))_{k \in P_1}$ with $\psi_k' \circ \mathcal{B}_k^i = D_i$ for $k\in P_1$ and all $i$. Now we choose a neighborhood $\SU$ of $D_0$ such that $\mathcal{B}_k^i \cap \SU = \emptyset$ for all $k$ and all $i > 0$. It is then possible to modify the isotopies $\psi_k'$ so that
\begin{align*}
 \psi_k' \circ \mathcal{B}_k^1 & = D_1 \subset \SU \\
 (\psi_k' \circ \mathcal{B}_k^i) \cap \SU & = \emptyset, \qquad\textrm{ for } i > 1. 
\end{align*}
Then we conclude by noting that the inductive hypothesis applies once again to $P_0$, $\{P_i\}_{i>1}$, the families $\{(\mathcal{B}_k^i)_{k\in P_i}\}_{i=2,\ldots,m}$, and the manifold $M \setminus \SU$. 
\end{proof}

For $i>0$ the embeddings $\mathcal{B}^i$ in the lemma correspond to the shell $B$, the restrictions $(B_k^i)_{k \in K_0 \cap K_i}$ of the shells $B_i$, and the copy of the overtwisted disc $(\Delta_k')_{k \in K_0}$. Note that assumption \eqref{e:contractible condition} follows from condition (iv) in Proposition \ref{prop:reduction}. 
 
\begin{proposition} \label{prop:connections exist}
There is a connection $\nu: [0,1] \times K_0 \lra M \times K_0$ between the certificate $\Delta'$ and the shell $B$. 
\end{proposition}

\begin{proof}
Let $B_k^i$ with $k \in K_i, i=1,\ldots,m$, be the collection of shells other than $B$. Recall that $V$ is of the form $(U \times K)\cup (M \times K')$. In particular, $K_0$ and $K'$ are disjoint. Using Lemma \ref{lem:BEM} we find a $K_0$-family $\psi_k: M \lra M$ of isotopies of $M$ with support in the complement of $U$ such that 
\begin{align*}
D & = \psi_k\circ B_k   &   D_\Delta & = \psi_k \circ \Delta_k \\ D_{\Delta'} & = \psi_k\circ\Delta_k' & D_i &  = \psi_k\circ B_k^i
\end{align*}
are embeddings of balls in $M \setminus U$ that do not depend on $k \in K_0$. 

In particular, after applying the isotopy $\psi_k$, we may assume that
\begin{itemize}
\item the region $\{z \leq 0, w \leq \varepsilon\} \subset B_k$, as a \emph{parametrized} subset of $M$,  and
\item the overtwisted disc $\Delta_k'$, as a \emph{parametrized} subset of $M$, 
\end{itemize}
are independent of $k\in K_0$. 
Since $M \setminus (U \cup D_{\Delta} \cup D_1 \cup \ldots \cup D_m)$ is connected manifold, there is an embedded path $\nu: [0,1] \lra M$ such that
\begin{itemize}
\item $\nu$ is disjoint from the balls $U \cup D_{\Delta} \cup D_1 \cup \ldots \cup D_m$,
\item $\nu(0)$ is mapped to $\Delta_k'(L/2,0,0,0) = D_{\Delta'}(L/2,0,0,0)$,
\item $\nu(1)$ is mapped to $B_k(0,-\varepsilon,0,0) = D(0,-\varepsilon,0,0)$,
\item $\nu$ is transverse to $\SW_k$ and tangent to $\SE_k$ at times $t \in \Op(\{0,1\})$.
\end{itemize}
The last property follows from the explicit models we have around the endpoints, allowing us to prefix the path $\nu$ there. Embeddedness follows from $C^\infty$-genericity, since the dimension of $M$ is $4$. 

We can regard $\nu$ as a submanifold $\nu: [0,1] \times K_0 \lra M \times K$ that does not actually depend on the parameter $k \in K_0$. Do note, however, that each $\nu_k$ maps into a \emph{different} even-contact manifold $(M,\SE_k)$. Still, since $K_0$ is a ball, we can apply Proposition \ref{prop:tangentCurvesExistence} to obtain a $C^0$-deformation of the family $\nu = (\nu_k)_{k \in K_0}$ relative to the endpoints. In this manner, we make $\nu$ tangent to  $\SE_{M \times K}$, while preserving embeddedness and transversality with $\SW_{M \times K}$. Therefore, the resulting manifold is a connection.
\end{proof}
We henceforth fix a connection $\nu$ between $B$ and $\Delta'$. Schematically, this is shown in Figure~\ref{fig:connection to shell}.
\begin{figure}
\begin{center}
\includegraphics[scale=0.9]{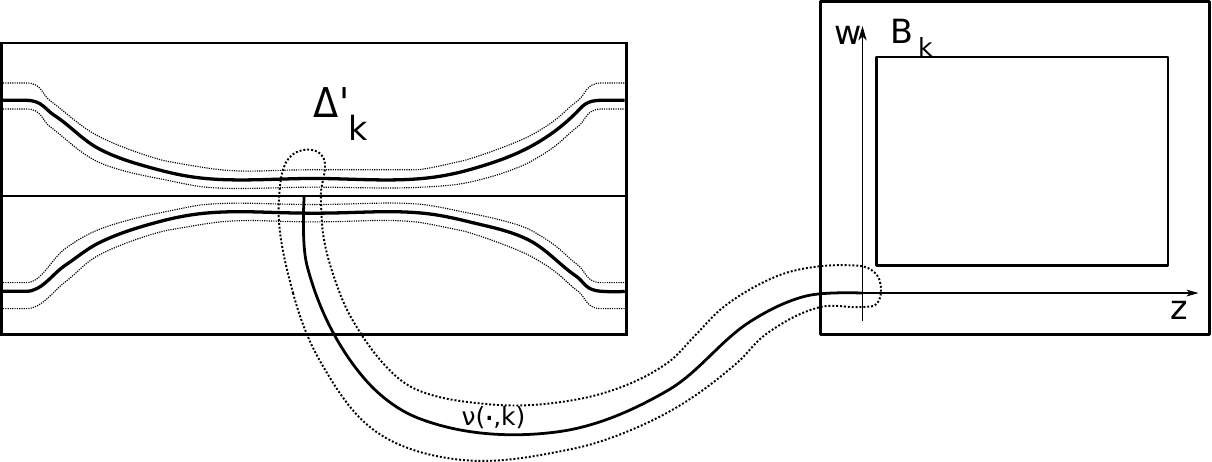} \label{fig:connection to shell}
\caption{A connection $\nu(\cdot,k)$ from an overtwisted disc $\Delta_k'$ to $B_k$.}
\end{center}
\end{figure}

\subsubsection{Twist systems} \label{sssec:twistSystem}

We want to place a portion of the universal twist system inside of the shell $B$. Let us recall some notation from Section~\ref{ssec:universalTwistSystem}, p.~\pageref{ssec:universalTwistSystem}. 

The universal twist system is a $1$-parametric family of  surfaces $(\SS_t)_{t \in [0,1]}$ contained in $I \times \R^3_-$, we consider the case when $I$ is a closed interval. Each of the surfaces $\SS_t$ consists of infinitely many cylinders stacked side to side along the plane $\{z=0\}$:
$$ \SS_t = \psi_\lambda\left(\bigcup_{n\in\Z}T^n(S_t)\right) = \bigcup_{n\in\Z}T^n_\lambda(\psi_\lambda(S_t)), $$
where $S_t$ is a single cylinder, $T_\lambda$ is the translation in $x$ of length $\lambda^2$, and $T = T_1$. The scaling constant $\lambda$ in the definition depends on two parameters (see Remark \ref{rem:twistSystemConstants}):
\begin{itemize}
\item $\varepsilon_0$, the size of $S_t$ in the $x$, $z$, and $w$ coordinates, and
\item $\delta_0$, the length of the interval at the ends of the cylinder in which the unlinking and shrinking take place.
\end{itemize}
A third parameter, $\lambda_0 \in (0,1]$, controls the shrinking of $S_t$ at its ends. It does not affect $\lambda$. In particular, once we fix $\lambda$ we are allowed to further shrink the ends of the universal twist system (while remaining transverse to the Engel structure and embedded). We single out the region  
$$ S_t^+ = S_t \cap \{z=0\} \subset S_t. $$
By construction, the union
$$ \bigcup_{n\in\Z} \psi_\lambda\left(T^n(S_t^+)\right)  $$
contains the infinite band $\{y\in [\delta_0,1-\delta_0],  z=w = 0 \}$. 

Observe that the bottom part of any shell, i.e. a sufficiently small collar of the bottom boundary $\{w=a\}$,  is itself a solid shell of the form 
$$ M([0,1],[-\varepsilon,\varepsilon],-\varepsilon,f_+,\textrm{arcsin}(w)). $$
By construction, the Engel structure in this region is the one we used to define the universal twist system:
\[ \SD_\trans = \ker(\alpha_\trans=dy-zdx)\cap\ker(\beta_\trans=dx-wdz). \]
\begin{definition} \label{def:twist system}
Let $\SS_t$ be a universal twist system with scaling constant $\lambda>0$. Consider a shell $M([0,1],[-\varepsilon,1],-\varepsilon,f_+,c)$ and a pair of integers $m_-,m_+ \in \Z$ satisfying
\begin{enumerate}
\item[(a)] $\psi_\lambda\left(\bigcup_{m_-\leq n\leq m_+} T^n(S_t)\right)$ is contained in $[0,1] \times (0,1) \times (-\varepsilon,0] \times (-\varepsilon,\varepsilon)$,
\item[(b)] $\psi_\lambda\left(\bigcup_{m_-\leq n\leq m_+} T^n(S_1^+)\right)$ contains the square $[\varepsilon,1-\varepsilon]^2 \times \{(0,0)\}$.
\end{enumerate}
Then, the collection of surfaces described in item (a) can be regarded as a subset of the shell and is called a \textbf{twist system}.
\end{definition}

\begin{proposition} \label{prop:twistSystem}
Any shell $M([0,1],[-\varepsilon,1],-\varepsilon,f_+,c)$ admits a twist system. This holds parametrically for compact families of shells.
\end{proposition}

\begin{proof}
We choose constants $\varepsilon_0$ and $\delta_0$ in the construction of the universal twist system to be much smaller than $\varepsilon$. Then there are integers $m_-$ and $m_+$ that satisfy the claim. The parametric statement follows by taking a sufficiently small constant $\varepsilon$ suitable for all the models in the family.
\end{proof}

We apply the Lemma to the shell $B = (B_k)_{k \in K_0}$. To each model $B_k$ we assign a twist system $\SS_{k,t}$. This surface does not actually depend on $k$ when we use the coordinates provided by $B_k$. Similarly, we write $\SS_{k,t}^+ \subset \SS_{k,t}$ for the corresponding collection of surfaces as defined in item (b) of Definition~\ref{def:twist system}. The cores of the cylinders in $\SS_{k,t}$ will be denoted by $\{\alpha_k^n\}_{m_-\leq n\leq m_+}$.

\begin{remark}
Twist systems will play a role analogous to the one of the $2$-torus $\SS$ in the proof of Proposition~\ref{prop:nonParametric}. Let us briefly compare the two.

The profile of the torus $\SS$ is a transverse unknot $\eta$ with many stabilizations. The precise number of stabilizations needed depends on how large the Engel region of the circular shell is and, as such, we do not have an a priori bound. However, we do not need such a bound for proving statements that are not relative in the parameter space. 

However, for a complete $h$-principle, one has to define what the overtwisted disc is. In our approach this means that one has to fix the profile of a transverse surface. We choose $\eta_{-3}$, the unknot with a single stabilization. This implies that we cannot introduce in our shells a single, wide surface like $\SS$, but rather a thin cylinder coming from the overtwisted disc. In order to be  able to use these cylinders as in the proof of Proposition~\ref{prop:nonParametric} we stack many of them next to each other. This is exactly what a twist system is.

It is worth noting that   we could have chosen as profile a knot with more stabilizations as well. However, we do not know whether an unstabilized unknot can be used.
\end{remark}

\subsubsection{Homotopies of cores after undoing the Engel-Lutz twist} \label{ssec:homotopies of cores}

In this section we show that the twist systems we have introduced can be obtained by homotoping the core of an overtwisted disc and transporting the transverse surface and the Engel-Lutz twist along. First, we describe how to move the core itself. We make extensive use of the notation introduced in Subsection~\ref{ssec:homotopyOTdisc}, p. \pageref{ssec:homotopyOTdisc}, which we briefly review. %
We continue using the notation from the previous subsection as well.

Given the shell $B = (B_k)_{k \in K_0}$, we produce a copy $\Delta' = (\Delta_k')_{k \in K_0}$ of the certificate using Lemma~\ref{lem:selfReplication2}.  The core of $\Delta'$ is denoted by $\gamma: [0,L] \times K_0 \lra M \times K$. Additionally, we fix a connection $\nu: [0,1] \times K_0 \lra M\times K$ between $B$ and $\Delta'$. Let $K_0' \subset K_0$ be a subset such that $B_k$ is a solid shell for $k\in K_0 \setminus K_0'$.

As in Section~\ref{ssec:homotopyOTdisc} we write $(M\times K_0)^{\Delta'} = M^{\Delta'} \times K_0 = (M_k^{\Delta'})_{k \in K_0}$ for the manifold obtained from $M \times K_0$ by cutting along the upper $\{y=L\}$ and lower $\{y=0\}$ boundaries of the certificate $\Delta'$. 

The restriction of the formal Engel structure $\SD_{M \times K}$ to the manifold $(M\times K_0)^{\Delta'}$ is obtained from some other structure $\SL^{-1}(\SD_{M \times K})$ by an Engel-Lutz twist along a $K_0$-family of cylinders $\Sigma_k$ with core $\gamma$. All elements we have been working with (shells and connections) are disjoint from the upper and lower boundary of $\Delta'$, so we can regard them as lying in $(M\times K_0)^{\Delta'}$. In particular, the shell $B$ inherits the same formal Engel structure from $\SL^{-1}(\SD_{M \times K})$.

We write $\SL^{-1}(\SD_k)$ for the restriction of $\SL^{-1}(\SD_{M \times K})$ to $M_k$ and $\SL^{-1}(\SE_k)$ for the induced even contact structure. In  $\image(\Delta')$ the structure $\SL^{-1}(\SD_k)$ is given by an Engel embedding
$$ \varphi: ([0,L] \times \D^3,\SD_\trans) \lra \left(M^{\Delta'}_k,\SL^{-1}(\SD_k)\right) $$
of the standard neighborhood of the core $\gamma(\cdot,k)$. We also recall that the length of the scaling region $l$ satisfies $L/2 > l > 1/\tau_0$ (where $\tau_0$ is the constant from Lemma \ref{lem:scalingProfile}) and that the cores of the surfaces forming a twist system $(\SS_{k,t})$ are independent of $t$. 

\begin{proposition} \label{prop:homotopiesCoreOTDisc}
In this setting, there is a path of families of embedded transverse curves
$$ \gamma_s: [0,L] \times K_0 \lra \left((M \times K_0)^{\Delta'}, \SL^{-1}(\SD_{M \times K})\right) $$
with $s\in [0,1]$ such that 
\begin{itemize}
\item $\gamma_0$ is the core $\gamma$ of $\Delta'$,
\item $\gamma_s$ is disjoint from $V$, $\Delta$, and all shells other than $B$,
\item for $k\in K_0'$, $\gamma_1$ contains the cores $\{\alpha_k^n\}, m_-\le n\le m_+$, of the twist systems $(\SS_{k,t})$, and 
\item $\gamma_s=\gamma_0$ outside of $\Op(\{L/2\}) \times K_0$. In particular, they agree in the scaling region.
\end{itemize}
\end{proposition}

\begin{proof}
Consider the family of open manifolds
$$C_k = \Op(\gamma \cup \nu(\cdot,k) \cup B_k(\{|w| < \varepsilon\}) $$
with $k\in K_0$. By construction, $\left(C_k,\SL^{-1}(\SD_k)\right)_{k \in K_0}$ is a family of genuine Engel manifolds diffeomorphic to a  ball. Moreover, the even-contact structure $\SL^{-1}(\SE_k)$ on $C_k$ does not depend on $k\in K_0$: First, note that we have explicit models in the vicinity of $\gamma$ (according to Proposition~\ref{prop:modelKnot}), in $B_k(\{|w| < \varepsilon\})$ (provided by the model structure), and in a neighborhood of $\nu(\cdot,k)$ (Proposition \ref{prop:tangentCurvesModel}). These three models glue, so we can identify all the manifolds $(C_k,\SL^{-1}(\SE_k))_{k \in K_0}$ with a fixed even-contact manifold $(C \cong \D^4,\SE_C)$.

Under this identification, the curves $\gamma(\cdot,k)$ are all identified with the same curve $\tilde{\gamma}_0$ in $(C,\SE_C)$. Similarly, for a given $n$, the cores $\alpha_k^n$ are all identified with a single curve $\widetilde{\alpha}_n$. By  the existence part of the $h$-principle for transverse knots  Lemma~\ref{lem:hPrincipleTransverseCurves} (on p.~ \pageref{ssec:transverseKnots}) there is a transverse embedded curve $\tilde\gamma_1: [0,1] \lra (C,\SE_C)$ such that
\begin{itemize}
\item $\widetilde{\gamma}_1(y) = \widetilde{\gamma}_0(y)$ for $y$ outside of $\Op(\{L/2\})$, and
\item $\widetilde{\gamma}_1 \supset \widetilde{\alpha}_n$ for all $n$.
\end{itemize}
According to the classification part of Lemma~\ref{lem:hPrincipleTransverseCurves} there is a homotopy $(\widetilde{\gamma}_s)_{s \in [0,1]}$, as embedded curves transverse to $\SE_C$, between $\widetilde{\gamma}_0$ and $\widetilde{\gamma}_1$.

Now we reintroduce $k$ in the discussion. Fix a bump function $\chi: K_0 \lra [0,1]$ such that
\begin{itemize}
\item $\chi(k) \equiv 0$ in the complement of $\Op(K_0')$, and
\item $\chi(k) \equiv 1$ in $K_0'$.
\end{itemize}
Using the identification between $\gamma(\cdot,k) \subset C_k \subset M_k^{\Delta'}$ and $\widetilde{\gamma}_0 \subset C$, we define the desired homotopy $\gamma_s$ of $\gamma$ 
\begin{align*}
\gamma_s: [0,1] \times K_0 & \lra (M \times K_0)^{\Delta'} \\ 
(y,k) & \longmapsto \gamma_s(y,k)=\widetilde{\gamma}_{\chi(k)s}(y). 
\end{align*}
This homotopy takes place, in the domain, in the region $\Op(\{1/2\}) \times \Op(K_0')$. Similarly, in the target space, it takes place within $(C_k)_{k \in K_0}$. 
\end{proof}

Now that we can move the core effectively, we explain how to move the overtwisted disc along the homotopy $(\gamma_s)_{s \in [0,1]}$. 
\begin{proposition} \label{prop:homotopiesOTDisc}
Fix a constant $t_0$ arbitrarily close to but smaller than $1$. There is a homotopy of formal Engel structures $(\SD_s)_{s \in [0,1]}$ in $M \times K$ satisfying 
\begin{itemize}
\item $\SD_0 = \SD_{M \times K}$,
\item $\SD_1$ has an Engel-Lutz twist along the twist systems $(\SS_{k,t_0})$ for $k \in K_0'$, and
\item $\SD_s$ differs from $\SD_0$ only in a neighborhood of the overtwisted disc $\Delta'$, the connection $\nu$, and the region $\{|w| < \varepsilon\} \subset B$.
\end{itemize}
\end{proposition}

\begin{proof}
We apply Proposition \ref{prop:homotopyOTdisc} to $\SD_{M \times K}$ and $(\gamma_s)_{s \in [0,1]}$, parametrically in $k$. This provides a homotopy of formal Engel structures $(\widetilde\SD_s)_{s \in [0,1]}$ such that
\begin{itemize}
\item $\widetilde\SD_0 = \SD_{M \times K}$,
\item $\widetilde\SD_s$ differs from $\SD_0$ only in a neighborhood of $\Delta'$, $\nu$, and $\{|w| < \varepsilon\} \subset B$, and
\item $\widetilde\SD_s$ is obtained from $\SL^{-1}(\SD_{M \times K})$ by an Engel-Lutz twist along a family of cylinders $\Sigma_s$ with cores $\gamma_s$. In the regions where $\gamma_1(\cdot,k)$ agrees with the cores of $\SS_{k,t_0}$, the profiles describing $\Sigma_1$ can be assumed to be a  rescaling of the profile of the twist system.
\end{itemize}
Recall now that the size of the ends of a twist system were controlled by a constant $\lambda_0$. In particular, we choose them to be arbitrarily small. This implies that we can enlarge the profiles of each $(\Sigma_s)_{s \in \Op(\{1\})}$ along the cores of the twist system until $\Sigma_1$ contains 
$$
\bigcup_{k \in K_0'} \SS_{k,t_0}.
$$
The resulting cylinders are still transverse. Adding an Engel-Lutz twist to $\SL^{-1}(\SD_{M \times K})$ along this path of families of cylinders yields the result.
\end{proof}

\subsection{Extension} \label{ssec:extension}

To conclude the proof of Theorem~\ref{thm:mainPrime}, we need one more ingredient. The following proposition contains the main geometric ideas in this paper, most of which are essentially a refinement of the method shown in Subsection \ref{ssec:nonParametric}. It will become apparent during the proof that the properties of a twist system (Subsection \ref{sssec:twistSystem}) are precisely what is needed for the argument to go through.
\begin{proposition} \label{prop:extension}
Fix a shell
$$
(B,\SD) = M([0,1],[-\varepsilon,1],-\varepsilon,f_+,c)
$$
and a twist system
$$
\SS = \bigcup_{m_-\leq n\leq m_+} \psi_\lambda(T^n(S_{t_0})) \subset \{z \leq 0; |w| < \varepsilon\}
$$
with $t_0$ sufficiently close to but smaller than $1$. Write $\SL(\SD)$ for the formal Engel structure obtained from $\SD$ by performing an Engel-Lutz twist along the surfaces $\SS$.

Then $\SL(\SD)$ is homotopic through formal Engel structures to an honest Engel structure $\SD'$ satisfying:
\begin{itemize}
\item the homotopy is relative to the boundary $\partial B$,
\item in $\{|w| > \varepsilon\}$, the even-contact structure remains fixed and the homotopy only modifies the angular function.
\end{itemize}
\end{proposition}

Before giving the proof, we want to emphasize that when we speak of introducing an Engel-Lutz twist in the shell $B$ along the twist system, we do not mean that a {\em new} Engel-Lutz twist is introduced. This Engel-Lutz twist arises via a homotopy of Engel structures as explained in Subsection~\ref{ssec:homotopies of cores}, so we are effectively comparing the formal Engel structure before ($\SD$) and after ($\SL(\SD)$) the homotopy of the overtwisted disc. In particular, note that $B$ is a shell with respect to the original formal Engel structure, but not with respect to the formal Engel structure obtained after the homotopy.

\begin{proof}
Let $S_n = \psi_\lambda\left(T^n(S_{t_0})\right), n=m_-,\ldots,m_+$,  be  the individual cylinders in the twist system, and 
\begin{align*}
S_n^+ & = S_n \cap \{y \in [\varepsilon,1-\varepsilon], z=0\}\\
 \SS^+ & = \bigcup_{m_-\leq n\leq m_+} S_n^+.  
\end{align*}
Let $\delta>0$ be a small constant which will be determined later. 

Because of property (5) of the profile of the universal twist system (Section~\ref{ssec:universalTwistSystem} p. \pageref{isotopy prop 1}), there is a function $\omega: \{z \geq 0; w=0\} \lra [0,\varepsilon)$ such that
\begin{itemize}
\item $\omega(y,x,z) \equiv 0$ whenever $z > \delta$,
\item the points in $S_{m_+}^+$ are of the form $(y,x,0,\omega(y,x,0))$, and
\item $\omega(y,x,0) > w$ for all $j < m_+$ and $(y,x,0,w) \in S_j$.
\end{itemize}
That is, the hyperplane $L(y,x,z) = (y,x,z,\omega(y,x,z))$ lies above the $(S_j)_{j < m_+}$, contains $S_{m_+}^+$, and agrees with $\{z \geq 0; w=0\}$ in the complement of $\Op(\{z,w=0\})$. Let us single out the following regions (c.f. Figure~\ref{fig:bla}).
\begin{itemize}
\item The portion of $B$ lying above $L$
$$
B' = \{(y,x,z,w) \,|\, z \geq 0; w \geq \omega(y,x,z) \}, 
$$
\item The surface contained in $L$ and lying directly above $S_j^+$
$$
L_j = \bigcup_{(y,x,0,w) \in S_j^+} \{(y,x,0,\omega(y,x,0))\}, 
$$
\item The strip connecting $S_j^+$ with $L_j$
$$
A_j = \bigcup_{(y,x,0,w) \in S_j^+} \{(y,x,0)\} \times [w,\omega(y,x,0)].
 $$
\end{itemize}
\begin{figure}
\begin{center}
\includegraphics[scale=1]{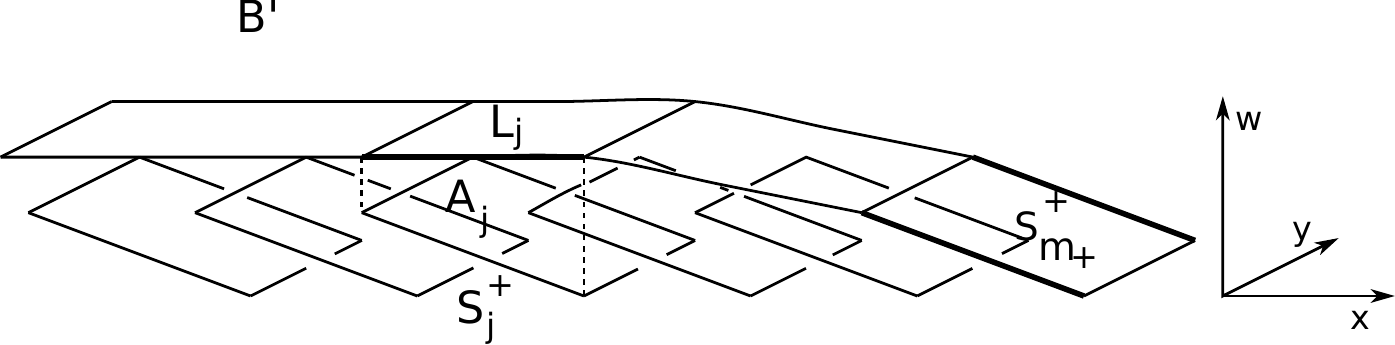}
\caption{Parts of $L,L_j$ and $A_j$. \label{fig:bla}}
\end{center}
\end{figure}

We claim that there is a path of formal Engel structures $(\SD_s)_{s \in [0,1]}$ with the following properties.
\begin{itemize}
\item $\SD_0$ is obtained from $\SD$ by an Engel-Lutz twist  along $(S_j)_{j < m_+}$.
\item $\SD_s = \SD_0$ in the complement of a small neighborhood of $\bigcup_{j < m_+} A_j$.
\item $\SD_s$ is transverse to $S_{m_+}$.
\item $(B',\SD_s)$ is a shell with an angular function $c_s$ such that
\begin{align*}
c_0&(y,x,z,w) - \pi/2 < c_s(y,x,z,w) \leq c_0(y,x,z,w) = c(y,x,z,w)\\
c_1 &\equiv -\pi/2+\delta   \quad\textrm{ on } \bigcup_{j < m_+} L_j.
\end{align*}
\end{itemize}
In particular, note that the leftmost region of $S_{m_+}^+$ (the region with smaller $x$-coordinate) is contained in $\bigcup_{j < m_+} L_j$.

For the construction of $\SD_s$ we will first modify $S_0$ inductively over the regions $\Op(A_j)$ for $j=m_-,\ldots,m_+-1$. Let us start with $j=m_-$: Consider the structure $\SD_{m_-,0}$ obtained from $\SD$ by adding an Engel-Lutz twist along $S_{m_-}$. By construction, $\SD_{m_-,0}$ describes more than one turn along the $\partial_w$-flowlines contained in the strip $A_{m_-}$ (these flowlines are leaves of the kernel). 

Using an isotopy tangent to $\SW$ we can push this turning of the structure upwards. When we push approximately a quarter of a turn in terms of the framing $\{\partial_z, \partial_x + z\partial_y\}$ upwards, we obtain a path of formal Engel structures $\SD_{m_-,s},s \in [0,1]$, satisfying the following conditions.
\begin{itemize}
\item $\SD_{m_-,s} = \SD_{m_-,0}$ in the complement of a small neighborhood of $A_{m_-}$.
\item  On the complement of a small neighborhood of $S_{m_-}$ the formal Engel structure $\SD_{m_-,s}$ is determined by an angular function $c_{m_-,s}$ such that
\begin{align*}
c(y,x&,z,w) - \pi/2 < c_{m_-,s}(y,x,z,w) \leq c(y,x,z,w), \\
c_{m_-,1} & \equiv -\pi/2+\delta   \quad\textrm{ on } L_{m_-}.
\end{align*}
\end{itemize}
We now use property (B) from the definition of the (universal) twist system crucially (c.f. p.~\pageref{condition B}): If the neighborhood of $A_{m_-}$ containing the support of the isotopy is small enough, then $\SD_{m_-,s}$ is transverse\footnote{This can be seen in Figure~\ref{fig:univ-twist-overlay-neu}: If $\eta$ is the profile of the twist system and $T(\eta)$ is its translate, the regions $\eta^+$ and $T(\eta^+)$ overlap over some interval $P$. It is in a neighborhood of this interval (or rather, its counterpart for the twist system) in which the angular function is changing. This corresponds to turning the line field depicted in the figure clockwise until it becomes almost parallel to the $x$-axis. It follows that $T(\eta)$ remains transverse.} to $S_j$ for all $j>m_-$. Since $S_{m_-+1}$ lies below $L_{m_-}$, it follows that $c_{m_-,s} \in (-\pi/2,-\pi/2 +\delta)$ on $S_{m_-+1}$.

This construction can be iterated. For each $n=m_-+1,\ldots,m_+-1$ we consider the family $(\SD_{n-1,s})_{s \in [0,1]}$ constructed in the previous inductive step. Using the fact that $\SD_{n-1,s}$ is transverse to $S_n$ we add an Engel-Lutz along $S_n$, parametrically in $s$. We denote the resulting family by $\SD_{n,s} = \SL\left(\SD_{n-1,2s}\right)$ for $s \in [0,1/2]$. By hypothesis, $\SD_{n,1/2}$ is described  by an angular function $c_{n,1/2}$ that is precisely $-\pi/2+\delta$ in the region $\cup_{j < n} L_j$ (away from a neighborhood of the $(S_j)_{j \leq n}$).  

As above we define a formal Engel homotopy $(\SD_{n,s})_{s \in [1/2,1]}$, of $\SD_{n,1/2}$ using an isotopy along $\partial_w$ in a neighborhood of the strip $A_n$ (i.e. along the kernel foliation). This allows us to modify the angular function so that it is precisely $-\pi/2+\delta$ in the region $\cup_{j \leq n} L_j$. This completes the induction.

As above, it follows that $S_{m_+}$ is transverse to the path of formal Engel structures $(\SD_s)_{s \in [0,1]}$ that we have constructed. Now we thicken $S_{m_+}$ to a family $(N_s)_{s\in [0,1]}$ of transverse $3$-dimensional manifolds. We may assume that
\begin{itemize}
\item $N_s$ lies in an arbitrarily small neighborhood of $S_{m_+}$,
\item $N_s$ fails to be transverse to the kernel of $\SD_s$ along two disjoint surfaces $N_s^{0,\pm}$ which are themselves transverse to $\SD_s$,
\item $N_s^{0,+}$ and $N_s^{0,-}$ are obtained from $S_{m_+}$ by an isotopy tangent to $\SD_s$, and
\item in the vicinity of $S_{m_+}^+$, this flow is along the Legendrian marking $TL \cap \SD_s$.
\end{itemize}
We can now define an isotopy of $3$-manifolds $(N_s)_{s\in [1,2]}$ transverse to $\SD_1$ by pushing $N_1^{0,+}$ further along the line field $TL \cap \SD_1$. By construction, the Legendrian marking $TL \cap \SD_1$ is given by the line field
$$
\cos(-\pi/2+\delta)\frac{\partial}{\partial z} + \sin(-\pi/2+\delta)\left(\frac{\partial}{\partial x}+ z\frac{\partial}{\partial y}\right)
$$
in a neighborhood of $\cup_{j < m_+} L_j$. In general, $TL \cap \SD_1$ is a line field lying in the tangent cone given by turning clockwise from $-(\partial_x + z\partial_y)$ to $\partial_z$.

\begin{figure}[htb]
\begin{center}
\includegraphics[scale=0.6]{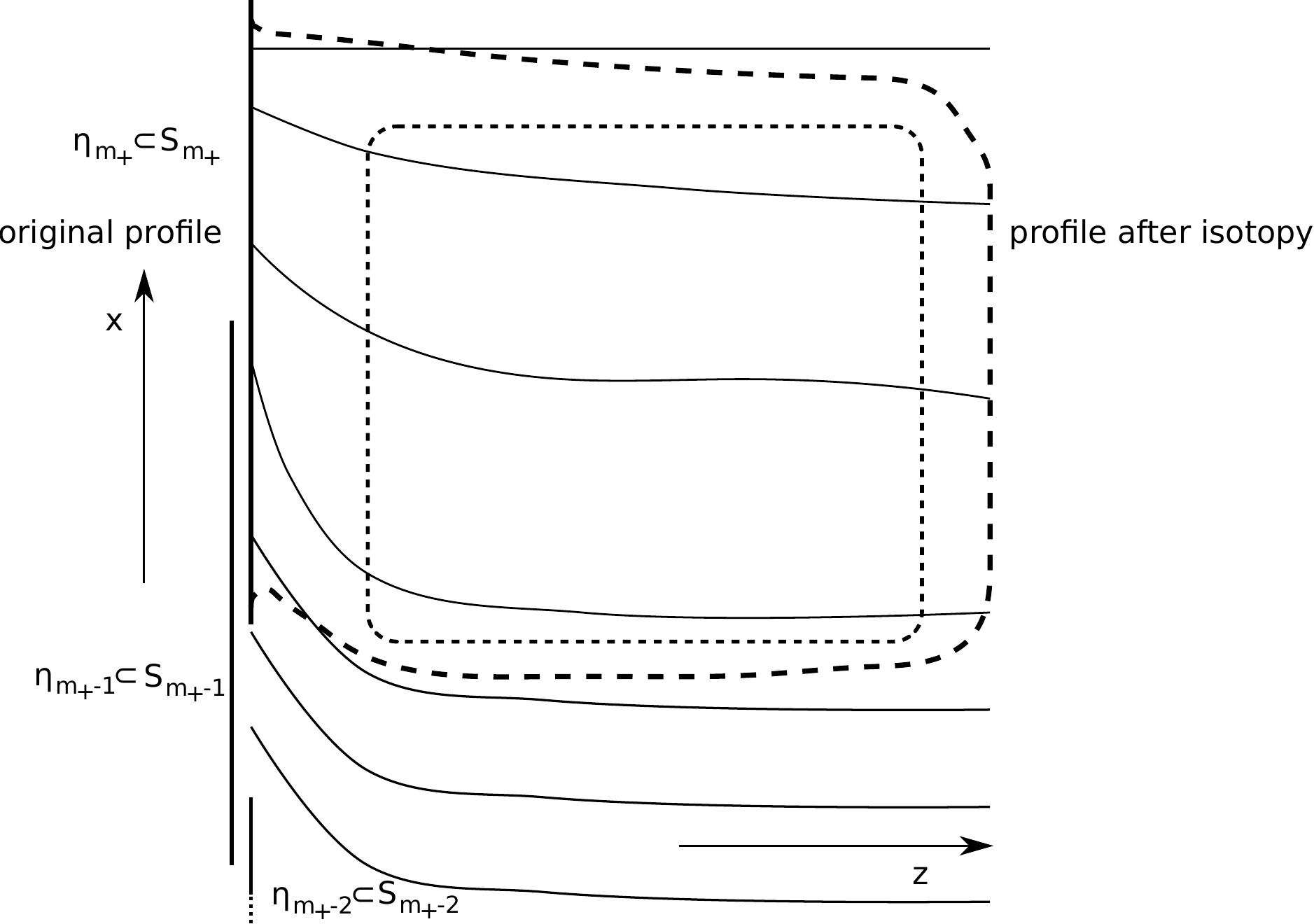}
\caption{The line field $TL\cap \SD_1$.} \label{fig:tilted-vf}
\end{center}
\end{figure}

If we flow $N_1$ along $TL \cap \SD_1$ for sufficiently long times, and $\delta$ was chosen to be sufficiently small, we will obtain a transverse $3$-manifold $N_2$ which contains  
$$ L \cap \{y,x \in [\varepsilon,1-\varepsilon]; z \in [\varepsilon,f_+-\varepsilon]\}. $$
In particular, note that the region where the Engel condition fails lies \emph{above} this hypersurface. The continuous lines in Figure~\ref{fig:tilted-vf} represent $TL \cap \SD_1$. The area contained in the pointed square lies below the region in which $\SD_1$ is not Engel. The thick line on the left corresponds to $S_{m_+}$. The dashed line depicts $N_2^{0,+}$, which is obtained from $S_{m_+}$ by pushing along the Legendrian marking. The manifold $N_2$ covers the region between the solid and the dashed curves.

Now we can finish the proof as in the Subsection \ref{ssec:nonParametric}. First, we construct a path of formal Engel structures $(\SL_s(\SD))_{s \in [0,1]}$ by adding Engel torsion to each $\SD_s$ along the corresponding $N_s$. This corresponds to adding an Engel-Lutz twist along $S_{m_+}$. Therefore, $\SL_0(\SD)$ is precisely $\SL(\SD)$, as desired. We define $(\SL_s(\SD))_{s \in [1,2]}$ to be the structures obtained by adding Engel torsion to $\SD_1$ along $(N_s)_{s \in [1,2]}$. In particular, we obtain a family for formal Engel structure  $\SL_2(\SD)$ that has Engel torsion below the region where the formal Engel structure is not necessarily induced by an Engel structure. As above, this implies that we can push this additional turning upwards, thereby modifying the angular function $c_2$ of $\SL_2(\SD)$ over the region $B'$. 

By definition,  $c_2(y,x,z,1) = c(y,x,z,1) > -\pi$ and an application of Lemma \ref{lem:Bolzano} produces a homotopy $(\SL_s(\SD))_{s \in [2,3]}$ of $\SL_2(\SD)$ with $\SL_3(\SD)$ Engel. This homotopy preserves the even-contact structure, since it simply changes the angular function.
\end{proof}

The parametric version of Proposition~\ref{prop:extension} is stated in the next corollary.
\begin{corollary} \label{cor:parametricExtension}
Fix a shell
$$ (B_k,\SD_k) = M([0,1],[-\varepsilon,1],-\varepsilon,f_{k,+},c_k) \qquad k \in K_0 $$
and a twist system
$$ \SS_k = \bigcup_{m_-\leq n\leq m_+} \psi_\lambda(T^n(S_{t_0})) \subset \{z \leq 0; |w| < \varepsilon\}  $$
where $t_0$ is sufficiently close to but smaller than $1$. Let $\SL(\SD_k)$ be the formal Engel structure obtained from $\SD_k$ by performing an Engel-Lutz twist along the surfaces $\SS_k$.

Then $\SL(\SD_k)$ is homotopic through formal Engel structures to an honest Engel structure $\SD_k'$ such that
\begin{itemize}
\item the homotopy is relative to a $\varepsilon$-neighborhood of $\partial (\cup_K B_k)$, and
\item  the even-contact structure remains fixed on $\{|w| > \varepsilon\}$. There, the homotopy only affects the angular function.
\end{itemize}
\end{corollary}
\begin{proof}
The argument is identical to the one in the proof of Proposition~\ref{prop:extension}, as one must simply add parameters and suitable cut--off functions. We leave the details to the reader, but we will briefly sketch the main points.

We fix a family of hypersurfaces $(L_k)_{k \in K_0}$ lying slightly above $\{w=0\}$ and extending the last component $\psi_\lambda(T^{m_+}(S_{t_0}))$ of the twist system (which, in the paremetrisation of the shell, does not depend on $k$). The proof has two parts. In the first one, the angular function is modified close to the hyperplane $\{z=0\}$ so that it agrees with $-\pi/2 + \delta$ in the region of $L_k$ lying above the other components of the twist system. This has to be done relative to $\partial K_0$ in the parameter. Since the argument boils down to isotoping the formal Engel structure using an isotopy of the angular function, this process can be capped off close to the boundary $\partial K_0$.

In second part of the argument we use flows along $TL_k \cap \SD_k$ to thicken the surface $\psi_\lambda(T^{m_+}(S_{t_0}))$ to a transverse $3$-manifold covering most of the hypersurface $L_k$. Again, this flow can be capped off as $k$ approaches $\partial K_0$, so that the resulting $3$-manifold is instead the boundary of a thin tubular neighborhood of the surface. In this manner, for $k \in \Op(\partial K_0)$, the resulting Engel structure is simply $\SL(\SD_k)$.

Lastly, we apply Lemma \ref{lem:Bolzano}, parametrically on $k$, to modify the angular function in the region lying above $L_k$. For $k \in \Op(\partial K_0)$, the angular function is already increasing and, hence, the relative nature of the lemma yields the claim.
\end{proof}

We can finally prove the main theorem.
\begin{proof}[Proof of Theorem~\ref{thm:mainPrime} (and Theorem~\ref{thm:main})]
After applying the reduction from Proposition~\ref{prop:reduction}, we obtain a family of formal Engel structures $\SD_{M \times K}$ that is Engel in the complement of a collection of shells. Given one such shell $(B = B_k)_{k \in K_0}$ we associate to it
\begin{itemize}
\item a copy $\Delta'$ of the certificate (Subsection \ref{ssec:replication}),
\item a connection $\nu$ between $B$ and $\Delta'$ (Proposition~\ref{prop:connections exist}), and
\item a $K_0$-family of twist systems $(\SS_k)_{k \in K_0}$ (Proposition~\ref{prop:twistSystem}).
\end{itemize}
Write $K_0'\subset K_0$ for a ball such that the formal Engel structure in $B_k$ is Engel for $k\in K_0\setminus K_0'$.

Using Proposition~\ref{prop:homotopiesCoreOTDisc} we construct a homotopy $(\gamma_s)_{s \in [0,1]}$ of the core $\gamma_0$ of the overtwisted disc $\Delta'$ so that $\gamma_1$ contains the cores of the twist systems $(\SS_k)_{k \in K_0'}$. Proposition~\ref{prop:homotopiesOTDisc} implies that there is a homotopy of formal Engel structures $(\SD_{M \times K, s})_{s \in [0,1]}$, starting from $\SD_{M \times K}$, where the overtwisted disc moves along $(\gamma_s)_{s \in [0,1]}$. The final structure $\SD_{M \times K, 1}$ is still only a formal Engel structure, but it has an Engel-Lutz twist along $\SS_k$ for all $k \in K_0'$.

On $B_k$, $k \in K_0'$, the structure $\SD_{M \times K, 1}$ is obtained from $\SD_{M \times K, 0}$ by adding an Engel-Lutz twist along $\SS_k$, $k \in K_0'$. Then Corollary~\ref{cor:parametricExtension} yields a formal Engel homotopy $(\SD_{M \times K, s})_{s \in [1,2]}$ between $\SD_{M \times K, 1}$ and an Engel structure  $\SD_{M \times K, 2}$. This homotopy is relative to the boundary of the shell $(B_k)_{k \in K_0'}$, so the argument can be iterated for the other shells. 
\end{proof}

\subsection{Proofs of Corollaries of Theorem~\ref{thm:main}} 

In this Subsection we  prove some consequences of Theorem~\ref{thm:main}.

\subsubsection{Proof of Corollary \ref{cor:main}} \label{ssec:corMain}

In Corollary \ref{cor:existence} we already showed that the inclusion $\Engel_\OT(M,\Delta) \lra \FEngel(M,\Delta)$ induces a surjection in homotopy groups. Alternatively, one can apply Theorem~\ref{thm:main} taking $K'$ to be the empty set and $K$ to be a sphere. 

For the general statement we need to show the vanishing of the relative homotopy groups $\pi_k(\FEngel(M,\Delta),\Engel_\OT(M,\Delta))$. This is also an application of Theorem~\ref{thm:main} with $K = \D^k$ and $K' = \partial\D^k$. \hfill $\Box$

\subsubsection{Proof of Corollary \ref{cor:extension}} \label{ssec:extensionProof}

The non-parametric statement was already proven in Subsection \ref{ssec:nonParametric}. The general statement is proven as follows: Set $M = \D^4$ and $U = \partial\D^4$. Take a point $p$ in the interior of $\D^4$ and find a family of intervals $(\gamma_k)_{k \in K}$ with $\gamma_k$ passing through $p$ and transverse to $\SE_0(k) = [\SD_0(k),\SD_0(k)]$; the space of such families is contractible. We can then deform $\SD_0$ in $\Op(\{p\}) \times K$ to yield a formal Engel family $\SD_{1/2}$ such that $(\Op(\{p\}),\SD_{1/2}(k))_{k \in K}$ is a locally trivial fibration with fiber $([0,L] \times \D^3, \SD_\OT)$, where $\gamma_k$ corresponds to $[0,L] \times \{0\}$. This is a certificate $\Delta \subset \D^4 \times K$. An application of Theorem~\ref{thm:main} yields a homotopy $(\SD_s)_{s \in [1/2,1]}$ relative to $U = \partial\D^4$ with $\SD_1$ honestly Engel. \hfill $\Box$

\subsubsection{Proof of Corollaries \ref{cor:pi0hPrinciple} and \ref{cor:pi3hPrinciple}} \label{pi3hPrincipleproof}

The proof of Corollary \ref{cor:pi0hPrinciple} goes as follows: Consider $\SD_0$ and $\SD_1$ overtwisted Engel structures with overtwisted discs $\Delta_0$ and $\Delta_1$, respectively. We first homotope them in $\Op(\Delta_i)$ arguing as in Lemma \ref{lem:selfReplication}: this allows us to assume that both overtwisted discs have the same length $L$. Then, using an isotopy of the manifold we set $\Delta_0 = \Delta_1$. A homotopy between $\SD_0$ and $\SD_1$ is then provided by Corollary \ref{cor:main}.

The proof of Corollary \ref{cor:pi3hPrinciple} is slightly more involved. Consider the $K$-family $\SD_0$ with certificate $\Delta^0 = (\Delta_k^0)_{k \in K}$. We can isotope $\SD_0$, parametrically in $k$, to assume that $\image(\Delta_k^0)$ is an arbitrarily small ball. If we assume that $\dim(K) < 4$, the union 
\[ \bigcup_{k \in K} \image(\Delta_0(k)) \subset M \]
does not cover the whole of $M$; choose a point $p$ disjoint from it. We can argue similarly for $\SD_1$ and assume that the certificate $\Delta^1$ misses the same point $p$.

Since $\SD_0$ and $\SD_1$ are formally homotopic, there are \emph{formal} Engel families $\widetilde\SD_0$ and $\widetilde\SD_1$ satisfying:
\begin{itemize}
\item $\widetilde\SD_0$ and $\widetilde\SD_1$ have a certificate $\Delta$ in $U \times K$ (where $U$ is a small neighborhood of $p$),
\item $\widetilde\SD_i$ is formally homotopic to $\SD_i$ and agrees with it in the complement of $\Op(U)$.
\end{itemize}
An application of Theorem~\ref{thm:main} provides a formal homotopy between $\widetilde\SD_0$ and some $\SD_{1/3}$ which is honestly Engel. This homotopy is relative to the complement of $\Op(U)$ and to $\Delta$ (and uses $\Delta$ as certificate). A second application of Theorem~\ref{thm:main} states that $\SD_0$ and $\SD_{1/3}$ are Engel homotopic (using $\Delta^0$ as certificate). Similarly, we produce a formal homotopy between $\widetilde\SD_1$ and some $\SD_{2/3}$ genuinely Engel, which is itself Engel homotopic to $\SD_1$ (using $\Delta^1$ as certificate) and to $\SD_{1/3}$ (using $\Delta$ as certificate). This concludes the proof. \hfill $\Box$

\subsubsection{Foliated results} \label{ssec:foliated}

A well-known observation due to Gromov says that any complete $h$-principle (that is, relative in the parameter and the domain) automatically yields a foliated $h$-principle. When this $h$-principle requires some extra data to be fixed (the certificate), the foliated analogue requires slightly more work, see \cite{cpp15,BEM}. Still, the methods used in the proof of Theorem~\ref{thm:main} (or  a careful application of the Theorem itself), readily imply the following.

Let $(W^{4+m}, \SF^4)$ be a manifold endowed with a smooth foliation of rank $4$. Let $\SW \subset \SD \subset \SE \subset \SF$ be a complete flag for the foliation and assume that we are additionally given bundle isomorphisms
\begin{align} 
\begin{split}\label{e:canonicalIsosFol}
 \det(\SD) &\cong \SE/\SW, \\
 \det(\SE/\SW)& \cong \SF/\SE.
\end{split}
\end{align}
This data is a \textbf{formal foliated Engel structure}. We want to homotope a formal foliated Engel structure so that $\SD$ is a leafwise Engel structure and $\SW \subset \SD \subset \SE$ is the corresponding leafwise Engel flag. We may suppose that there is a closed subset $V \subset W$ where the Engel condition already holds.

Let $K$ be a compact, possibly disconnected, $m$-manifold. Suppose we are given a (embedded) foliation chart
$$
\Delta: \left(\D^4 \times K, \coprod_{k \in K} \D^4 \times \{k\}\right) \lra (W \setminus V,\SF)
$$
satisfying $\Delta^i(\cdot,k)^* \SD = \SD_\OT$ and, additionally, every leaf of $(W \setminus V,\SF)$ intersects the image of $\Delta$. We say that $\Delta$ is a \textbf{certificate of overtwistedness} for the formal leafwise Engel structure. 

\begin{remark}
The case where $(W,\SF)$ is a trivial fibration is precisely the usual parametric setting for the $h$-principle. In this case, the two definitions of certificate are different but equivalent up to homotopy. Indeed, one may use Theorem \ref{thm:main} to produce a certificate in the usual sense if we are given a certificate in the foliated sense, and vice versa. We leave this to the reader. \hfill $\Box$
\end{remark}

The $h$-principle in the foliated setting reads:
\begin{theorem}
Let $(W,\SF,\SW_0,\SD_0,\SE_0)$ be a formal foliated Engel manifold. Suppose that the formal structure is already Engel over some closed subset $V$ and that there is a certificate of overtwistedness $\Delta \subset W \setminus V$.

Then, there is a homotopy of formal foliated Engel structures $(\SW_s \subset \SD_s \subset \SE_s)_{s \in [0,1]}$ with $\SW_1 \subset \SD_1 \subset \SE_1$ a foliated Engel flag. This homotopy is relative to $\Delta$ and $V$.
\end{theorem}
In particular, this statement recovers Theorem \ref{thm:main}.

\end{document}